\documentclass[11pt,a4paper]{article}

\usepackage[T1]{fontenc}
\usepackage[utf8]{inputenc}
\usepackage{tikz, scalerel}
\usepackage{amsmath}
\usepackage{enumerate}
\usepackage{amssymb}
\usepackage{color}
\usepackage{amsthm}
\usepackage{verbatim}
\usepackage{mathabx}
\usepackage{comment}
\usepackage{hyperref}
\usepackage{mathrsfs}
\usepackage{calligra}
\usepackage[refpage, intoc]{nomencl}
\usepackage[automake, toc, style=super]{glossaries-extra}
\usepackage{stmaryrd}

\usepackage[left=1in, right=1in, top=1in, bottom=1in]{geometry}
\newtheorem{definition}{Definition}[section]
\newtheorem{theorem}[definition]{Theorem}
\newtheorem{corollary}[definition]{Corollary}
\newtheorem{remark}[definition]{Remark}
\newtheorem{proposition}[definition]{Proposition}
\newtheorem{lemma}[definition]{Lemma}
\newtheorem{example}[definition]{Example}
\newtheorem{convention}[definition]{Convention}
\newtheorem*{theorem*}{Theorem}

\newcommand{\RR}{\ensuremath{\mathbb{R}}}
\newcommand{\NN}{\ensuremath{\mathbb{N}}}
\newcommand{\Sw}{\ensuremath{\mathcal{S}}}
\newcommand{\ZZ}{\ensuremath{\mathbb{Z}}}

\newcommand{\eps}{\varepsilon}

\newcommand{\GG}{\ensuremath{\mathcal{G}}}
\newcommand{\GGe}{\GG^\eps}
\newcommand{\wGG}{\ensuremath{\widehat{\mathcal{G}}}}
\newcommand{\wGGe}{\wGG^\eps}
\newcommand{\FF}{\ensuremath{\mathcal{F}}}
\newcommand{\FFr}{\mathcal{F}_{\RR^d}}
\newcommand{\FFg}{\mathcal{F}_{\GG}}
\newcommand{\FFge}{\mathcal{F}_{\GG^{\eps}}}
\newcommand{\dd}{\mathrm{d}}
\newcommand{\supp}{\ensuremath{\mathrm{supp}\,}}
\newcommand{\Dw}{\mathcal{D}}
\newcommand{\EE}{\mathcal{E}}
\newcommand{\CC}{\mathbb{C}}
\newcommand{\CR}{\mathscr{R}}
\newcommand{\para}{\varolessthan}
\newcommand{\lpara}{\varogreaterthan}
\newcommand{\reso}{\varodot}
\newcommand{\naivepara}{\varolessthan}
\newcommand{\naivereso}{\varodot}
\newcommand{\naivempara}{\mathord{\prec\hspace{-6pt}\prec}}
\newcommand{\mpara}{\mathord{\prec\hspace{-6pt}\prec}}

\newcommand{\kkg}{\Psi^{\GG,j}}

\newcommand{\kkge}{\Psi^{\GG^\eps,j}}

\newcommand{\je}{\langle j \rangle_\eps}
\newcommand{\rr}{\boldsymbol{\rho}}
\newcommand{\ww}{\boldsymbol{\omega}}
\newcommand{\mm}{\boldsymbol{\mu}}
\newcommand{\rA}{\mathcal{A}}
\newcommand{\dist}{\mathrm{dist}}
\newcommand{\B}{\mathcal{B}}
\newcommand{\cC}{\mathcal{C}}
\newcommand{\Sww}{\Sw_\omega}
\newcommand{\Eww}{C^\infty_\omega}
\newcommand{\Dww}{\Dw_\omega}
\newcommand{\lsim}{\precsim}
\newcommand{\gsim}{\succsim}
\makeatletter
\newcommand*\bigcdot{\mathpalette\bigcdot@{.5}}
\newcommand*\bigcdot@[2]{\mathbin{\vcenter{\hbox{\scalebox{#2}{$\m@th#1\bullet$}}}}}
\makeatother
\newcommand{\scl}{\bigcdot}
\newcommand{\plc}{\cdot}
\newcommand{\ext}[1]{\left(#1\right)_{\mathrm{ext}}}
\newcommand{\aste}{\ast_{\GG^\eps}}
\newcommand{\rast}{\ast_{\RR^d}}

\newcommand{\wexps}{\omega^{\mathrm{exp}}_\sigma}
\newcommand{\wpol}{\omega^{\mathrm{pol}}}
\newcommand{\w}{ p}
\newcommand{\ii}{\imath}
\newcommand{\diam}{\mathrm{diam}}
\newcommand{\short}{\!\!\!\!}
\newcommand{\xshort}{\short\short\short}
\newcommand{\Ff}{\mathcal{F}}
\newcommand{\blacklozeng}{\mathbin{\scriptstyle \blacklozenge}}

\newcommand{\1}{\mathbf{1}}

\usepackage{scalerel,stackengine}
\stackMath
\newcommand\wcheck[1]{%
\savestack{\tmpbox}{\stretchto{%
  \scaleto{%
    \scalerel*[\widthof{\ensuremath{#1}}]{\kern-.6pt\bigwedge\kern-.6pt}%
    {\rule[-\textheight/2]{1ex}{\textheight}}
  }{\textheight}%
}{0.5ex}}%
\stackon[1pt]{#1}{\scalebox{-1}{\tmpbox}}%
}

\title{Paracontrolled distributions on Bravais lattices and \\ weak universality of the 2d parabolic Anderson model}
\author{Jörg Martin\thanks{Financial support by the DFG via Research Training Group RTG 1845 is gratefully acknowledged.} \\
  Institut f\"ur Mathematik \\
  Humboldt--Universit\"at zu Berlin \\
  \texttt{martin@math.hu-berlin.de}  
  \and
  Nicolas Perkowski\thanks{Financial support by the DFG via the Heisenberg program and via Research Unit FOR 2402 is gratefully acknowledged.} \\
  Institut f\"ur Mathematik, Humboldt--Universit\"at zu Berlin \\
  \& Max-Planck-Institut f\"ur Mathematik in den Naturwissenschaften, Leipzig \\
  \texttt{perkowsk@math.hu-berlin.de}}

\makeglossaries

\newglossaryentry{varphij}
{
	name={$\varphi_j^{\GG}$},
	description={(discrete) Dyadic partition unity},
	sort={fj}
}

\newglossaryentry{varDeltaj}
{
	name={$\varDelta_j$},
	description={Littlewood-Paley block constructed from the dyadic partititon of unity \gls[noindex]{varphij}},
	sort={deltaj}
}

\newglossaryentry{PsiGGj}
{
	name ={$\Psi^{\GG,j}$},
	description={Fourier transform of $\varphi^{\GG}_j$},
	sort={psij}
}

\newglossaryentry{PsiGGleqj}
{
	name={$\Psi^{\GG,<j}$},
	description={Abbreviation for $\sum_{i<j} \Psi^{\GG,i}$},
	sort={psijj}
}

\newglossaryentry{GG}
{
	name={$\GG,\,\GG^\eps$},
	description={Bravais lattices, $\GG^\eps=\eps\cdot \GG$ denotes the scaled lattice},
	sort={g}
}

\newglossaryentry{widehatGG}
{
	name={$\widehat{\GG}$},
	description={Fourier cell for a Bravais lattice $\GG$},
	sort={g}
}

\newglossaryentry{ww}
{
	name={$\ww$},
	description={Set of functions $\wpol,\,\wexps$ that classify weights},
	sort={o}
}

\newglossaryentry{rromega}
{
	name={$\rr(\omega)$},
	description={The set of weights, whose growth/decay is controlled by $\omega\in \mbox{\gls[noindex]{ww}}$},
	sort={ro}
}

\newglossaryentry{mathscrEeps}
{
	name={$\EE^\eps$},
	description={Extension from Bravais lattices $\GG^\eps$ to $\RR^d$},
	sort={Eeps}
}

\newglossaryentry{Balphapq}
{
	name={$\B^\alpha_{p,q}$},
	description={Besov space},
	sort={Bapq}
}

\newglossaryentry{NN}
{
	name={$\NN$},
	description={Natural numbers including $0$, $\NN=\{0,1,2,\ldots\}$},
	sort={N}
}

\newglossaryentry{FF}
{
	name={$\FFr$},
	description={Fourier transform with convention: $\FFr f(x)=\int \dd \xi\, e^{2\pi \ii x \xi }\,f(\xi)$},
	sort={f}
}

\newglossaryentry{mmomega}
{
	name={$\mm(\omega)$},
	description={Set of jump measures for symmetric random walks},
	sort={mo}
}

\newglossaryentry{spectralsupport}
{
	name={Spectral support},
	description={The support of the Fourier transform of an (ultra-) distribution},
	sort={ss}
}

\newglossaryentry{times}
{
	name={$\times$},
	description={Symbol used to connect products with factors in different lines},
	sort={0}
}

\newglossaryentry{Sww}
{
	name={$\Sww$},
	description={Ultra-differentiable Schwartz functions},
	sort={So}
}

\newglossaryentry{Eww}
{
	name={$\Eww$},
	description={Ultra-differentiable functions},
	sort={Co}
}

\newglossaryentry{Cnb}
{
	name={$C^n_b$},
	description={Functions with bounded derivatives up to order $n$},
	sort={Cnb}
}

\newglossaryentry{lesssim}
{
	name={$\lesssim$},
	description={Means $\leq$ ``up to a multiplicative, deterministic constant''},
	sort={0l}
}

\newglossaryentry{lsim}
{
	name={$\lsim$},
	description={Used for indices $i,j\in \ZZ$. Means $\leq$ ``up to an additive, deterministic constant''},
	sort={0l}
}

\newglossaryentry{LpRRdrho}
{
	name={$L^p(\RR^d,\rho)$},
	description={Weighted $L^p$ space with the convention $\|f\|_{L^p(\RR^d,\rho)}=\|f\cdot \rho\|_{L^p(\RR^d)}$},
	sort={Lp}
}

\newglossaryentry{CR}
{
	name={$\CR$},
	description={Reciprocal Lattice},
	sort={R}
}

\newglossaryentry{jGG}
{
	name={$j_{\GG}$},
	description={The index where $\supp\,\varphi_j$ touches $\partial\,\wGG$},
	sort={jg}
}

\newglossaryentry{varDeltaGGj}
{
	name={$\varDelta_j^\GG$},
	description={Discrete Littlewood-Paley block},
	sort={djg}
}

\newglossaryentry{pkappa}
{
	name={$p^\kappa$},
	description={Polynomial, decaying weight $p^\kappa(x)=(1+|x|)^{-\kappa}$},
	sort={pk}
}

\newglossaryentry{esigmal}
{
	name={$e^\sigma_l$},
	description={Time-dependent, sub-exponential weight},
	sort={esl}
}

\newglossaryentry{mathscrLgammaalphapT}
{
	name={$\mathscr{L}^{\gamma,\alpha}_{p,T}$},
	description={Parabolic space},
	sort={Lgapt}
}

\newglossaryentry{mathcalMgammaTX}
{
	name={$\mathcal{M}^\gamma_T X$},
	description={Weighted space},
	sort={Mgtx}
}

\newglossaryentry{lepsmu}
{
	name={$l^\eps_\mu$},
	description={Fourier multiplier for the diffusion operator $L^\eps_\mu$},
	sort={lem}
}

\newglossaryentry{mathscrLepsmu}
{
	name={$L^{(\eps)}_\mu,\,\mathscr{L}^{(\eps)}_\mu$},
	description={(discrete) Diffusion operator and its associated operator $\mathscr{L}^{(\eps)}_\mu=\partial_t-L^{(\eps)}_\mu$},
	sort={Lepsmu}
}

\newglossaryentry{nwGG}
{
	name={$[\cdot]_{\wGG}$},
	description={Periodic map from $\RR^d$ to $\wGG$},
	sort={0wG}
}

\newglossaryentry{para}
{
	name={$\naivepara$},
	description={Paraproduct, either on $\RR^d$ or on a Bravais lattice},
	sort={0p}
}

\newglossaryentry{reso}
{
	name={$\naivereso$},
	description={Resonant term, either on $\RR^d$ or on a Bravais lattice},
	sort={0r}
}

\newglossaryentry{diamond}
{
	name={$\diamond$},
	description={Wick product},
	sort={0w}
}

\newglossaryentry{bullet}
{
	name={$\bullet$},
	description={Renormalized resonant term},
	sort={0r2}
}

\newglossaryentry{mpara}
{
	name={$\naivempara$},
	description={Modified paraproduct},
	sort={0pm}
}

\newglossaryentry{blacklozenge}
{
	name={$\blacklozenge$},
	description={Renormalized product for PAM (on $\RR^2$)},
	sort={0r3}
}

\newglossaryentry{mathscrDgammaalphapT}
{
	name={$\mathscr{D}^{\gamma,\alpha}_{p,T}$},
	description={Space of paracontrolled distributions for PAM},
	sort={Dga}
}

\newglossaryentry{mathcalCalpha}
{
	name={$\mathcal{C}^\alpha$},
	description={Besov space with $p=q=\infty$},
	sort={Ca}
}

\newglossaryentry{mathcalCalphap}
{
	name={$\mathcal{C}^\alpha_p$},
	description={Besov space with $q=\infty$},
	sort={Cap}
}

\newglossaryentry{FFg}
{
	name={$\FFg$},
	description={Fourier transform on a Bravais lattice $\GG$},
	sort={fg}
}

\newglossaryentry{psi}
{
	name={$\psi$},
	description={Smear function},
	sort={psi}
}

\begin{document}

\maketitle

\begin{abstract}
   We develop a discrete version of paracontrolled distributions as a tool for deriving scaling limits of lattice systems, and we provide a formulation of paracontrolled distributions in weighted Besov spaces. Moreover, we develop a systematic martingale approach to control the moments of polynomials of i.i.d. random variables and to derive their scaling limits. As an application, we prove a weak universality result for the parabolic Anderson model: We study a nonlinear population model in a small random potential and show that under weak assumptions it scales to the linear parabolic Anderson model. 
\end{abstract}

%




\section{Introduction}
Paracontrolled distributions were developed in~\cite{GIP} to solve \emph{singular SPDEs}, stochastic partial differential equations that are ill-posed because of the interplay of very irregular noise  and nonlinearities. A typical example is the two-dimensional continuous parabolic Anderson model,
\[
   \partial_t u = \Delta u + u \xi - u \infty,
\]
where $u\colon \RR_+ \times \RR^2 \to \RR$ and $\xi$ is a space white noise, the centered Gaussian distribution whose covariance is formally given by $\mathbb{E}[\xi(x) \xi(y)] = \delta(x-y)$. The irregularity of the white noise prevents the solution from being a smooth function, and therefore the product between $u$ and the distribution $\xi$ is not well defined. To make sense of it we need to eliminate some resonances between $u$ and $\xi$ by performing an infinite renormalization that replaces $u \xi$ by $u \xi - u \infty$. The motivation for studying singular SPDEs comes from mathematical physics, because they arise in the large scale description of natural microscopic dynamics. For example, if for the parabolic Anderson model we replace the white noise $\xi$ by its periodization over a given box $[-L,L]^2$, then it was recently shown in~\cite{Chouk2017} that the solution $u$ is the limit of $u^\eps(t,x) = e^{-c^\eps t} v^\eps( t/\eps^{2},  x/\eps)$, where $v^\eps \colon \RR_+ \times \{-L/\eps, \dots, L/\eps \}^2 \to \RR$ solves the lattice equation
\[
   \partial_t v^\eps = \Delta^\eps v^\eps + \eps v^\eps \eta,
\]
where $\Delta^\eps$ is the periodic discrete Laplacian and $(\eta(x))_{x \in \{-L/\eps, \dots, L/\eps\}^2}$ is an i.i.d. family of centered random variables with unit variance and sufficiently many moments.

Results of this type can be shown by relying more or less directly on paracontrolled distributions as they were developed in~\cite{GIP} for functions of a continuous space parameter. But that approach comes at a cost because it requires us to control a certain random operator, which is highly technical and a difficulty that is not inherent to the studied problem. Moreover, it just applies to lattice models with polynomial nonlinearities. See the discussion below for details. Here we formulate a version of paracontrolled distributions that applies directly to functions on Bravais lattices and therefore provides a much simpler way to derive scaling limits and never requires us to bound random operators. Apart from simplifying the arguments, our new approach also allows us to study systems on infinite lattices that converge to equations on $\RR^d$, while the formulation of the Fourier extension procedure we sketch below seems much more subtle in the case of an unbounded lattice. Moreover, we can now deal with non-polynomial nonlinearities which is crucial for our main application, a weak universality result for the parabolic Anderson model. Besides extending paracontrolled distributions to Bravais lattices we also develop paracontrolled distributions in weighted function spaces, which allows us to deal with paracontrolled equations on unbounded spaces that involve a spatially homogeneous noise. And finally we develop a general machinery for the use of discrete Wick contractions in the renormalization of discrete, singular SPDEs with i.i.d. noise which is completely analogous to the continuous Gaussian setting, and we build on the techniques of~\cite{Caravenna} to provide a criterion that identifies the scaling limits of discrete Wick products as multiple Wiener-It\^o integrals.

Our main application is a weak universality result for the two-dimensional parabolic Anderson model. We consider a nonlinear population model $v^\eps \colon \RR_+ \times \ZZ^2 \to \RR$,
\begin{equation}\label{eq:intro-nonlin-pam}
   \partial_t v^\eps(t,x) = \Delta^{(\mathrm d)} v^\eps(t,x) + F(v^\eps(t,x)) \eta^\eps(x),
\end{equation}
where $\Delta^{(\mathrm d)}$ is the discrete Laplacian, $F \in C^2$ has a bounded second derivative and satisfies $F(0) = 0$, and $(\eta^\eps(x))_{x \in \ZZ^2}$ is an i.i.d. family of random variables with $\mathrm{Var}(\eta^\eps(0)) = \eps^2$ and $\mathbb{E}[\eta^\eps(0)] = - F'(0) \eps^2 c^\eps$ for a suitable sequence of diverging constants $c^\eps \sim |\log \eps|$. The variable $v^\eps(t,x)$ describes the population density at time $t$ in the site $x$. The classical example would be $F(u) = u$, which corresponds to the discrete parabolic Anderson model in a small potential $\eta^\eps$. In that case $v^\eps$ describes the evolution of a population where every individual performs an independent random walk and finds at every site $x$ either favorable conditions if $\eta^\eps(x) > 0$ that allow the individual to reproduce at rate $\eta^\eps(x)$, or non-favorable conditions if $\eta^\eps(x)<0$ that kill the individual at rate $-\eta^\eps(x)$. We can include some interaction between the individuals by choosing a nonlinear function $F$. For example, $F(u) = u(C - u)$ models a saturation effect which limits the overall population size in one site to $C$ because of limited resources. In Section~\ref{sec:PAM} we will prove the following result:
\begin{theorem*}[see Theorem~\ref{thm:ConvergencePAM}]
	Assume that $F$ and $(\eta^\eps(x))$ satisfy the conditions described above and also that the $p$-th moment of $\eta^\eps(0)$ is uniformly bounded in $\eps$ for some $p>14$. Then there exists a unique solution $v^\eps$ to~\eqref{eq:intro-nonlin-pam} with initial condition $v^\eps(0,x) = \mathbf{1}_{\cdot = 0}$, up to a possibly finite explosion time $T^\eps$ with $T^\eps \to \infty$ for $\eps \to 0$, and $u^\eps(t,x) = \eps^{-2} v^\varepsilon(\eps^{-2} t, \eps^{-1} x)$ converges in law to the unique solution $u \colon \RR_+ \times \RR^2 \to \RR$ of the linear continuous parabolic Anderson model
	\[
	   \partial_t u = \Delta u + F'(0) u \xi - F'(0)^2 u \infty, \qquad u(0) = \delta,
	\]
	where $\delta$ denotes the Dirac delta.
\end{theorem*}

\begin{remark}
	It may appear more natural to assume that $\eta^\eps(0)$ is centered. However, we need the small shift of the expectation away from zero in order to create the renormalization $-F'(0)^2 u \infty$ in the continuous equation. Making the mean of the variables $\eta^\eps(x)$ slightly negative (assume $F|_{[0,\infty)} \ge 0$ so that $F'(0) \ge 0$) gives us a slightly higher chance for a site to be non-favorable than favorable. Without this, the population size would explode in the scale in which we look at it. A similar effect can also be observed in the Kac-Ising/Kac-Blume-Capel model, where the renormalization appears as a shift of the critical temperature away from its mean field value~\cite{Mourrat2017, Shen2016}. Note that in the linear case $F(u)=u$ we can always replace $\eta^\eps$ by $\eta^\eps + c$ if we consider $e^{c t} v^\eps(t)$ instead. So in that case it is not necessary to assume anything about the expectation of $\eta^\eps$, we only have to adapt our reference frame to its mean.
\end{remark}
\begin{remark}
The condition $p>14$ might seem rather arbitrary. Roughly speaking this requirement is needed to apply a form of Kolmogorov's continuity criterion, see Remark \ref{rem:fourteen} for details.
\end{remark}

\paragraph{Structure of the paper}

Below we provide further references and explain in more details where to place our results in the current research in singular SPDEs and we fix some conventions and notations. In Sections~\ref{sec:BravaisLattices}-~\ref{sec:ParacontrolledAnalysisonBravaisLattices} we develop the theory of paracontrolled distributions on unbounded Bravais lattices, and in particular we derive Schauder estimates for quite general random walk semigroups. Section~\ref{sec:PAM} contains the weak universality result for the parabolic Anderson model, and here we present our general methodology for dealing with multilinear functionals of independent random variables. The appendix contains several proofs that we outsourced. Finally, there is a list of important symbols at the end of the paper.  

\paragraph{Related works}

As mentioned above, we can also use paracontrolled distributions for functions of a continuous space parameter to deal with lattice systems. The trick, which goes back at least to \cite{Mourrat2017} and was inspired by~\cite{Hairer2012Spatial}, is to consider for a lattice function $u^\eps$ on say $\{k\eps: -L/\eps \le k \le L/\eps\}^2$ the unique periodic function $\mathrm{Ext}(u^\eps)$ on $(\RR / (2L\ZZ))^2$ whose Fourier transform is supported in $[-1/\eps,1/\eps]^2$ and that agrees with $u^\eps$ in all the lattice points. If the equation for $u^\eps$ involves only polynomial nonlinearities, we can write down a closed equation for $\mathrm{Ext}(u^\eps)$ which looks similar to the equation for $u^\eps$ but involves a certain ``Fourier shuffle'' operator that is not continuous on the function spaces in which we would like to control $\mathrm{Ext}(u^\eps)$. But by introducing a suitable random operator that has to be controlled with stochastic arguments one can proceed to study the limiting behavior of $\mathrm{Ext}(u^\eps)$ and thus of $u^\eps$. This argument has been applied to show the convergence of lattice systems to the KPZ equation~\cite{Gubinelli2017KPZ}, the $\Phi^4_3$ equation~\cite{Zhu2015}, and to the parabolic Anderson model~\cite{Chouk2017}, and the most technical part of the proof was always the analysis of the random operator. The same argument was also applied to prove the convergence of the Kac-Ising / Kac-Blume-Capel model~\cite{Mourrat2017, Shen2016} to the $\Phi^4_2$ / $\Phi^6_2$ equation. This case can be handled without paracontrolled distributions, but also here some work is necessary to control the Fourier shuffle operator. This difficulty is of a technical nature and not inherent to the studied problems, and the line of argumentation we present here avoids that problem by analysing directly the lattice equation rather than trying to interpret it as a continuous equation.

Other intrinsic approaches to singular SPDEs on lattices have been developed in the context of regularity structures by Hairer, Matetski and Erhard~\cite{Hairer2015Discretisations, Erhard2017} and in the context of the semigroup approach to paracontrolled distributions by Bailleul and Bernicot~\cite{Bailleul2016}, and we expect that both of these works could be combined with our martingale arguments of Section~\ref{sec:PAM} to give an alternative proof of our weak universality result.

We call the convergence of the nonlinear population model to the linear parabolic Anderson model a ``weak universality'' result in analogy to the weak universality conjecture for the KPZ equation. The (strong) KPZ universality conjecture states that a wide class of (1+1)-dimensional interface growth models scale to the same universal limit, the so called KPZ fixed point~\cite{Matetski2016}, while the weak KPZ universality conjecture says that if we change some ``asymmetry parameter'' in the growth model to vanish at the right rate as we scale out, then the limit of this \emph{family} of models is the KPZ equation. Similarly, here the influence of the random potential on the population model must vanish at the right rate as we pass to the limit, so the parabolic Anderson model arises as scaling limit of a \emph{family} of models. Similar weak universality results have recently been shown for other singular SPDEs such as the KPZ equation~\cite{Goncalves2014, HairerQuastel, Gubinelli2015Energy, Gubinelli2016Hairer} (this list is far from complete), the $\Phi^{2n}_d$ equations~\cite{Mourrat2017, Hairer2016Large, Shen2016}, or the (stochastic) nonlinear wave equation~\cite{Gubinelli2017, Oh2017}.

A key task in singular stochastic PDEs is to renormalize and to construct certain a priori ill-defined products between explicit stochastic processes. This problem already arises in rough paths \cite{Lyons} but there it is typically not necessary to perform any renormalizations and general construction and approximation results for Gaussian rough paths were developed in~\cite{Friz2010}. For singular SPDEs the constructions become much more involved and a general construction of regularity structures for equations driven by Gaussian noise was found only recently and is highly nontrivial~\cite{Bruned2016, Chandra2016}. For Gaussian noise it is natural to regroup polynomials of the noise in terms of Wick products, which goes back at least to~\cite{DaPratoDebussche} and is essentially always used in singular SPDEs, see~\cite{SolvingKPZ, RegularityStructures, CatellierChouk, Gubinelli2017KPZ} and many more. Moreover, in the Gaussian case all moments of polynomials of the noise are equivalent, and therefore it suffices to control variances. In the non-Gaussian case we can still regroup in terms of Wick polynomials~\cite{Mourrat2017, Hairer2015Central, Chandra2016Moment, Shen2016Weak}, but a priori the moments are no longer comparable and new methods are necessary. In~\cite{Mourrat2017} the authors used martingale inequalities to bound higher order moments in terms of variances.

In our case it may look as if there are no martingales around because the noise is constant in time. But if we enumerate the lattice points and sum up our i.i.d. variables along this enumeration, then we generate a martingale. This observation was used in~\cite{Chouk2017} to show that for certain polynomial functionals of the noise (``discrete multiple stochastic integrals'') the moments are still comparable, but the approach was somewhat ad-hoc and only applied directly to  the product of two variables in ``the first chaos''. 

Here we develop a general machinery for the use of discrete Wick contractions in the renormalization of discrete, singular SPDEs with i.i.d. noise which is completely analogous to the continuous Gaussian setting. Moreover, we build on the techniques of~\cite{Caravenna} to provide a criterion that identifies the scaling limits of discrete Wick products as multiple Wiener-It\^o integrals. Although these techniques are only applied to the discrete $2d$ parabolic Anderson model, the approach extends in principle to any discrete formulation of popular singular SPDEs such as the KPZ equation or the $\Phi^4_d$ models.

\subsection{Conventions and Notation}
We use the common notation $\lesssim,\,\gtrsim$ in estimates to denote $\leq,\,\geq$ up to a positive constant. The symbol $\approx$ means that both $\lesssim$ and $\gtrsim$ hold true. For discrete indices we mean by $i \lsim j$ that there is a $N\geq 0$ (independent of $i,j$) such that $i\leq j+N$, i.e. that $2^i \lsim 2^j$, and similarly for $j \gsim i$; the notation $i \sim j$ is shorthand for $i\lsim j$ and $j\lsim i$.  

We denote partial derivatives by $\partial^{\alpha}$ for $\alpha\in \mathbb{N}^d := \{0,1,2, \dots\}^d$ and for $\alpha=(\1_{i = j})_j$ we write $\partial^i = \partial^\alpha$. Our Fourier transform follows the convention that for $f\in L^1(\RR^d)$  
\begin{align*}
	\FFr f(y) :=\int_{\RR^d} f(x) e^{-2\pi \imath x \scl y} \, \dd x
,\qquad
	\FFr^{-1} f (x) &:=\int_{\RR^d} f(y) e^{2\pi \imath x \scl y} \,\dd y\,,
\end{align*} 
where $x \scl y$ denotes the usual inner product on $\RR^d$. The most relevant notations are listed in a glossary at the  end of this article.

\section{Weighted Besov spaces on Bravais lattices}
\label{sec:BravaisLattices}

\subsection{Fourier transform on Bravais lattices}
\label{subsec:LatticeFourierTransform}

A \emph{Bravais-lattice} in $d$ dimensions consists of the integer combinations of $d$ linearly independent vectors $a_1,\ldots,a_d \in \RR^d $, that is
\glsadd{GG}
\begin{align}
\label{eq:Lattice}
	\GG:=\ZZ\, a_1 + \ldots +\ZZ\, a_d\,.
\end{align}
Given a Bravais lattice we define the basis $\widehat{a}_1,\ldots,\widehat{a}_d$  of the reciprocal lattice by the requirement 
\begin{align}
\label{eq:ReciprocalLattice}
	\widehat{a}_i \scl a_j= \delta_{ij}\,,
\end{align}
and we set $\mathscr{R} := \ZZ\, \widehat{a}_1 +\ldots + \ZZ\,\widehat{a}_d$.
\glsadd{CR}
However, we will mostly work with the (centered) parallelotope which is spanned by the basis vectors $\widehat{a}_1,\ldots,\widehat{a}_d$:
\glsadd{widehatGG}
\begin{align*}
	\widehat{\GG} &:=[0,1)\,\widehat{a}_1+\ldots +[0,1)\, \widehat{a}_d-\frac{1}{2}(\widehat{a}_1 + \ldots +\widehat{a}_d)=[-1/2,1/2)\,\widehat{a}_1+\ldots +[-1/2,1/2)\, \widehat{a}_d\,.
\end{align*} 
We call $\widehat{\GG}$ the \textit{bandwidth} or \textit{Fourier-cell} of $\GG$ to indicate that the Fourier transform of a map on $\GG$ lives on $\widehat{\GG}$, as we will see below. We also identify $\widehat{\GG}\simeq \RR^d / \mathscr{R}$ and turn $\widehat{\GG}$ into an additive group which is invariant under translations by elements in $\mathscr{R}$. 

\begin{example}
If we choose the canonical basis vectors $a_1=e_1,\ldots,a_d=e_d$, we have simply
\begin{equation*}
	\GG =\ZZ^d \,,\qquad \mathscr{R} = \ZZ^d \,,\qquad \widehat{\GG} =\mathbb{T}^d=[-1/2,1/2)^d \,.
\end{equation*}
Compare also the left lattice in Figure~\ref{fig:BravaisLattices}.
\end{example}

In Figure \ref{fig:BravaisLattices} we sketched some Bravais lattices $\GG$ together with their Fourier cells $\widehat{\GG}$. Note that the dashed lines  between the points of the lattice are at this point a purely artistic supplement. However, they will become meaningful later on: If we imagine a particle performing a random walk on the lattice $\GG$, then the dashed lines could be interpreted as the jumps it is allowed to undertake. From this point of view the lines will be drawn by the diffusion operators we introduce in Section ~\ref{sec:DiffusionOperators}.

\begin{figure}[t!]
\label{fig:BravaisLattices}
\centering 
\includegraphics[width=\textwidth]{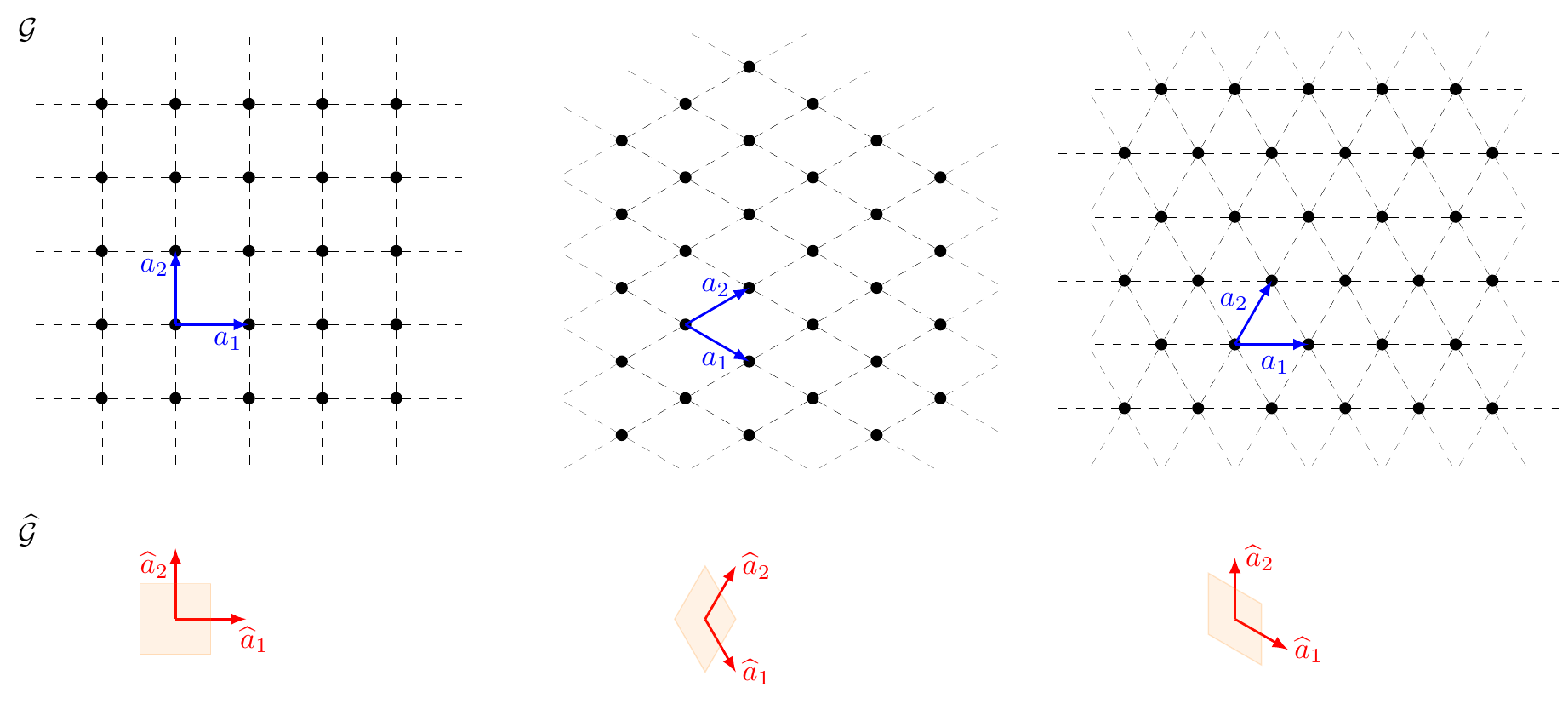}
\caption{Depiction of some Bravais lattices $\GG$ with their bandwiths $\widehat{\GG}$: a square lattice, an oblique lattice and the so called hexagonal lattice. The length of the reciprocal vectors $\widehat{a}_i$ is rather arbitrary since it actually depends on the units in which we measure $a_i$.}
\end{figure}

\begin{definition}
\label{def:LatticeSequence}
Given a Bravais lattice $\GG$ as defined in \eqref{eq:Lattice} we write 
\begin{align*}
	\GG^\eps:=\eps \GG
\end{align*}
for the sequence of Bravais lattice we obtain by \textit{dyadic} rescaling with $\eps=2^{-N},\, N\geq 0$. 
Whenever we say a statement (or an estimate) holds for $\GG^\eps$ we mean that it holds (uniformly) for all $\eps=2^{-N},\,N\geq 0$. 
\end{definition}
\begin{remark}
\label{rem:DyadicLatticeSequence}
The restriction to dyadic lattices fits well with the use of Littlewood-Paley theory which is traditionally built from dyadic decompositions. However, it turns out that we do not lose much generality by this. Indeed, all the estimates below will hold uniformly as soon as we know that the scale of our lattice is contained in some interval $(c_1,c_2)\subset\subset (0,\infty)$. Therefore it is sufficient to group the members of any positive null-sequence $(\eps_n)_{n \geq 0}$ in dyadic intervals $[2^{-(N+1)},2^{-N})$ to deduce the general statement. 
\end{remark}

Given $\varphi \in \ell^1(\GG)$ we define its Fourier transform as
\glsadd{FFg}
\begin{equation}
   \FFg \varphi(x):=|\GG| \sum_{k\in\GG} \varphi(k) e^{- 2\pi \imath k \scl x},\qquad x \in \widehat{\GG},
\end{equation}
where we introduced a ``normalization constant'' $|\GG|:=|\det\left(a_1,\ldots,a_d\right)|$
that ensures that we obtain the usual Fourier transform on $\RR^d$ as $|\GG|$ tends to 0.
We will also write $|\widehat{\GG}|$ for the Lebesgue measure of the Fourier cell $\wGG$.

If we consider $\FFg \varphi$ as a map on $\RR^d$, then it is periodic under translations in $\mathscr{R}$. By the dominated convergence theorem $\FFg \varphi$ is continuous, so since $\wGG$ is compact it is in $L^1(\wGG):= L^1(\wGG, \dd x)$, where $\dd x$ denotes integration with respect to the Lebesgue measure. For any $\psi \in L^1(\wGG)$ we define its inverse Fourier transform as
\begin{equation}
   \FFg^{-1} \psi(k) :=  \int_{\wGG} \psi(x) e^{2\pi \imath k \scl x} \dd x,\qquad k \in \GG.
\end{equation}
Note that $|\GG| = 1/|\wGG|$ and therefore we get at least for $\varphi$ with finite support $\FFg^{-1} \FFg \varphi = \varphi$. The Schwartz functions on $\GG$ are
\[
   \Sw(\GG) := \left\{ \varphi \colon \GG \to \CC : \sup_{k \in \GG} (1+|k|)^m |\varphi(k)| < \infty \text{ for all } m \in \NN\right\},
\]
and we have $\FFg \varphi \in C^\infty(\widehat{\GG})$ (with periodic boundary conditions) for all $\varphi \in \Sw(\GG)$, because for any multi-index $\alpha \in \NN^d$ the dominated convergence theorem gives
\[
   \partial^\alpha \FFg \varphi(x) = |\GG| \sum_{k\in\GG} \varphi (k) (-2\pi \imath k)^{\alpha} e^{- 2\pi \imath k \scl x}.
\]
By the same argument we have $\FFg^{-1} \psi \in \Sw(\GG)$ for all $\psi \in C^\infty(\widehat{\GG})$, and as in the classical case $\GG = \ZZ^d$ one can show that $\FFg$ is an isomorphism from $\Sw(\GG)$ to $C^\infty(\wGG)$ with inverse $\FFg^{-1}$.
Many relations known from the $\ZZ^d$-case carry over readily to Bravais lattices, e.g. Parseval's identity
\begin{align}
	\label{eq:ParsevalsIdentity}
	\sum_{k\in \GG} |\GG|\cdot |\varphi(k)|^2=\int_{\widehat{\GG}} |\widehat{\varphi}(x)|^2  \,\dd x
\end{align}
(to see this check for example with the Stone-Weierstrass theorem that $(|\GG|^{1/2}e^{2 \pi \imath k \cdot})_{k \in \GG}$ forms an orthonormal basis of $L^2(\widehat{\GG}, \dd x)$) and the relation between convolution and multiplication
\begin{align}
\label{eq:ConvolutionLattice}
	\FFg\left(\varphi_1 \ast_{\GG} \varphi_2\right)(x) & :=\FFg\left(\sum_{k\in \GG } |\GG|\, \varphi_1(k) \varphi_2(\cdot-k) \right)(x)= \FFg\varphi_1(x) \cdot \FFg \varphi_2(x), \\
	\FFg^{-1}\left(\psi_2 \ast_{\widehat{\GG}} \psi_2\right)(k)&:=\FFg^{-1} \left( \int_{\widehat{\GG}} \psi_1(x) \psi_2([\plc-x]_{\wGG})\dd x \right)(k)=\FFg^{-1}\psi_1(k) \cdot \FFg^{-1} \psi_2(k).
\label{eq:ConvolutionBandwidth}
\end{align}
where $[z]_{\wGG}$ is for $z\in \RR^d$ the unique element in $\wGG$ such that $z-[z]_{\wGG}\in \CR$.
\glsadd{nwGG}

Since $\Sw(\GG)$ consists of functions decaying faster than any polynomial, the Schwartz distributions on $\GG$ are the functions that grow at most polynomially,
\[
   \Sw'(\GG) := \left\{ f \colon \GG \to \CC : \sup_{k \in \GG} (1+|k|)^{-m} |f(k)| < \infty \text{ for some } m \in \NN\right\},
\]
and $f(\varphi) := |\GG| \sum_{k \in \GG} f(k) \varphi(k)$ is well defined for $\varphi\in \Sw(\GG)$. We extend the Fourier transform to $\Sw'(\GG)$ by setting
\[
   (\FFg f)(\psi)  := f\left( \overline{\FFg^{-1}\overline{\psi}}\right) = |\GG| \sum_{k \in \GG} f(k) \overline{\FFg^{-1}\overline{\psi}}(k), \qquad \psi \in C^\infty(\widehat{\GG}),
\]
where $\overline{\,(\ldots) \,}$ denotes the complex conjugate. This should be read as $(\FFg f)(\psi) = f(\FFg \psi)$, which however does not make any sense because for $\psi \in C^\infty(\widehat{\GG})$ we did not define the Fourier transform $\FFg \psi$ but only $\FFg^{-1} \psi$. The Fourier transform $(\FFg f)(\psi)$ agrees with $\int_{\wGG} \FFg f(x)  \cdot \psi(x) \,\dd x$ in case $f \in \Sw(\GG)$. It is possible to show that $\widehat{f} \in \Sw'(\wGG)$, where
\[
   \Sw'(\wGG):= \{u\colon C^\infty(\wGG) \to \CC: u \text{ is linear and } \exists\, C>0, m\in \NN \text{ s.t. } |u(\psi)| \le C \|\psi\|_{C^m_b(\wGG)}\}
\]
for $\|\psi\|_{C^m_b(\wGG)} := \sum_{|\alpha| \le m} \|\partial^\alpha \psi \|_{L^\infty(\widehat{\GG})}$, and that $\FFg$ is an isomorphism from $\Sw'(\GG)$ to $\Sw'(\wGG)$ with inverse
\begin{align}
\label{eq:InverseFourierTransformLattice}
   (\FFg^{-1} u)(\varphi) := |\GG| \sum_{k \in \GG} u(e^{2\pi \imath k \scl (\plc)}) \varphi(k).
\end{align}

As in the classical case $\GG=\mathbb{Z}$ it is easy to see that we can identify every $f\in \Sw'(\GG)$ with a ``Dirac comb'' distribution $f_{\mathrm{dir}}\in \Sw'(\RR^d)$ by setting
\begin{align}
\label{eq:DiracExtension}
f_{\mathrm{dir}}=|\GG|\sum_{k\in\GG} f(k) \delta(\cdot-k),
\end{align}
where $\delta(\cdot-k)\in \Sw'(\RR^d)$ denotes a shifted Dirac delta distribution. We can identify any element $g\in \Sw'(\widehat{\GG})$ of the frequency space with an $\mathscr{R}$-periodic distribution $g_{\mathrm{ext}}\in \Sw'(\RR^d)$ by setting
\label{IdentificationsFourierTheory}
\begin{align}
\label{eq:PeriodicExtension}
	g_{\mathrm{ext}}(\varphi):=g\left(\sum_{k\in \mathscr{R}} \varphi(\cdot-k)\right),\qquad\varphi\in \Sw(\RR^d)\,.
\end{align}
If $g\in \Sw'(\widehat{\GG})$ coincides with a function on $\widehat{\GG}$ one sees that
\begin{align}
\label{eq:PeriodicExtension2}
g_{\mathrm{ext}}(x)=g([x]_{\wGG})
\end{align}
where $[x]_{\wGG}$ is, as above, the (unique) element $[x]_{\wGG}\in\widehat{\GG}$ such that $[x]_{\widehat{\GG}}-x\in \ZZ \widehat{a}_1+\ldots+\ZZ \widehat{a}_d=\CR$.
Conversely, every $\mathscr{R}$-periodic distribution $g\in\Sw'(\RR^d)$ can be seen as a restricted element $g_{\mathrm{res}}\in\Sw'(\widehat{\GG})$, e.g. by considering 
\begin{align}
\label{eq:PeriodicRestriction}
g_{\mathrm{res}}(\varphi):= (\psi\cdot g)(\varphi_{\mathrm{ext}})=  g(\psi \cdot \varphi_{\mathrm{ext}}),\qquad \varphi\in C^\infty(\widehat{\GG})
\end{align}
where $\psi\in C^\infty_c(\RR^d)$ is chosen such that $\sum_{k\in\mathscr{R}} \psi(\cdot-k)=1$ and where we used in the second equality the definition of the product between a smooth function and a distribution. To construct such a $\psi$ it suffices to convolve $\1_{\widehat{\GG}}$ with a smooth, compactly supported mollifier, and it is easy to check that $(g_{\mathrm{ext}})_{\mathrm{res}} = g$ for all $g \in \Sw'(\GG)$ and that $g_{\mathrm{res}}$ does not depend on the choice of $\psi$. This motivates our definition of the extension operator $\EE$ below in Lemma \ref{lem:ClassicalCaseExtensionOperator}.

With these identifications in mind we can interpret the concepts introduced above as a sub-theory of the classical Fourier analysis of tempered distributions. We will sometimes use the following identity for $f\in \Sw'(\GG)$
\begin{align}
\label{eq:DiracComb}
\ext{\FFg f}=\FFr(f_{\mathrm{dir}})\,,
\end{align}
which is easily checked using the definitions above.

\smallskip
Next, we want to introduce Besov spaces on $\GG$. Recall that one way of constructing Besov spaces on $\RR^d$ is by making use of a dyadic partition of unity. 
\begin{definition}
\label{def:DyadicPartitionOfUnity}
A \emph{dyadic partition of unity} is a family $(\varphi_j)_{j\geq -1} \subseteq C^\infty_c(\RR^d)$ of nonnegative radial functions such that
\begin{itemize}
	\item $\supp \varphi_{-1}$ is contained in a ball around 0, $\supp \varphi_j $ is contained in an annulus around 0 for $j\geq 0\,$,
	\item $\varphi_j=\varphi_0(2^{-j}\cdot)$ for $j\geq 0\,$,
	\item $\sum_{j\geq -1} \varphi_j (x) =1$ for any $x\in\RR^d\,$,
	\item If $|j-j'|>1$ we have $\supp \varphi_j \cap \supp \varphi_{j'}=\emptyset\,$,
\end{itemize} 
\end{definition}

Using such a dyadic partition as a family of Fourier multipliers leads to the Littlewood-Paley blocks of a distribution $f \in \Sw'(\RR^d)$,
\begin{align*}
	\varDelta_j f:=\FFr^{-1} (\varphi_j \cdot \FFr f).
\end{align*}
Each of these blocks is a smooth function and it represents a ``spectral chunk'' of the distribution. By choice of the $(\varphi_j)_{j\geq -1}$ we have $f=\sum_{j\geq -1} \varDelta_j f$ in $\Sw'(\RR^d)$, and measuring the explosion/decay of the Littlewood-Paley blocks gives rise to the Besov spaces
\begin{align}
\label{eq:ClassicalBesovSpaceUnweighted}
	\mathcal{B}^\alpha_{p,q} (\RR^d)=\left\{ f\in \Sw'(\RR^d) : \| (2^{j\alpha} \| \varDelta_j f\|_{L^p})_{j\geq -1} \|_{\ell^q}<\infty \right\}.
\end{align}

In our case all the information about the Fourier transform of $f\in \Sw'(\GG)$, that is $\FFg f\in \Sw'(\widehat{\GG})$, is stored in a finite bandwidth $\widehat{\GG}$. Therefore, it is more natural to decompose the compact set $\wGG$, so that we consider only finitely many blocks. However, there is a small but delicate problem: We should decompose $\wGG$ in a smooth periodic way, but if $j$ is such that the support of $\varphi_j$ touches the boundary of $\widehat{\GG}$, the function $\varphi_j$ will not necessarily be smooth in a periodic sense. We therefore redefine the dyadic partition of unity for $x\in \wGG$ as
\glsadd{varphij}
\begin{align}
	\varphi^\GG_j(x)=
	\left\lbrace
	\begin{array}{ll}
	 \varphi_j(x), & j<j_\GG\, , \\
	 1-\sum_{j<j_\GG} \varphi_j(x), &j=j_\GG\,,
	\end{array}\right.
	\label{eq:LatticePhi}
\end{align}
\label{DiscreteDyadicPartition}
where \glsadd{jGG} $j\leq j_\GG:=\inf\{j\,:\, \supp \varphi_j\cap \partial\wGG \neq \emptyset\}$.
 Now we set for $f\in \Sww'(\GG)$
 \glsadd{varDeltaGGj}
\begin{align*}
\varDelta^\GG_j f := \FFg^{-1}(\varphi^\GG_j \cdot \FFg f)\,,
\end{align*}
which is now a function defined on $\GG$. As in the continuous case we will also use the notation $S_j^{\GG} f=\sum_{i<j} \varDelta_i^{\GG} f$.

 Of course, for a fixed $\GG$ it may happen that $\varDelta^\GG_{-1} = \mathrm{Id}$, but if we rescale the lattice $\GG$ to $\eps \GG$, the Fourier cell $\wGG$ changes to $\eps^{-1} \wGG$ and so for $\eps \to 0$ the following definition becomes meaningful.

\begin{definition}
Given $\alpha\in \RR$ and $p,q\in[1,\infty]$ we define 
\begin{align*}
\mathcal{B}^\alpha_{p,q}(\GG):=\left\{ f\in \Sw'(\GG) \,\vert \, \|f\|_{\mathcal{B}^\alpha_{p,q}(\GG)} = \| (2^{j\alpha } \|    \varDelta_j^\GG f\|_{L^p(\GG)})_{j=-1,\ldots,j_{\GG}}\|_{\ell^q}<\infty \right\},
\end{align*}
where we define the $L^p(\GG)$ norm by 
\begin{align}
	\label{eq:Normlp}
	\|f\|_{L^p(\GG)}:=\left(  |\GG| \sum_{k\in \GG} |f(k)|^p \right)^{1/p} = \| |\GG|^{1/p} f \|_{\ell^p}.
\end{align}
We write furthermore $\mathcal{C}^\alpha_p(\GG):=\mathcal{B}^\alpha_{p,\infty}(\GG)$. 

\end{definition}

The reader may have noticed that since we only consider finitely many $j=-1,\ldots,j_{\GG}$ (and since $\varDelta_j\colon L^p(\GG)\rightarrow L^p(\GG)$ is a bounded operator, uniformly in $j$, as we will see below), the two spaces  $\mathcal{B}^\alpha_{p,q}(\GG)$ and $L^p(\GG)$ are in fact identical with equivalent norms!
However, since we are interested in uniform bounds on $\GGe$ for $\eps \to 0$, we are of course not allowed to switch between these spaces. Whenever we consider sequences $\GG^\eps$ of lattices we construct all dyadic partitions of unity $(\varphi_j^{\GG^\eps})_{j=-1,\ldots,j_{\GG^\eps}}$ from the \emph{same} partition of unity $(\varphi_j)_{j\geq -1}$ on $\RR^d$.

With the above constructions at hand it is easy to develop a theory of paracontrolled distributions on a Bravais lattice $\GG$ which is completely analogous to the one on $\RR^d$. For the transition from the rescaled lattice models on $\GG^\eps$ to models on the Euclidean space $\RR^d$ we need to compare discrete and continuous distributions, so we should extend the lattice model to a distribution in $\Sw'(\RR^d)$. One way of doing so is to simply consider the identification with a Dirac comb, already mentioned in \eqref{eq:DiracExtension}, but this has the disadvantage that the extension can only be controlled in spaces of quite low regularity because the Dirac delta is quite irregular. We find the following extension convenient:

\begin{lemma}
\label{lem:ClassicalCaseExtensionOperator}
   Let $\psi \in C^\infty_c(\RR^d)$ be a positive function with $\sum_{k \in \mathcal R} \psi(\cdot - k) \equiv 1$ and set
   \[
      \EE f := \FFr^{-1}\big(\psi \cdot \ext{\FFg f}\big), \qquad f \in \Sw'(\GG),
   \]
where the periodic extension $\ext{\cdot}\colon\Sw'(\wGG)\rightarrow \Sw'(\RR^d)$ is defined as in \eqref{eq:PeriodicExtension}. Then $\EE f \in C^\infty(\RR^d) \cap \Sw'(\RR^d)$ and $\EE f(k) = f(k)$ for all $k \in \GG$.
\end{lemma}

\begin{proof}
   We have $\EE f \in \Sw'(\RR^d)$ because  $\ext{\FFg f}$ is in $\Sw'(\RR^d)$, and therefore also $\EE f=\FFr^{-1} ( \psi\cdot  \ext{\FFg f}) \in \Sw'(\RR^d)$.
%
%
%
   Knowing that $\EE f$ is in $\Sw'(\RR^d)$, it must be in $C^\infty(\RR^d)$ as well because it has compact spectral support by definition. Moreover,  we can write for $k \in \GG$
   \begin{align*}
      \EE f(k) & = \ext{\FFg f}(\psi \cdot  e^{2 \pi \imath k \scl(\cdot)})= \FFg f\left(\sum_{r\in \mathcal{R}} \psi(\cdot-r) e^{2\pi \imath k \scl (\cdot-r)} \right)=\mathcal{F}_{\GG}f(e^{2\pi \imath k\scl (\cdot)})=f(k)\,,
   \end{align*}
   where we used the definition of $\ext{\cdot}$ from \eqref{eq:PeriodicExtension} and that $k \scl r \in \ZZ$ for all $k \in \GG$ and $r \in \CR$.
\end{proof}

It is possible to show that if $\EE^\eps$ denotes the extension operator on $\GG^\eps$, then the family $(\EE^\eps)_{\eps > 0}$ is uniformly bounded in $L(\mathcal{B}^\alpha_{p,q}(\GG^\eps), \mathcal{B}_{p,q}^\alpha(\RR^d))$, and this can be used to obtain uniform regularity bounds for the extensions of a given family of lattice models.

\smallskip

However, since we are interested in equations with spatially homogeneous noise, we cannot expect the solution to be in $\mathcal{B}^\alpha_{p,q}(\GG)$ for any $\alpha, p, q$ and instead we have to consider weighted spaces. In the case of the parabolic Anderson model it turns out to be convenient to even allow for subexponential growth of the form $e^{|\cdot|^\sigma}$ for $\sigma \in (0,1)$, which means that we have to work on a larger space than $\Sw'(\GG)$, where only polynomial growth is allowed. So before we proceed let us first recall the basics of the so called \emph{ultra-distributions} on $\RR^d$.

\subsection{Ultra-distributions on Euclidean space}
\label{subsec:UltraDistributions}
A drawback of Schwartz's theory of tempered distributions is the restriction to polynomial growth. As we will see later, it is convenient to allow our solution to have subexponential growth of the form $ e^{\lambda |\cdot|^\sigma}$ for $\sigma\in (0,1)$ and $\lambda > 0$. It is therefore necessary to work in a larger space $\Sw_\omega'(\RR^d)\supseteq \Sw'(\RR^d)$, the space of so called \textit{(tempered) ultra-distributions}, which has less restrictive growth conditions but on which one still has a Fourier transform.
Similar techniques already appear in the context of singular SPDEs in \cite{WeberMourrat}, where the authors use Gevrey functions that are characterized by a condition similar to the one in Definition \ref{def:Gevrey} below. Here, we will follow a slightly different approach that goes back to Beurling and Björck \cite{Bjoerck}, and which mimics essentially the definition of tempered distribution via Schwartz functions. For a broader introduction to ultra-distributions see for example \cite[Chapter 6]{Triebel} or \cite{Bjoerck}. 
\smallskip

Let us fix, once and for all, the following weight functions which we will use throughout this article. 
\begin{definition}
\label{def:Weights}
We denote by 
\begin{align*}
\wpol(x):=\log(1+|x|),\,\wexps(x):=|x|^\sigma,\,\sigma\in (0,1)\,.
\end{align*}
where $x\in \RR^d,\,\sigma\in (0,1)$
For $\omega\in \boldsymbol{\omega}:=\{\wpol\}\cup \{\wexps\,\vert\,\sigma\in (0,1)\}$ we denote by $\boldsymbol{\rho}(\omega)$ the set of measurable, strictly positive $\rho:\RR^d\rightarrow (0,\infty) $ such that
\begin{align}
\label{eq:WeightQuotient}
\rho(x)\lesssim \rho(y) e^{\lambda \omega(x-y)}
\end{align}
for some $\lambda=\lambda(\rho)>0$. We also introduce the notation $\rr(\ww):=\bigcup_{\omega\in\ww}\rr(\omega)$. The objects $\rho\in \rr(\ww)$ will be called \emph{weights}. \glsadd{ww} \glsadd{rromega}
\end{definition}
Note that the sets $\boldsymbol{\rho}(\omega)$ are stable under addition and multiplication for a fixed $\omega\in \ww$. The indices ``$\mathrm{pol}$'' and ``$\mathrm{exp}$'' of the elements in $\ww$ indicate the fact that elements in $\rho\in \boldsymbol{\rho}(\wpol)$ are polynomially growing or decaying while elements in $\boldsymbol{\rho}(\wexps)$ are allowed to have subexponential behavior. Note that
\begin{align*}
 \rr(\wpol)\subseteq \rr(\wexps)
 \end{align*} 
 and that 
 \begin{align}
 \label{eq:Defw}
(1+|x|)^{\lambda}\in \rr(\wpol)
 \end{align}
 and $e^{\lambda|x|^\sigma}\in \rr(\wexps)$ for $\lambda\in \RR,\,\sigma\in (0,1)$. The reason why we only allow for $\sigma<1$ will be explained in Remark \ref{rem:Sigmageq1} below. 

We are now ready to define the space of ultra-distributions.
\begin{definition}
We define for $\omega\in\ww$ the locally convex space 
\begin{align}
\Sw_\omega(\RR^d):=\{f\in \Sw(\RR^d)\,\vert\, \forall	\lambda>0,\,\alpha\in \NN^d\qquad p_{\alpha,\lambda}^\omega (f)+\pi_{\alpha,\lambda}^\omega(f)<\infty\}\,,
\end{align}
which is equipped with the seminorms 
\begin{align}
p_{\alpha,\lambda}^\omega(f):=\sup_{x\in \RR^d} e^{\lambda \omega(x)} |\partial^\alpha f(x)|\,,\label{eq:SemiNormP} \\
\pi_{\alpha,\lambda}^\omega(f):=\sup_{x\in \RR^d} e^{\lambda \omega(x)} |\partial^\alpha  \FFr f(x)| \,.
\label{eq:SemiNormPi}
\end{align}
Its topological dual $\Sww'(\RR^d)$ is called the space of tempered ultra-distributions. \glsadd{Sww}
\end{definition}
\begin{remark}
We here follow \cite[Def. 6.1.2.3]{Triebel} and equip the dual $\Sww'(\RR^d)$ with the strong topology. The choice of the weak-* topology is however also common in the literature \cite{Bahouri}.   
\end{remark}
\begin{remark}
\label{rem:Sigmageq1}
The reason why we excluded the case $\sigma\geq 1$ for $\wexps$ in Definition \ref{def:Weights} is that we want $\Sww$ to contain functions with compact support, which then allows for localization and thus for a Littlewood-Paley theory. But if $\omega=\wexps$ with $\sigma\geq 1$ and $f\in \Sww(\RR^d)$ the requirement $\pi^\omega_{0,\lambda}(f)<\infty$ implies that $\FFr f$ can be bounded by
$e^{-c|x|}$, $c>0$ , which means that $f$ is analytic and the only compactly supported $f\in \Sww(\RR^d)$ is the zero-function $f=0$.
\end{remark}
In the case $\omega=\wexps,\sigma\in (0,1)$ the space $\Sw_\omega'$ is strictly larger than $\Sw'$. Indeed: $e^{c|\cdot|^{\sigma'}}\in\Sw_\omega'(\RR^d)\backslash\Sw'(\RR^d)$ for $\sigma'\in (0,\sigma]$. In the case $\omega=\wpol$ we simply have
\begin{align*}
\Sw_\omega(\RR^d)=\Sw(\RR^d)
\end{align*}
with a topology that can also be generated by only using the seminorms $p^{\omega}_{\alpha,\lambda}$ so that the dual of $\Sw_\omega(\RR^d)=\Sw(\RR^d)$ is given by  
\begin{align*}
\Sw'_\omega(\RR^d)=\Sw'(\RR^d)\,.
\end{align*}
The theory of ``classical'' tempered distributions is therefore contained in the framework above.  

The role of the triple 
\begin{align*}
\Dw(\RR^d):=C^\infty_c(\RR^d)\subseteq \Sw(\RR^d)\subseteq C^\infty(\RR^d)
\end{align*} 
in this theory will be substituted by spaces $\Dww(\RR^d),\,C^\infty_\omega(\RR^d)$ such that 
\begin{align*}
\Dww(\RR^d)\subseteq \Sw_\omega(\RR^d)\subseteq \Eww(\RR^d)\,.
\end{align*} 
\begin{definition}
\label{def:Gevrey}
Let $U\subseteq \RR^d$ be an open set and $\omega\in \ww=\{\wpol\}\cup \{\wexps\,\vert\,\sigma\in (0,1)\}$. We define for $\omega=\wexps$ the set $\Eww(U)$ to be the space of $f\in C^\infty(U)$ such that for every $\eps>0$ and compact $K\subseteq U$ there exists $C_{\eps,K}>0$ such that for all $\alpha\in \NN^d$
\begin{align}
\label{eq:BeurlingUsefullCondition}
 \sup_{K} |\partial^\alpha f|\leq C_{\eps,K}\, \eps^{|\alpha|} (\alpha!)^{1/\sigma}\,.
 \end{align} 
For $\omega=\wpol$ we set $C^\infty_\omega(U)=C^\infty(U)$. We also define 
 \begin{align}
 \label{eq:CompactBeurling}
 \Dww(U)=C^\infty_\omega(U)\cap C^\infty_c(U)\,.
 \end{align}
 The elements of $C^\infty_\omega(U)$ are called \emph{ultra-differentiable functions} and the elements of the dual space $\Dww'(\RR^d)$ are called \emph{ultra-distributions}. \glsadd{Eww}
\end{definition}

\begin{remark}
The space $\Dww'(\RR^d)$ is equipped with a suitable topology \cite[Section 1.6]{Bjoerck} which we did not specify since this space will not be used in this article and is just mentioned for the sake of completeness. 
\end{remark}
\begin{remark}
The factor $\alpha!$ in \eqref{eq:BeurlingUsefullCondition} can be replaced by $|\alpha|!$ or $|\alpha|^{|\alpha|}$ \cite[Proposition 1.4.2]{Rodino} as can be easily seen from $\alpha!\leq |\alpha|!\leq d^{|\alpha|} \alpha!$ and Stirlings formula. 
\end{remark}

The relation between $\Dww,\Sww,\Eww$ and their properties are specified by the following lemma. 
\begin{lemma} 
Let $\omega\in \ww $.
\begin{itemize}
	\item[i)] We have $\Sww(\RR^d)\subseteq \Eww(\RR^d)$ and 
\begin{align}
\label{eq:DwwViaSww}
\Dww(\RR^d)=\Sww(\RR^d)\cap C^\infty_c(\RR^d)\,. 
\end{align}
In particular $\Dww(\RR^d)\subseteq \Sww(\RR^d)\subseteq  C^\infty_c(\RR^d)$. 

	\item[ii)] The space $\Sww(\RR^d)$ is stable under addition, multiplication and convolution.

	\item[iii)] The space $\Eww(\RR^d)$  is stable under addition, multiplication and division in the sense that $f/g \cdot \mathbf{1}_{\supp f}\in \Eww(\RR^d)$ for $f,g \in \Eww(\RR^d),\,\supp f \subseteq \overset{\circ}{\supp g}$.
\end{itemize}
\end{lemma}
\begin{proof}[Sketch of the proof]
We only have to prove the statements for $\omega\in \{\wexps\,\vert\, \sigma\in (0,1)\}$.
Take $f\in \Sww(\RR^d)$ and $\eps>0$. We then have for $\alpha\in \NN^d$
\begin{align*}
\partial^\alpha f(x)=(2\pi \ii)^{|\alpha|} \int_{\RR^d}  e^{2\pi \ii x \scl \xi} \,\xi^\alpha\,\FFr f(\xi)\,\dd \xi
\end{align*}
Using further that for $\lambda>0$ (we here follow \cite[Lemma 12.7.4]{HoermanderII})
\begin{align*}
\int |\xi|^{|\alpha|} e^{-\lambda |\xi|^\sigma  } \dd \xi \lesssim \int_0^\infty r^{|\alpha|+d-1} e^{-\lambda r^\sigma} \dd r \lesssim \lambda^{-|\alpha|/\sigma} \Gamma\big((|\alpha|+d)/\sigma\big)\overset{\mbox{Stirling}}{\lesssim} \lambda^{-|\alpha|/\sigma} C^{|\alpha|} |\alpha|^{|\alpha|/\sigma},
\end{align*}
we obtain for $x\in \RR^d$
\begin{align*}
|\partial^\alpha f (x)|\lesssim C_{\lambda}  \lambda^{-|\alpha|/\sigma} C^{|\alpha|} |\alpha|^{|\alpha|/\sigma} \cdot \pi^\omega_{0,\lambda}(f)\,.
\end{align*}
Choosing $\lambda>0$ big enough shows that $f$ satisfies the estimate in \eqref{eq:BeurlingUsefullCondition} (with global bounds) and thus $f\in \Eww(\RR^d)$ and $\Sww(\RR^d)\subseteq \Eww(\RR^d)$. In particular we get $\Sww(\RR^d)\cap C^\infty_c(\RR^d)\subseteq \Dww(\RR^d)$. To show the inverse inclusion consider $f\in \Dww(\RR^d)$. We only have to show that $\pi_{\alpha,\lambda}^\omega(f)<\infty$ for any $\lambda>0$ and $\alpha\in \NN^d$. And indeed for $x\in \RR^d$ with $|x|\geq 1$ (without loss of generality)\footnote{We here follow ideas from  \cite[Proposition A.2]{WeberMourrat}.}
\begin{align*}
 | e^{\lambda |x|^\sigma} \FFr f(x) | &\leq \sum_{k=0}^\infty \frac{\lambda^k}{k!} |x|^{\sigma k} |\FFr f(x)|
\leq \sum_{k=0}^\infty \frac{\lambda^k C^k}{k!} |x|^{\lceil\sigma k \rceil} |\FFr f(x) |\\
&\leq \sum_{i=1}^d\sum_{k=0}^\infty \frac{\lambda^k C^k}{k!} |x_i|^{\lceil\sigma k \rceil} |\FFr f(x) |
=\sum_{i=1}^d\sum_{k=0}^\infty \frac{\lambda^k C^k}{k!} \Big| \int e^{2\pi \ii \xi} \partial^{\lceil \sigma k \rceil e_i} f(\xi)  \,\dd\xi  \Big|\\
&\overset{\eqref{eq:BeurlingUsefullCondition} \,\&\,\scriptsize \mbox{Stirling}}{\leq}  C_{\eps} \sum_{k=0}^\infty \lambda^k C^k \eps^k<\infty 
\end{align*}
where $C, C_\eps>0$ denote as usual constants that may change from line to line and where in the last step we chose $\eps>0$ small enough to make the series converge; note that the bound~\eqref{eq:BeurlingUsefullCondition} holds on all of $\RR^d$ because $f$ is compactly supported by assumption.
 
The stability of $\Sw_\omega(\RR^d)$ under addition, multiplication and convolution are quite easy to check, see \cite[Proposition 1.8.3]{Bjoerck}. 

It is straightforward to check that $f\cdot g\in \Eww(U)$ for $f,g\in \Eww(U)$ using Leibniz's rule. For the stability under composition see e.g. \cite[Proposition 3.1]{RainerSchindl}, from which the stability under division can be easily derived.
\end{proof}

Many linear operations such as addition or derivation that can be defined on distributions can be translated immediately to the space of ultra-distributions $\left(\Dww(\RR^d)\right)'$. We see with \eqref{eq:CompactBeurling} that $\Eww(\RR^d)$ should be interpreted as the set of smooth multipliers for ultra-distributions in $\Dww'(\RR^d)$ and in particular for tempered ultra-distributions $\Sww'(\RR^d)\subseteq \Dww'(\RR^d)$.  The space $\Sw_\omega'(\RR^d)$ is small enough to allow for a Fourier transform. 

\begin{definition}
For $f\in \Sw_\omega'(\RR^d)$ and $\varphi\in\Sw_\omega(\RR^d)$ we set 
\begin{align*}
	\FFr f (\varphi)&:=f(\FFr \varphi), \\
	\FFr^{-1} f(\varphi)&:=f(\FFr^{-1} \varphi).
\end{align*}
By definition of $\Sw_\omega(\RR^d)$ we have that $\FFr$ and $\FFr^{-1}$ are isomorphisms on $\Sw_{\omega}(\RR^d)$ which implies that $\FFr$ and $\FFr^{-1}$ are isomorphisms on $\Sw_\omega'(\RR^d)$. 
\end{definition}

The following lemma proves that the set of compactly supported ultra-differentiable functions $\Dww(\RR^d)$ is rich enough to localize ultra-distributions, which gets the Littlewood-Paley theory started and allows us to introduce Besov spaces based on ultra-distributions in the next section. 
 
\begin{lemma}[\cite{Bjoerck}, Theorem 1.3.7.]
\label{lem:ResolutionOfUnity}
Let $\omega\in \ww$. For every pair of compact sets $K\subsetneq K'\subseteq \RR^d$ there is a $\varphi\in \Dww(\RR^d)$ such that 
\[
	\varphi\vert_{K} =1\,, \qquad \supp \varphi \subseteq K'\,.
\]
\end{lemma}

\subsection{Ultra-distributions on Bravais lattices}\label{subsec:LatticeUltraDistributions}

For the discrete setup we essentially proceed as in Subsection \ref{subsec:LatticeFourierTransform} and define spaces
\begin{align*}
	\mathcal{S}_\omega (\GG)=\left\{ f:\GG\rightarrow \CC\,\middle\vert\,\, \sup_{k\in \GG} e^{\lambda \omega(k)} |f(k)|<\infty \mbox{ for all } \lambda>0\right\},
\end{align*}
and their duals (when equipped with the natural topology)
\begin{align*}
	\mathcal{S}_\omega'(\GG)=\left\{ f:\GG\rightarrow \CC\,\middle\vert\,\, \sup_{k\in \GG} e^{-\lambda \omega(k)} |f(k)|<\infty \mbox{ for some } \lambda>0\right\},
\end{align*}
with the pairing $f(\varphi)=|\GG|\sum_{k\in \GG} f(k) \varphi(k)$, $\varphi \in \mathcal{S}_\omega(\GG)$. As in Subsection \ref{subsec:LatticeFourierTransform} we can then define a Fourier transform $\FFg$ on $\mathcal{S}_\omega'(\GG)$ which maps the discrete space $\mathcal{S}_\omega(\GG)$ into the space of ultra-differentiable functions $S_\omega(\widehat{\GG}):=C^\infty_\omega(\widehat{\GG})$ with periodic boundary conditions. The dual space $\Sw_\omega'(\widehat{\GG})$ can be equipped with a Fourier transform $\FFg^{-1}$ as in \eqref{eq:InverseFourierTransformLattice} such that $\FFg,\FFg^{-1}$ become isomorphisms between $\mathcal{S}_\omega '(\GG)$ and  $\Sw_\omega'(\widehat{\GG})$ that are inverse to each other. For a proof of these statements we refer to Lemma \ref{lem:FourierAnalysisUltraDistributions}.

Performing identifications as in the case of $\Sw'(\RR^d)$ we can interpret these concepts as a sub-theory of the Fourier analysis on $\Sw_\omega'(\RR^d)$ with the only difference that we have to choose the function $\psi$, satisfying $\sum_{k\in \mathscr{R}} \psi(\cdot -k)=1$, on page \pageref{IdentificationsFourierTheory} as an element of $\Dw_\omega(\RR^d)$, see page~\pageref{subsec:ExtensionOperator} below for details.

\subsection{Discrete weighted Besov spaces}
\label{subsec:discreteBesovspaces}

We can now give our definition of a discrete, weighted Besov space, where we essentially proceed as in Subsection~\ref{subsec:LatticeFourierTransform} with the only difference that $\rho\in \boldsymbol{\rho}(\omega)$ is included in the definition and that the partition of unity $(\varphi_j)_{j\geq -1}$, from which $(\varphi_j^{\GG})_{j\geq -1}$ is constructed as on page \pageref{DiscreteDyadicPartition}, must now be chosen in $\Dww(\RR^d)$. 

\begin{definition}
\label{def:DiscreteBesov}
Given a Bravais lattice $\GG$, parameters $\alpha\in \RR$, $p,q\in[1,\infty]$ and a weight $\rho \in \rr(\omega)$ for $\omega\in \ww$ we define 
\begin{align*}
\mathcal{B}^\alpha_{p,q}(\GG,\rho):=\left\{ f\in \Sw_\omega'(\GG) \,\vert \, \|f\|_{\mathcal{B}^\alpha_{p,q}(\GG,\rho)}: = \|( 2^{j\alpha } \|\rho   \cdot  \varDelta_j^\GG f\|_{L^p(\GG)})_{j=-1,\ldots,j_{\GG}}\|_{\ell^q}<\infty \right\}\,,
\end{align*}
where the Littlewood-Paley blocks $(\varDelta_j^\GG)_{j=-1,\ldots,j_{\GG}}$ are built from a dyadic partition of unity $(\varphi_j^\GG)_{j=-1,\ldots,j_{\GG}}\subseteq \Eww(\widehat{\GG})$ on $\wGG$ constructed from some dyadic partition of unity $(\varphi_j)_{j\geq -1}\subseteq \Dww(\RR^d)$ on $\RR^d$ as on page \pageref{DiscreteDyadicPartition}. If we consider a sequence $\GG^\eps$ as in Definition \ref{def:LatticeSequence} we take \emph{the same} $(\varphi_j)_{j\geq -1} \subseteq \Dww(\RR^d)$ to construct for all $\eps$ the partitions $(\varphi_j^{\GGe})_{j=-1,\ldots, j_{\GGe}}$ on $\wGGe$.

We write furthermore $\mathcal{C}^\alpha_p(\GG,\rho)=\mathcal{B}^\alpha_{p,\infty}(\GG,\rho)$ and define 
\begin{align*}
	L^p(\GG,\rho):=\{f\in \Sw_\omega(\GG) \,\vert\, \|f\|_{L^p(\GG,\rho)}:=\|\rho f\|_{L^p(\GG)}<\infty \}\,,
\end{align*}
i.e. $\|f\|_{\mathcal{B}^\alpha_{p,q}(\GG,\rho)} = \| (2^{j\alpha } \|\varDelta_j^\GG f\|_{L^p(\GG,\rho)})_j\|_{\ell^q}$. \glsadd{Balphapq} \glsadd{mathcalCalphap}
\end{definition}
\begin{remark}
   When we introduce the weight we have a choice where to put it. Here we set $\|f\|_{L^p(\GG,\rho)}=\|\rho f\|_{L^p(\GG)}$, which is analogous to~\cite{Triebel} or~\cite{HairerLabbeR2}, but different from~\cite{WeberMourrat} who instead take the $L^p$ norm under the measure $\rho(x) \dd x$. For $p=1$ both definitions coincide, but for $p=\infty$ the weighted $L^\infty$ space of Mourrat and Weber does not feel the weight at all and it coincides with its unweighted counterpart.
\end{remark}
\begin{remark}
\label{rem:ContinousSpaces}
The formulation of this definition for continuous spaces $\mathcal{B}^\alpha_{p,q}(\RR^d,\rho)$, $\mathcal{C}^\alpha_p(\RR^d,\rho)$ and $L^p(\RR^d,\rho)$ is analogous. 
\end{remark}

We can write the Littlewood-Paley blocks as convolutions (on $\GG$):
\begin{align}
\label{eq:BlockAsConvolution}
	\varDelta_j^{\GG} f(x)=\kkg \ast_{\GG} f(x)= |\GG|\,\sum_{k\in \GG} \kkg(x -k) f(k)\,, \qquad x \in \GG\,,
\end{align}
where
\glsadd{PsiGGj}
\[
	\kkg:=\FFg^{-1} \varphi^{\GG}_j.
\]
We also introduce the notation
\glsadd{PsiGGleqj}
\begin{align*}
\Psi^{\GG,<j}:=\sum_{i<j} \Psi^{\GG,j}\,.
\end{align*}
Due to our convention to only consider dyadic scalings we always have the useful property
\begin{align}
\label{eq:PsijScaling}
\kkge=2^{jd}\phi_{\je}(2^j\cdot)
\end{align}
for a lattice sequence $\GG^\eps$ as in Definition~\ref{def:LatticeSequence}, where 
\begin{align}
\label{eq:je}
\je=\begin{cases}
-1, & j=-1, \\
0, & -1<j<j_{\GG^\eps}, \\
\infty, & j=j_{\GG^\eps},
\end{cases}
\end{align}
and where $\phi_{-1},\,\phi_0,\,\phi_{\infty}\in \Sw(\RR^d)$ are Schwartz functions \emph{on} $\RR^d$ with $\FFr \phi_{\je}\in \Dww(\RR^d)$. The functions $\phi_{-1},\,\phi_0,\,\phi_{\infty}$ depend on the lattice $\GG$ used to construct $\GG^\eps=\eps \GG$ but are independent of $\eps$. In a way, this is a discrete substitute for the scaling one finds on $\RR^d$ for $\Psi^j:=\FFr^{-1} \varphi_j =2^{jd} (\FFr^{-1} \varphi_0)(2^j\cdot)$ (for $j\geq 0$) due to the choice of the dyadic partition of unity in Definition \ref{def:DyadicPartitionOfUnity}. We prove the identity \eqref{eq:PsijScaling}, together with a similar result for $\Psi^{\GG,<j}$, in Lemma \ref{lem:ScalingBoundary} below. It turns out that \eqref{eq:PsijScaling} is helpful in translating arguments from the continuous theory into our discrete framework.
Let us once more stress the fact that $\phi_{\je}$ is defined on all of $\RR^d$, and therefore \eqref{eq:BlockAsConvolution} actually makes sense for all $x\in\RR^d$. With the $\phi_{\je}$ from Lemma \ref{lem:ScalingBoundary} this ``extension'' coincides with $\EE^\eps (\Delta^\GG_j f)$, where the extension operator $\EE^\eps$ is defined as in Lemma \ref{lem:ExtensionOperator} below. 

The following Lemma, a discrete weighted Young inequality, allows us to handle convolutions such as \eqref{eq:BlockAsConvolution}.
\begin{lemma}
\label{lem:DiscreteWeightedYoungInequality}
Given $\GG^\eps$ as in Definition~\ref{def:LatticeSequence} and $\Phi\in \Sw_\omega(\RR^d)$ for $\omega\in \ww$ we have for any $\delta\in (0,1]$ with $\delta\gtrsim \eps$ and $p\in [1,\infty]$, $\lambda>0$ for $\Phi^\delta:=\delta^{-d} \Phi(\delta^{-1}\cdot)$ the bound
\begin{align}
\label{eq:KjUniformlyIntegrableContinuousVersion}
	\sup_{x\in \RR^d} \| \Phi^\delta(\cdot+x)\|_{L^p(\GG^\eps,e^{\lambda\omega(\cdot+x)})}\lesssim \delta^{-d(1-1/p)}  \,.
\end{align}
where the implicit constant is independent of $\eps>0$. In particular, $\| \Phi^\delta\|_{L^p(\GG^\eps,e^{\lambda\omega})}\lesssim \delta^{-d(1-1/p)}$ and for $\rho\in \boldsymbol{\rho}(\omega)$
\begin{align}
\label{eq:DiscreteWeightedYoungInequality}
	\|\Phi^\delta\aste f\|_{L^p(\GG^\eps,\rho)}\lesssim \|f\|_{L^p(\GG^\eps,\rho)},\qquad \|\Phi^\delta \aste f\|_{L^p(\RR^d,\rho)}\lesssim \|f\|_{L^p(\GG^\eps,\rho)}\,,
\end{align} 
where we used in the second estimate that 
\begin{align*}
x\mapsto (\Phi^\delta \aste f)(x)= |\GG^\eps|\sum_{k\in \GG^\eps} \Phi^\delta(x-k)\,f(k)
\end{align*}
can be canonically extended to $\RR^d$. 
\end{lemma}
\begin{remark}
\label{rem:ScalingPsij}
Using $\delta=2^{-j}$ for $j\in\{-1,\ldots,j_{\GG^\eps}\}$ this covers in particular the functions $\Psi^{\GG^\eps,j}=\FFge^{-1} \varphi_j^{\GG^\eps}$ via \eqref{eq:PsijScaling}.  
\end{remark}
\begin{proof}
The case $p = \infty$ follows from the definition of $\Sw_\omega(\RR^d)$ and $e^{\lambda \omega(k)}\leq e^{\lambda \omega(\delta^{-1} k)}$, so that we only have to show the statement for $p<\infty$.  And indeed we obtain 
\begin{align*}
			\|\Phi^\delta\|^p_{L^p(\GG^\eps,e^{\lambda\omega})} &=\sum_{k\in\GG^\eps} |\GG^\eps| |\Phi^\delta(k)|^p e^{ p\lambda\omega(k)} = \delta^{- d p}\eps^d  \sum_{k\in\GG} |\GG| |\Phi(\delta^{-1} \eps k)|^p e^{ p\lambda \omega(\eps k)} \\
			& \leq \delta^{- d p}\eps^{d } \sum_{k\in\GG} |\GG| |\Phi(\delta^{-1} \eps k)|^p e^{p\lambda\omega(\delta^{-1} \eps k)} \lesssim \delta^{-d(p-1)}  \sum_{k\in\GG} |\GG| \delta^{ -d}\eps^d  \frac{1}{1+|\delta^{-1} \eps k|^{d+1}} \\
			&\overset{\mbox{ \footnotesize Lemma~\ref{lem:IntegralTest}}}{\lesssim} \delta^{-d(p-1)} \int_{\RR^d} \dd z\, (\delta^{-1} \eps)^d \frac{1}{1+|\delta^{-1} \eps z|^{d+1}} \lesssim \delta^{-d(p-1)}\,,
	\end{align*}
where we used that $\Phi\in \Sw_\omega(\RR^d)$ and in the application of Lemma \ref{lem:IntegralTest} that for $|x-y|\lesssim 1 $ the quotient $\frac{1+|\delta^{-1} \eps x|}{1+|\delta^{-1} \eps y|}$ is uniformly bounded.
Inequality \eqref{eq:KjUniformlyIntegrableContinuousVersion} can be proved in the same way since it suffices to take the supremum over $|x|\lesssim \eps$. 

The estimates for $\Phi^\delta\aste f $ then follow by Young's inequality on $\GG^\eps$ and a mixed Young inequality, Lemma \ref{lem:MixedYoungInequality} below, applied to the right hand side of
\begin{align*}
 \rho(x)\,|\Phi^\delta\aste f (x)| &\leq \sum_{k\in \GGe} |\GGe|\,\rho(x)|\Phi^\delta(x - k)|\cdot |f(k)| \\
  &\overset{(\star)}{\lesssim} \sum_{k\in \GGe} |\GGe| \,e^{\lambda \omega(x-k) }\, |\Phi^\delta(x-k)|\cdot \rho(k)|f(k)|= |e^{\lambda \omega} \Phi| \aste |\rho f| (x)\,. 
 \end{align*}
In the step $(\star)$ we used that $\rho(x)\lesssim e^{\lambda \omega(x-k) }\, \rho(k)$ for some $\lambda>0$ due to \eqref{eq:WeightQuotient}. 
\end{proof}

From Lemma \ref{lem:DiscreteWeightedYoungInequality} ( and Remark \ref{rem:ScalingPsij}) we see in particular that the blocks $\varDelta_j^{\GG^\eps}$ map the space $L^p(\GG^\eps,\rho)$ into itself for any $p\in [1,\infty]$:
\begin{align}
\label{eq:LpLp}
\|\varDelta_j^{\GGe} f\|_{L^p(\GGe,\rho)}=\| \Psi^{\GGe,j}\aste f  \|_{L^p(\GGe,\rho)} \overset{\mbox{\tiny Lemma \ref{lem:DiscreteWeightedYoungInequality}}}{\lesssim} \|f\|_{L^p(\GGe,\rho)} \,,
\end{align}
where the involved constant is independent of $\eps$ and $j=-1,\ldots,j_{\GGe}$. This is the discrete analogue of the continuous version
\begin{align}
\label{eq:LpLpRRd}
\|\varDelta_j f\|_{L^p(\RR^d,\rho)} \lesssim \|f\|_{L^p(\RR^d,\rho)}
\end{align}
for $j\geq -1$ (which can be proved in essentially the same manner).

As in the continuous case we can state an embedding theorem for discrete Besov spaces. Since it can be shown exactly as its continuous (and unweighted) cousin (\cite[Proposition 2.71]{Bahouri} or \cite[Theorem 4.2.3]{EdmundsTriebel}) we will not give its proof here. 

\begin{lemma}
\label{lem:DiscreteEmbedding}
Given $\GG^\eps$ as in Definition \ref{def:LatticeSequence} for any $\alpha_1 \in \RR$, $1\leq p_1 \leq p_2 \leq \infty ,\, 1\leq q_1\leq q_2 \leq \infty$ and weights $\rho_1,\rho_2$ with $\rho_2\lesssim \rho_1$ we have the continuous embedding (with norm of the embedding operator independent of $\eps\in (0,1]$) 
\begin{align*}
	\mathcal{B}^{\alpha_1}_{p_1,q_1}(\GG^\eps,\rho_1) \subseteq \mathcal{B}^{\alpha_2}_{p_2,q_2}(\GG^\eps,\rho_2)
\end{align*}
for $\alpha_2-\frac{d}{p_2}\leq \alpha_1-\frac{d}{p_1}$. If $\alpha_2< \alpha_1-d(1/p_1-1/p_2)$ and $\lim_{|x|\rightarrow \infty} \rho_2(x)/\rho_1(x)=0$ the embedding is compact.
\end{lemma}

For later purposes we also recall the continuous version of this embedding.

\begin{lemma}[\cite{EdmundsTriebel}, Theorem 4.2.3]
\label{lem:ContinuousEmbedding}

For any $\alpha_1 \in \RR$, $1\leq p_1 \leq p_2 \leq \infty ,\, 1\leq q_1\leq q_2 \leq \infty$ and weights $\rho_1,\rho_2$ with $\rho_2\lesssim \rho_1$ we have the continuous embedding (with norm independent of $\varepsilon\in (0,1]$) 
\begin{align*}
	\mathcal{B}^{\alpha_1}_{p_1,q_1}(\RR^d,\rho_1) \subseteq \mathcal{B}^{\alpha_2}_{p_2,q_2}(\RR^d,\rho_2)
\end{align*}
for $\alpha_2\leq \alpha_1-d(1/p_1-1/p_2)$. If $\alpha_2< \alpha_1-d(1/p_1-1/p_2)$ and $\lim_{|x|\rightarrow \infty} \rho_2(x)/\rho_1(x)=0$ the embedding is compact. 
\end{lemma}

\subsubsection*{The extension operator}
\label{subsec:ExtensionOperator}
Given a Bravais lattice $\GG$ and a dyadic partition of unity $(\varphi_j)_{j\geq -1}$ on $\RR^d$ such that $j_{\GG}$, as defined on page \pageref{DiscreteDyadicPartition}, is strictly greater than $0$ we construct a discrete dyadic partition of unity $(\varphi_j^{\GG})_{-1,\ldots,j_\GG}$ from $(\varphi_j)_{j\geq -1}$ as on page \pageref{DiscreteDyadicPartition}.

We choose a symmetric function $\psi\in \Dww(\RR^d)$ which we refer to as the \textit{smear function} and which satisfies the following properties:
\begin{enumerate}
\item $\sum_{k \in \mathscr{R}} \psi(\cdot-k)=1$,
\item $\psi=1$ on $\supp \varphi_j$ for $j<j_{\GG}$,
\item $
\left(\supp \psi\cap \supp \ext{\varphi^{\GG}_j}\right)\backslash \wGG \neq \varnothing \,\Rightarrow j=j_{\GG}\,.
$
\end{enumerate}
\glsadd{psi}
The last property looks slightly technical, but actually only states that the support of $\psi$ is small enough such that it only touches the support of the periodically extended $\varphi_j^{\GG}$ with $j<j_{\GG}$ \emph{inside} $\wGG$. Using $\dist(\partial \wGG,\bigcup_{j<j_{\GG}}\supp (\varphi^\GG_{j})_{\mathrm{ext}})>0$ it is not hard to construct a function $\psi$ as above: Indeed choose via Lemma \ref{lem:ResolutionOfUnity} some $\tilde{\psi}\in \Dww(\RR^d)$ that satisfies property 3 and $\tilde{\psi}\vert_{\wGG}=1$ and set $\psi:=\tilde{\psi}/\sum_{k\in \CR} \tilde{\psi}(\cdot-k)$.

The rescaled $\psi^\eps := \psi(\eps\cdot)$ satisfies the same properties on $\GG^\eps$ (remember that by convention we construct the sequence $(\varphi_j^{\GG^\eps})_{j=-1,\ldots,j_{\GG^\eps}}$ from the same $(\varphi_j)_{j\geq -1}$). 
This allows us to define an extension operator $\EE^\eps$ in the spirit of Lemma \ref{lem:ClassicalCaseExtensionOperator} as
\begin{align*}
	\EE^\eps f:= \mathcal{F}^{-1}_{\RR^d}(\psi^\eps \cdot  \ext{\FFge f}),\qquad f\in \Sw_\omega'(\GG^\eps),
\end{align*}
\glsadd{mathscrEeps}
and as in Lemma~\ref{lem:ClassicalCaseExtensionOperator} we can show that $\EE^\eps f\in \Eww(\RR^d)\cap\Sw_\omega'(\RR^d)$ and $\EE^\eps f\vert_{\GG^\eps}=f$. 

Using \eqref{eq:DiracComb} we can give a useful, alternative formulation of $\EE^\eps f$
\begin{align}
\EE^\eps  f &=\FFr^{-1}\psi^\eps\rast \FFr^{-1}\ext{\FFge f} =\FFr^{-1}\psi^\eps \rast f_{\mathrm{dir}} \nonumber\\
 &=\FFr^{-1}\psi^\eps \aste f=|\GG^\eps|\sum_{z\in \GG^\eps} \FFr^{-1}\psi^\eps(\cdot-z)\,\, f(z)\,,
\label{eq:ExtensionAsConvolution}
\end{align}
where as in \eqref{eq:BlockAsConvolution} we read  the convolution in the second line as a function on $\RR^d$ using that $\FFr^{-1}\psi^\eps\in \Sww(\RR^d)$ is defined on $\RR^d$.  By property 3 of $\psi$ we also have for $j<j_{\GG^\eps}$
\begin{align}
\label{eq:ExtensionLowBlocks}
\varDelta_j \EE^\eps f=\EE^\eps \varDelta_j^{\GG^\eps} f
\end{align}
Finally, let us study the interplay of $\EE^\eps$ with Besov spaces.

\begin{lemma}
\label{lem:ExtensionOperator}
For any $\alpha\in\RR,\,p,q\in [1,\infty]$ and $\rho\in \rr(\ww) $  the family of operators
\begin{align*}
 \EE^\eps \colon \mathcal{B}^{\alpha}_{p,q}(\GG^\eps, \rho )\longrightarrow \mathcal{B}^{\alpha}_{p,q}(\RR^d,\rho)\,,
\end{align*} 
defined above, is uniformly bounded in $\eps$.
\end{lemma}
\begin{proof}
We have to estimate $\varDelta_j \EE^\eps f$ for $j\geq -1$. For $j<j_{\GG^\eps}$ we can apply \eqref{eq:ExtensionLowBlocks} and \eqref{eq:ExtensionAsConvolution} together with Lemma  \ref{lem:DiscreteWeightedYoungInequality} to bound 
\begin{align*}
\|\varDelta_j \EE^\eps f\|_{L^p(\RR^d,\rho)}=\|\eps^{-d}(\FFr\psi)(\eps^{-1}\cdot)\aste \varDelta_j^{\GG^\eps} f\|_{L^p(\RR^d,\rho)}\lesssim \|\varDelta_j^{\GG^\eps} f\|_{L^p(\GG^\eps,\rho)}\lesssim 2^{-j\alpha} \|f\|_{\mathcal{B}^{\alpha}_{p,q}(\GG^\eps, \rho )}
\end{align*}
For $j\geq j_{\GG^\eps}$ only $j\sim  j_{\GG^\eps}$ contributes due to the compact support of $\psi^\eps$. By spectral support properties we have
\begin{align*}
  \varDelta_j \EE^\eps f=\varDelta_j (\EE^\eps \sum_{i\sim j_{\GG^\eps}} \varDelta_i^{\GG^\eps} f)
\end{align*}  
  From \eqref{eq:LpLpRRd} we know that $\varDelta_j$ maps $L^p(\RR^d,\rho)$ into itself and we thus obtain
\begin{align*} 
   \|\varDelta_j \EE^\eps f\|_{L^p(\RR^d,\rho)}\lesssim \|\EE^\eps \sum_{i\sim j_{\GG^\eps}} \varDelta_{i}^{\GG^\eps} f\|_{L^p(\GG^\eps,\rho)}\lesssim 2^{-j_{\GG^\eps}\alpha} \|f\|_{\mathcal{B}^{\alpha}_{p,q}(\GG^\eps,\rho)}\,,
 \end{align*} 
where we applied once more \eqref{eq:ExtensionAsConvolution} and Lemma \ref{lem:DiscreteWeightedYoungInequality} in the second step. 
\end{proof}

Below, we will often be given some functional $F(f_1,\ldots,f_n)$ on discrete Besov functions taking values in a discrete Besov space $X$ (or some space constructed from it) that satisfies a bound of the type
\begin{align}
\label{eq:EstimatePropertyE}
\|F(f_1,\ldots,f_n)\|_X \leq c(f_1,\ldots,f_n).
\end{align}
We then say that the estimate \eqref{eq:EstimatePropertyE} has the property $(\EE)$ (on $X$) if there is a ``continuous version'' $\overline{F}$ of $F$ and a continuous version $\overline{X}$ of $X$ and a sequence of constants $o_\eps\rightarrow 0$ 
such that 

\begin{align}
\label{eq:EProperty}
\tag{$\EE$}
\|\EE^\eps F(f_1,\ldots,f_n)-\overline{F}(\EE^\eps f_1,\ldots,\EE^\eps f_n )\|_{\overline X} \leq o_\eps \cdot c(f_1,\ldots,f_n)\,.
\end{align} 
In other words we can pull the operator $\EE^\eps$ inside $F$ without paying anything in the limit. 
With the smear function $\psi$ introduced above when can now also give the proof of the announced scaling property \eqref{eq:PsijScaling} of the functions $\Psi^{\GG^\eps,j}$. 

\begin{lemma}
\label{lem:ScalingBoundary}
Let $\GG^\eps$ be as in Definition \ref{def:LatticeSequence} and let $\omega\in \ww$. Let $(\varphi^{\GG^\eps}_j)_{j=-1,\ldots,j_{\GG^\eps}} \subseteq \Dww(\wGGe)$ be a partition of unity of $\widehat{\GG^\eps}$ as defined on page \pageref{DiscreteDyadicPartition} and take $\Psi^{\GG^\eps,j}=\FFge^{-1}\varphi^{\GG^\eps}_j$ and $\Psi^{\GG^\eps,<j}:=\sum_{i<j} \Psi^{\GG^\eps,i}$. The extensions
\begin{align*}
 \tilde{\Psi}^{\eps,j} &:=\EE^{\eps}\Psi^{\GG^\eps,j}=\FF^{-1}_{\RR^d} (\psi^\eps\cdot \ext{\varphi^{\GG^\eps}_j}) \\
  \tilde{\Psi}^{\eps,<j} &:=\EE^{\eps}\Psi^{\GG^\eps,<j}=\FF^{-1}_{\RR^d} \Big(\psi^\eps\cdot \Big(\sum_{i<j}\varphi^{\GG^\eps}_i\Big)_{\mathrm{ext}}\Big) 
\end{align*}
are elements of $\Sww(\RR^d)$. Moreover there are $\check{\phi}_{-1},\check{\phi}_0,\check{\phi}_\infty,\,\check{\phi}_{\Sigma}\in \Dww(\RR^d)$, independent of $\eps$, such that for for $j=-1,\ldots,j_{\GG^\eps}$ and $j'=0,\ldots,j_{\GG^\eps}$ with $\je$ as in \eqref{eq:je}
\begin{align}
\label{eq:ScalingBoundary1}
\psi^\eps\cdot \ext{\varphi^{\GG^\eps}_j}=\check{\phi}_{\je}(2^{-j}\cdot) \,, \\
\label{eq:ScalingBoundary2}
\psi^\eps\cdot \Big(\sum_{i<j'}\varphi^{\GG^\eps}_i\Big)_{\mathrm{ext}}=\check{\phi}_{\Sigma}(2^{-j'}\cdot)\,.
\end{align}
The functions $\check{\phi}_0$ and $\check{\phi}_\infty$ have support in an annulus $\rA\subseteq \RR^d$.

In particular we have for $j=-1,\ldots,j_{\GG^\eps}$ and  $j'=0,\ldots,j_{\GG^\eps}$. 
\begin{align*}
\tilde{\Psi}^{\eps,j}=2^{jd} \cdot \phi_{\je}(2^j \cdot)\,,\qquad \tilde{\Psi}^{\eps,<j'}=2^{j'd} \cdot \phi_{\Sigma}(2^{j'} \cdot)
\end{align*}
where $\phi_i:=\FF^{-1}_{\RR^d} \check{\phi}_i$ for $i\in \{-1,0,\infty,\,\Sigma\}$.
\end{lemma}

\begin{proof}
Denote by $(\varphi_j)_{j \geq -1}\subseteq \Dww(\RR^d)$ the partition of unity on $\RR^d$ from which the partitions $(\varphi_j^{\GGe})_{j=-1,\ldots,j_{\GGe}}$ are constructed. Let us recall the following facts about $(\varphi_j)_{j\geq -1}$
\begin{align}
\label{eq:ScalingVarphij}
\varphi_j &=\varphi_0(2^{-j}\cdot) \mbox{\qquad for $j\geq 0$,} \\
\sum_{i<j'} \varphi_i &= \varphi_{-1}(2^{-j'}\cdot) \mbox{\qquad for $j'\geq 0$.}
\label{eq:ScalingSumVarphij}
\end{align}
The second property can be seen by rewriting 
\begin{align*}
\sum_{i<j'} \varphi_i =1-\sum_{l\geq j'} \varphi_0(2^{-l}\cdot)=1-\sum_{l'\geq 0} \varphi_0( 2^{-(j'+l')}\cdot)=\Big( 1-\sum_{l'\geq 0} \varphi_{l'}\Big)(2^{-j'}\cdot )=\varphi_{-1}(2^{-j'}\cdot)\,.
\end{align*}
Recall further that $\varphi_0$ has support in an annulus around $0$.

To prove the claim we only have to show \eqref{eq:ScalingBoundary1} and \eqref{eq:ScalingBoundary2}. For $j<j_{\GG^\eps}$ and $0\leq j'\leq j_{\GG^\eps}$ we use that by construction of $\varphi^{\GG^\eps}_j$ out of  $(\varphi_j)_{j\geq -1}$ we have \emph{inside} $\wGGe$
\begin{align*}
\varphi^{\GG^\eps}_j=\varphi_j\,,\qquad \sum_{i<j'}\varphi^{\GG^\eps}_i=\sum_{i<j'}\varphi_i
\end{align*}
so that due to property 2 and 3 of the smear function $\psi^\eps$ and \eqref{eq:ScalingSumVarphij} it is enough to take 
\begin{align*}
\check{\phi}_{\Sigma}=\varphi_{-1}
\end{align*}
and for $j<j_{\GG^\eps}$ by the scaling property of $\varphi_j$ from \eqref{eq:ScalingVarphij}
\begin{align*}
\check{\phi}_{\je}:=\varphi_j(2^{j}\cdot)\in \left\{\varphi_{-1}(\cdot/2),\,\varphi_0\right\}\,.
\end{align*}
For the construction of $\phi_\infty$ a bit more work is required. 
Recall that by definition of our lattice sequence $\GG^\eps$ we took a dyadic scaling $\eps=2^{-N}$ which implies in particular 
\begin{align}
\label{eq:eps2j}
2^{-j_{\GG^\eps}}=\eps\cdot 2^k
\end{align}
for some fixed $k\in \ZZ$. Using once more \eqref{eq:ScalingSumVarphij} and relation \eqref{eq:eps2j} we can write for $x\in \wGGe$
\begin{align*}
\varphi^{\GG^\eps}_{j_{\GG^\eps}}(x)=1-\sum_{j<j_{\GG^\eps}} \varphi_j(x)=1-\varphi_{-1}(2^{-j_{\GG^\eps}}x)=\chi(\eps x)
\end{align*}
for some symmetric function $\chi\in \Eww(\RR^d)$. As in \eqref{eq:PeriodicExtension2} let us denote for $x\in \RR^d$ by $[x]_{\wGGe} \in \widehat{\GG}^\varepsilon$ the unique element of $\widehat{\GG}^\varepsilon$ for which $x-[x]_{\wGGe}\in \CR^\eps$. One then easily checks
\begin{align}
\label{eq:ScalingPeriodicity}
 \eps [x]_{\wGGe}= [\eps x]_{\wGG}\,.
 \end{align} 
 Applying \eqref{eq:PeriodicExtension2} and \eqref{eq:ScalingPeriodicity} we obtain for $x\in \RR^d$ that the periodic extension
\begin{align*}
\ext{\varphi^{\GG^\eps}_{j_{\GG^\eps}}}(x)=\varphi^{\GG^\eps}_{j_{\GG^\eps}}([x]_{\wGGe})=\chi(\eps [x]_{\wGGe})=\chi( [\eps x]_{\wGG})
\end{align*}
is the $\eps$ scaled version of the smooth, $\CR$-periodic function $\chi([\cdot]_{\wGG})\in \Eww(\wGG)$ (to see that the composition with $[\cdot]_{\wGG}$ does not change the smoothness, note that $\chi$ equals $1$ on a neighborhood of $\partial \wGG$). Consequently
\begin{align*}
\psi(\eps\cdot)\, \ext{\varphi^{\GG^\eps}_{j_{\GG^\eps}}}=\left(\psi \chi([\cdot]_{\wGG})\right)(\eps\cdot)\,,
\end{align*}
so that setting $\check{\phi}_\infty=\left(\psi \chi([\cdot]_{\wGG})\right)(2^{-k}\cdot)$ with $k$ as in \eqref{eq:eps2j} finishes the proof. 
\end{proof}

\section{Discrete diffusion operators}
\label{sec:DiffusionOperators}

Our aim is to analyze differential equations on Bravais lattice that are in a certain sense semilinear and ``parabolic'', i.e. there is a leading order linear difference operator, which here we will always take as the infinitesimal generator of a random walk on our Bravais lattice. In the following we analyze the regularization properties of the corresponding ``heat kernel''.

\subsection{Definitions}

Let us construct a symmetric random walk on a Bravais lattice $\GG^\eps$ with mesh size $\eps$ which can reach every point (our construction follows \cite{LawlerLimic}). First we choose a subset of ``jump directions'' $\{g_1,\ldots,g_l\}\subseteq \GG\backslash \{0\}$ such that $\ZZ g_1 + \ldots + \ZZ g_l =\GG$
and a map $\kappa\colon \{ g_1, \ldots,g_l \} \rightarrow (0,\infty)$. We then take as a rate for the jump from $z\in \GG^\eps$ to $z\pm \eps g_i\in \GG^\eps$ the value $\kappa(g_i)/{2\eps^2}$. In other words the generator of the random walk is
\begin{align}
\label{eq:GeneratorFiniteRangeRandomWalk}
	L^\eps u (y) &= \eps^{-2} \sum_{e\in \{\pm g_i\}} \frac{\kappa(e)}{2} (u(y+\eps e)-u(y)) \,,
\end{align} 
 which converges (for $u \in C^2(\RR^d)$) pointwise to $L u =\frac{1}{2}\sum_{i=1}^{l} \kappa(g_i)\, g_i\scl \nabla^2 u\,g_i$ as $\eps$ tends to 0. In the case $\GG=\ZZ^d$ and $\kappa(e_i)=1/d$ we obtain the simple random walk with limiting generator $L=\frac{1}{2d} \Delta$.  We can reformulate~\eqref{eq:GeneratorFiniteRangeRandomWalk} by introducing a signed measure 
\begin{align*}
	\mu = \kappa(g_1)\, \left(\frac{1}{2}\delta_{g_1}+\frac{1}{2} \delta_{-g_1}\right)+\ldots + \kappa(g_l) \, \left(\frac{1}{2}\delta_{g_l}+\frac{1}{2}\delta_{-g_l}\right)- \sum_{i=1}^l \kappa(g_i) \delta_0 \,,
\end{align*}	
which allows us to write $L^\eps u= \eps^{-2} \int_{\RR^d} u(x+\eps y) \, \dd \mu (y)$ and $L u= \frac{1}{2} \int_{\RR^d} y\scl  \nabla^2 u\, y \, \dd \mu(y)$. In fact we will also allow the random walk to have infinite range.

\begin{definition}
\label{def:Measure}
We write $\mu \in\mm(\omega)=\mm(\omega,\GG)$ for $\omega\in \ww$ if $\mu$ is a finite, signed measure on a Bravais lattice $\GG$ such that 
\begin{itemize}
	\item $\langle \supp \mu \rangle =\GG$,
	\item $\mu\vert_{ \{0\}^c} \geq 0$,
	\item for any $\lambda>0$ we have $\int_{\GG} e^{\lambda \omega(x)} \, \dd |\mu| (x)<\infty$, where $|\mu|$ is the total variation of $\mu$,
	\item $\mu(A)=\mu(-A)$ for $A\subseteq \GG$ and $\mu(\GG)=0$,
\end{itemize}
where $\langle \cdot \rangle$ denotes the  subgroup generated by $\cdot$ in $(\GG,+)$.
We associate a norm on $\RR^d$ to $\mu\in \mm(\omega)$ which is given by
\begin{align*}
	\|x\|_\mu^2 =\frac{1}{2} \int_{\GG} | x\scl y |^2 \dd\mu(y)\,.
\end{align*}
We also write $\mm(\ww):=\bigcup_{\omega\in \ww}\mm(\omega)$. \glsadd{mmomega}
\end{definition}

\begin{lemma}
\label{lem:MuNorm}
The function $\lVert \cdot \rVert_\mu$ of Definition \ref{def:Measure} is indeed a norm.  
\end{lemma}

\begin{proof}
	The homogeneity is obvious and the triangle inequality follows from Minkowski's inequality. If $\|x\|_\mu=0$ we have $x\scl g=0$ for all $g\in \supp \mu$. Since $\langle \supp \mu\rangle =\GG$ we also have $x\scl a_i=0$ for the linearly independent vectors $a_1,\ldots,a_d$ from \eqref{eq:Lattice}, which implies $x=0$. 
\end{proof}

Given $\mu\in \mm(\ww)$ as in Definition \ref{def:Measure} we can then generalize the formulas we found above. 

\begin{definition}
\label{def:DiffusionOperator}
For  $\omega\in \ww$, $\mu\in \mm(\omega)$ as in Definition \ref{def:Measure} and $\GG^\eps$ as in Definition \ref{def:LatticeSequence} we set
\begin{align*}
	L^\eps_{\mu} u (x)= \eps^{-2} \int_{\GG} u(x+\eps y)\, \dd \mu(y) 
\end{align*}
for $u\in \Sw_\omega'(\GG^\eps)$ and 
\begin{align*}
	(L_\mu u)\,(\varphi):= \frac{1}{2} \int_{\GG} y \scl \nabla^2 u \,y\, \dd \mu(y)\,(\varphi):= \frac{1}{2} \int_{\GG} y \scl \nabla^2 u(\varphi) \,y \,\dd \mu (y) 
\end{align*}
for $u\in \Sw'_\omega(\RR^d)$ and $\varphi\in \Sw_\omega(\RR^d)$. We write further $\mathscr{L}^\eps_\mu,\mathscr{L}_\mu$ for the parabolic operators $\mathscr{L}^\eps_\mu=\partial_t - L^\eps_\mu$ and $\mathscr{L}_\mu =\partial_t-L_\mu$. 
\glsadd{mathscrLepsmu}
\end{definition}

$L^\eps_\mu$ is nothing but the infinitesimal generator of a random walk with sub-exponential moments (Lemma \ref{lem:MomentsRandomWalk}). By direct computation it can be checked that for $\GG=\ZZ^d$ and with the extra condition $\int y_i y_j \dd \mu(y)=2\,\delta_{ij}$ we have the identities $\lVert\cdot\rVert_\mu=\lvert\cdot\rvert$ and $L_\mu=\Delta_{\RR^d}$. In general $L_\mu$ is an elliptic operator with constant coefficients,
\begin{align*}
	L_\mu  u =  \frac{1}{2} \int_{\GG} y \scl \nabla^2 u \,y \,\dd \mu(y)= \frac{1}{2}\sum_{i,j} \int_{\GG} \, y_i y_j \,\dd \mu(y) \cdot  \partial^{ij} u =:  \frac{1}{2} \sum_{i,j}  a^\mu_{ij} \cdot  \partial^{ij}u \,,
\end{align*}
where $(a^\mu_{ij})$ is a symmetric matrix. The ellipticity condition follows from the relation $x \scl (a^ \mu_{ij})x = 2 \|x\|_{\mu}^2$ and the equivalence of norms on $\RR^d$. In terms of regularity we expect therefore that $L^\eps_\mu$ behaves like the Laplacian when we work on discrete spaces.

\begin{lemma}
\label{lem:LaplaceRegularity}
We have for $\alpha\in\RR$, $p\in [1,\infty],\,\omega\in \ww$ and $\mu \in \mm(\omega),\,\rho\in \rr (\omega)$ 
\begin{align*}
	\|L^\eps_\mu u \|_{\mathcal{C}^{\alpha-2}_p(\GG^\eps,\rho)} \lesssim \|u\|_{\mathcal{C}^\alpha_p(\GG^\eps,\rho)}\,,
\end{align*}
where $\mathcal{C}^\alpha_p(\GG^\eps,\rho)=\B^\alpha_{p,\infty}(\GG^\eps,\rho)$ is as in Definition \ref{def:DiscreteBesov}, and where the implicit constant is independent of $\eps$. For $\delta\in [0,1]$ we further have
\begin{align*}
	\|(L^\eps_\mu-L_\mu )u\|_{\mathcal{C}^{\alpha-2-\delta}_p(\RR^d,\rho)} \lesssim \eps^\delta  \|u\|_{\mathcal{C}^{\alpha}_p(\RR^d,\rho)}\,,
\end{align*}
where the action of $L^\eps_\mu $ on $u\in \Sww'(\RR^d)$ should be read as 
\begin{align}
\label{eq:DiscretizedActionDiffusionOperator}
(L^\eps_\mu u)(\varphi)= u\left(\eps^{-2}\int_{\GG} \varphi(\cdot+\eps y)\dd\mu(y) \right) = u\left(\eps^{-2}\int_{\GG} \varphi(\cdot-\eps y)\dd\mu(y) \right) = u(L^\varepsilon_\mu \varphi) 
\end{align}
for $\varphi\in \Sww(\RR^d)$, where we used the symmetry of $\mu$ in the second step.
\end{lemma}
\begin{proof}
We start with the first inequality. With $\overline{\Psi}^{\GG^\eps,j}:=\sum_{-1\leq i \leq j_{\GG^\eps}:\, |i-j|\leq 1} \Psi^{\GG^\eps,i} \in \Sw_\omega(\GG^\eps)$ we have by spectral support properties $\varDelta_j^{\GG^\eps} u=\overline{\Psi}^{\GG^\eps, j}\aste \varDelta_j^{\GG^\eps}  u $. Via \eqref{eq:PsijScaling} we can read $\Psi^{\GG^\eps,j}$ and thus $\overline{\Psi}^{j,\GG^\eps}$ as a smooth function in $\Sww(\RR^d)$ defined on all of $\RR^d$. In this sense we read
\begin{align}
\label{eq:DiffusionOperatorConvolution}
\varDelta_j^{\GG^\eps} u=|\GG^\eps|\sum_{z\in |\GG^\eps|} \overline{\Psi}^{\GG^\eps, j}(\cdot-z)\, \varDelta_j^{\GG^\eps} u(z)\,,
\end{align}
as a smooth function on $\RR^d$ in the following. Since $\mu$ integrates affine functions to zero we can rewrite
\begin{align*}
	\varDelta_j^{\GG^\eps} L_\mu^\eps u(x)  &=\eps^{-2} \int_{\GG}\, \dd \mu (y) \,[\varDelta_j^{\GG^\eps} u(x+\eps y)-\varDelta_j^{\GG^\eps} u(x)-\nabla(\varDelta_j^{\GG^\eps} u)(x)\cdot \eps y] \\
	&=\int_{\GG}\, \dd \mu(y) \int_0^1 \dd \zeta_1 \int_0^1 \dd \zeta_2\,\,\,y\scl \nabla^2(\varDelta_j^{\GG^\eps} u)(x+\eps  \zeta_1 \zeta_2 y) y.
\end{align*}
Using \eqref{eq:WeightQuotient} and the Minkowski inequality on the support of $\mu$ we then obtain
\begin{align*}
\|\rho\varDelta_j^{\GG^\eps} L_\mu^\eps u \|_{L^p(\GG^\eps)} \lesssim \int_{\GG}\, \dd \mu(y) \int_0^1 \dd \zeta_1\int_0^1 \dd \zeta_2 e^{\lambda \omega(\eps \zeta_1 \zeta_2 y)} |y|^2\, \Big\|\rho(\cdot+\eps \zeta_1 \zeta_2 y) |\nabla^2(\varDelta_j^{\GG^\eps} u)(\cdot +\eps  \zeta_1 \zeta_2 y)| \Big\|_{L^p(\GG^\eps)}\,,
\end{align*}
where $\lambda$ is as in \eqref{eq:WeightQuotient}. By definition of $\mm(\omega)$ and monotonicity of $\omega \in \ww$ we have
\begin{align*}
 \int_0^1 \dd \zeta_1 \int_0^1 \dd \zeta_2 \int_{\GG} \dd \mu(y)\, |y|^2 e^{\lambda \omega(\eps \zeta_1 \zeta_2 y)}\leq  \int_0^1 \dd \zeta_1 \int_0^1 \dd \zeta_2 \int_{\GG} \dd \mu(y)\, |y|^2 e^{\lambda \omega( y)}<\infty
\end{align*}
so that we are left with the task of estimating
\begin{align*}
\Big\|\rho(\cdot+\eps \zeta_1 \zeta_2 y) |\nabla^2(\varDelta_j^{\GG^\eps} u)(\cdot +\eps  \zeta_1 \zeta_2 y)| \Big\|_{L^p(\GG^\eps)}\lesssim \|\nabla^2\overline{\Psi}^{\GG^\eps,j}(\cdot+\eps \zeta_1\zeta_2)\|_{L^1(\GG^\eps,e^{\lambda \omega(\cdot+\eps \zeta_1\zeta_2)})}\,\|\varDelta_j^{\GG^\eps} u\|_{L^p(\GG^\eps,\rho)}\,,
\end{align*}
where we applied \eqref{eq:DiffusionOperatorConvolution} and Young's convolution inequality on $\GG^\eps$. Due to \eqref{eq:PsijScaling} and Lemma \ref{lem:DiscreteWeightedYoungInequality} we can estimate the first factor by $2^{j2}$ so that we obtain the total estimate
\begin{align*}
\|\varDelta_j^{\GG^\eps} L_\mu^\eps u \|_{L^p(\GG^\eps,\rho)} \lesssim 2^{-j(\alpha-2)} \|u\|_{\mathcal{C}^\alpha_p(\GG^\eps,\rho)}
\end{align*}
and the first estimate follows. 

To show the second inequality we proceed essentially the same but use instead $\overline{\Psi}^j=\sum_{i:\,|i-j|\leq 1} \Psi^i$, where $\Psi^j=\FF^{-1}_{\RR^d} \varphi_j$ now really denotes the inverse transform of the partition $(\varphi_j)_{j\geq -1}$ \emph{on all of} $\RR^d$. We then have $\varDelta_j=\overline{\Psi}^j \ast \varDelta_j$, so that 
\begin{align*}
	\varDelta_j(L^\eps_\mu -L_\mu)u=\int_0^1 \dd \zeta_1 \int_0^1 \dd \zeta_2\int_{\GG}\dd \mu(y)\,  \int_{\RR^d} \dd z\,  y \scl (\nabla^2\overline{\Psi}^j(\cdot +\eps \zeta_1\zeta_2 y-z)-\nabla^2  \overline{\Psi}^j(\cdot -z)) y \,\varDelta_j u(z)\,.
\end{align*}
As above we can then either get $2^{-j(\alpha-2)} \|u\|_{\mathcal{C}^\alpha_p(\GG^\eps,\rho)}$, by bounding each of the two second derivatives separately, or  $2^{-j(\alpha-3)} \eps \|u\|_{\mathcal{C}^\alpha_p(\GG^\eps,\rho)}$, by exploiting the difference to introduce the third derivative. We obtain the second estimate by interpolation.
\end{proof}
 
\subsection{Semigroup estimates}

In Fourier space $L^\eps_\mu$ can be represented by a Fourier multiplier $l^\eps_\mu:\, \wGGe \rightarrow \RR$:
\begin{align*}
	\FFge (L_\mu^\eps u)=-l^\eps_\mu  \cdot \FFge u \,,
\end{align*} 
\glsadd{lepsmu}
for $u\in \Sww'(\GG^\eps)$. The multiplier $l^\eps_\mu$ is given by 
\begin{align}
\label{eq:MultiplierOperator}
	l^\eps_\mu(x)=-\int_{\GG} \frac{e^{ \imath \eps 2\pi x \scl y}}{\eps^2}\, \dd \mu(y)=\int_{\GG} \frac{1-\cos(\eps 2\pi x \scl y)}{\eps^2}\, \dd \mu(y)=2\int_{\GG} \frac{\sin^2(\eps \pi x \scl y)}{\eps^2} \,\dd \mu(y)\,,
\end{align}
where we used that $\mu$ is symmetric with $\mu(\GG)=0$ and the trigonometric identity $1-\cos=2\sin^2$. 
The following lemma shows that $l^\eps_\mu$ is well defined as a multiplier (i.e. $l^\eps_\mu \in C^\infty_\omega(\widehat{\GG^\eps})$). It is moreover the backbone of the semigroup estimates shown below.

\begin{lemma}
\label{lem:SchauderPrelude}
Let $\omega \in \ww$ and $\mu\in \mm(\omega)$. The function $l^\eps_\mu$ defined in \eqref{eq:MultiplierOperator} is an element of $\Sw_\omega(\widehat{\GG^\eps})=C^\infty_\omega(\widehat{\GG^\eps})$ and
\begin{itemize}
	\item if $\omega=\wexps$ with $\sigma\in (0,1)$ it satisfies $|\partial^k l_\mu^\eps(x)|\lesssim_\delta \eps^{(|k|-2)\vee 0} (1+|x|^{2}) \delta^{|k|} (k!)^{1/\sigma} $ for any $\delta>0,\,k \in \NN^d$,
	\item for every compact set $K \subseteq \RR^d$ with $K\cap \mathscr{R}=\{0\}$, where $\mathscr{R}$ is the reciprocal lattice of the \emph{unscaled} lattice $\GG$, we have $l^\eps_\mu(x) \gtrsim_K \lvert x\rvert^2$ for all $x \in \eps^{-1} K$.
\end{itemize}
The implicit constants are independent of $\eps$.
\end{lemma}

\begin{proof}
We start by showing $|\partial^k l^\eps_\mu(x)|\lesssim_\delta \eps^{(|k|-2)\vee 0} (1+|x|^{2}) \delta^{|k|} (k!)^{1/\sigma}$ if $\omega = \wexps$, 
which implies in particular $l^\eps_\mu\in \Sw_\omega(\widehat{\GG^\eps})$ in that case. The proof that $l^\eps_\mu\in \Sw_\omega(\widehat{\GG^\eps})$ for $\mu\in \mm(\wpol)$ is again similar but easier and therefore omitted. We study derivatives with $|k|=0,1$ first. We have 
\begin{align*}
|l^\eps_\mu(x)|=2\left|\int_{\GG} \frac{\sin^2 (\eps\pi x\scl y)}{\eps^2} \dd \mu(y) \right| & \lesssim  \left|\int_{\GG} \frac{\sin^2 (\eps\pi x\scl y)}{|\eps \pi x\scl y|^2} |x\scl y|^2  \dd \mu(y) \right| \\
& \lesssim \int_{\GG} |y|^2 \dd |\mu| (y)\cdot  |x|^2 \lesssim |x|^2,
\end{align*}
and for $i=1,\ldots,d$
\begin{align*}
	|\partial^{i} l_\mu^\eps(x)|| \lesssim \int_{\GG} \frac{|\sin(\eps\pi x\scl y)|}{|\eps \pi x\scl y|} |x||y|^2 \dd |\mu| (y)  \lesssim |x|\,.
\end{align*}
For higher derivatives we use that $\partial_x^k e^{\imath 2\pi \eps x\scl y}=(\ii 2\pi\eps)^{|k|} y^k e^{\imath 2\pi \eps x\scl y}$
which gives (where $C>0$ denotes as usual a changing constant) 
\begin{align*}
	|\partial^k l_\mu^\eps(x)| &\le \eps^{|k|-2} C^{|k|}\int_{\GG} |y|^{|k|}   \dd |\mu|(y)\leq \eps^{|k|-2}   C^{|k|} \max_{t \geq 0} (t^{|k|} e^{-\lambda t^{\sigma}}) \int_{\GG} e^{\lambda |y|^{\sigma}} \dd \mu(y) 	
\end{align*}
for any $\lambda>0$. 
Using $\max_{t\geq 0} t^a e^{-\lambda t^{\sigma}}= \lambda^{-a/\sigma} (a/\sigma)^{a/\sigma} e^{-a/\sigma}$ for $a>0$ we end up with 
\begin{align*}
|\partial^k l_\mu^\eps(x)|\lesssim \eps^{|k|-2}  \frac{1}{\lambda^{|k|/\sigma}}   C^{|k|} |k|^{|k|/\sigma}\lesssim \eps^{|k|-2} \frac{1}{\lambda^{|k|/\sigma}}  C^{|k|}  (k!)^{1/\sigma}\,,
\end{align*}
and our first claim follows by choosing $\lambda^{1/\sigma} := C/\delta$. 

It remains to show that $l_\mu^\eps/\lvert\cdot \rvert^2\gtrsim 1$ on $\eps^{-1} K$, which is equivalent to $l_\mu^1/\lvert \cdot \rvert^2\gtrsim 1$ on $K$. We start by finding the zeros of $l^1_\mu$ which, by periodicity can be reduced to finding all $x\in \widehat{\GG}$ with $l^1_\mu(x)=0$. But if $l^1_\mu(x)=0$, then $y\scl x\in \ZZ$ for any $y\in \supp \mu$, which yields with $\langle  \supp \mu \rangle=\GG$ that we must have	$a_i \scl x\in \ZZ$ for $a_i$ as in \eqref{eq:Lattice}. But since $x \in \widehat{\GG}$ we have $x=x_1 \hat{a}_1+\ldots +x_d \hat{a}_d$ with $x_i\in [-1/2,1/2)$ and $\hat{a}_i$ as in \eqref{eq:ReciprocalLattice}. Consequently 
\begin{align*}
	x_i=x\scl a_i\in \ZZ\cap [-1/2,1/2)=\{0\}\,,
\end{align*} 
and hence $x = 0$. Since $l^1_\mu$ is periodic under translations in the reciprocal lattice $\mathscr R$, its zero set is thus precisely $\mathscr{R}$. By assumption $K\cap \mathscr{R}=\{0\}$ and it remains therefore to verify $l^1_\mu(x) \gtrsim |x|^2$ in an environment of $0$ to finish the proof.

Note that there is a finite subset $V\subseteq \supp \mu$ such that $0\in V$ and $\langle V \rangle =\GG$, since  only finitely many $y\in \supp \mu$ are needed to generate $a_1,\ldots,a_d$. We restrict ourselves to $V$:
\begin{align*}
	l^1_\mu(x)&=2 \int_{\GG} \sin^2 (\pi x\scl y) \dd \mu(y) \geq 2 \int_{V} \sin^2 (\pi x\scl y) \dd \mu(y)
\end{align*}
For $x\in \widehat{\GG}\backslash \{0\}$ small enough we can now bound $
	\int_{V} \sin^2 (\pi x\scl y) \dd \mu(y) \gtrsim \int_V |x\scl y|^2 \dd \mu(y)$. 
 The term on the right hand side defines (the square of) a norm by the same arguments as in Lemma \ref{lem:MuNorm}, and since it must be equivalent to $\lvert\cdot\rvert^2$ the proof is complete.  
\end{proof}

Using that $\Sw_\omega(\widehat{\GG^\eps})=C^\infty_\omega(\widehat{\GG^\eps})$ is stable under composition with functions in $\Eww(\RR^d)$ we see that $e^{-t l^\eps_\mu}\in C^\infty_\omega(\widehat{\GG^\eps})$ for $t\geq 0$ and can thus define the Fourier multiplier
\begin{align*}
	e^{t L^\eps_\mu} f:= \FFge^{-1} (e^{-t l^\eps_\mu} \FFge f)
\end{align*}
for $t\geq 0$ and $f\in \Sw_\omega'(\GG^\eps)$, which gives the (weak) solution to the problem $\mathscr{L}_\mu^\eps g=0$, $g(0)=f$. The regularizing effect of the semigroup is described in the following proposition.

\begin{proposition}
\label{lem:SemiGroup}
We have for $\alpha\in \RR,\beta \geq 0$, $p\in [1,\infty],\,\omega\in\ww,\,\mu\in \mm(\omega)$ and $\rho\in \rr(\omega)$ 
\begin{align}
	\| e^{t L^\eps_\mu} f	 \|_{\mathcal{C}_p^{\alpha+\beta}(\GG^\eps,\rho)} &\lesssim t^{-\beta/2} \| f \|_{\mathcal{C}^\alpha_{p}(\GG^\eps,\rho)} \,, 
\label{eq:SemigroupHoelder}	
	\\ 
	\|e^{t L^\eps_\mu} f\|_{\mathcal{C}^\beta_p(\GG^\eps,\rho)} &\lesssim t^{-\beta/2} \|f\|_{L^p(\GG^\eps,\rho)} \,,
\label{eq:SemigroupLp}
\end{align}
and for $\alpha\in (0,2)$
\begin{align}
	\|(e^{t L^\eps_\mu}-\mathrm{Id}) f\|_{L^p(\GG^\eps,\rho)} \lesssim t^{\alpha/2} \|f\|_{\mathcal{C}^\alpha_p(\GG^\eps,\rho)}\,,
\label{eq:SemigroupContinuity}
\end{align}
uniformly on compact intervals $t\in [0,T]$. The involved constants are independent of $\eps$.
\end{proposition}

\begin{proof}
We show the claim for $\omega=\wexps=|x|^\sigma,\,\sigma\in (0,1)$ , the arguments for $\omega=\wpol$ are similar but easier. 
Using spectral support properties we can rewrite for $j= -1,\ldots,j_{\GGe}$
\begin{align}
	\Delta_j^{\GGe} e^{tL^\eps_\mu} f
=\FFge^{-1}  \left(\sum_{i:\,|i-j|\leq 1} \varphi^{\GG^\eps}_i e^{-t l^\eps_\mu} \cdot  \FFge \varDelta_j^{\GG^\eps} f \right)=\mathscr{K}_j(t,\cdot)\aste \varDelta_j^{\GG^\eps} f\,,
\label{eq:SemiGroupAsConvolution}
\end{align}
where we set for $z\in \GG^\eps$
\begin{align*}
\mathscr{K}_j(t,z):=\int_{\wGGe} \dd y \, e^{2\pi \ii z\scl y }\, \sum_{i:\,|i-j|\leq 1}\varphi^{\GG^\eps}_i(y) e^{-t l^\eps_\mu(y)}.
\end{align*}

Using the smear function $\psi^\eps=\psi(\eps\cdot)$ from Subsection \ref{subsec:ExtensionOperator}	we can rewrite this as an expression that is well-defined \emph{for all} $x\in \RR^d$
\begin{align*}
\mathscr{K}_j(t,x):=\int_{\RR^d} \dd y \, e^{2\pi \ii x \scl y }\,\psi^\eps(y) \,\sum_{i:\,|i-j|\leq 1}\ext{\varphi^{\GG^\eps}_i}(y) \cdot e^{-t l^\eps_\mu(y)}\,,
\end{align*}
where $\ext{\cdot}$ is given as in \eqref{eq:PeriodicExtension2} and where we extended $l^\eps_\mu$ (periodically) to all of $\RR^d$ by relation \eqref{eq:MultiplierOperator}. 
Consequently, we can apply Lemma \ref{lem:ScalingBoundary} to give an expression for the scaled kernel
\begin{align*}
\mathscr{K}_{(j)}(t,x):=2^{-jd}\mathscr{K}_j(t,2^{-j}x)=\int_{\RR^d} \dd y \, e^{2\pi \ii x \scl y } \varphi_{(j)}(y) \cdot e^{-t l^\eps_\mu(2^j y)}\,,
\end{align*}
where we wrote $\varphi_{(j)}=\sum_{i:\,|i-j|\leq 1} \check{\phi}_{\langle i \rangle_\eps}(2^{-(i-j)}\cdot)$ with $\check{\phi}_{\langle i \rangle_\eps}$ as in Lemma \ref{lem:ScalingBoundary}. Suppose we already know that for any $\lambda>0$ and $x\in\GG^\eps$ the estimate
\begin{align}
\label{eq:SchauderKeyBound}
	|\mathscr{K}_{(j)}(t,x)|\lesssim_\lambda  e^{-\lambda |x|^{\sigma}} 2^{-j \beta} t^{-\beta/2}=: 2^{-j\beta} t^{-\beta/2} \Phi(x) \,
\end{align}
holds. We then obtain from \eqref{eq:SemiGroupAsConvolution} with $\Phi^{2^{-j}}(x):=2^{jd}\Phi(2^j x)=2^{jd}e^{-\lambda |2^j x|^{\sigma}}$ the bound
\begin{align*}
\|\Delta_j^{\GGe} e^{tL^\eps_\mu} f\|_{L^p(\GG^\eps,\rho)}\lesssim 2^{-j\beta} t^{-\beta/2} \|\Phi^{2^{-j}}\aste |\Delta_j^{\GGe} e^{tL^\eps_\mu} f|\|_{L^p(\GG^\eps,\rho)}
\end{align*}
and an application of Lemma \ref{lem:DiscreteWeightedYoungInequality} shows \eqref{eq:SemigroupHoelder} and \eqref{eq:SemigroupLp} (for \eqref{eq:SemigroupLp} we also need \eqref{eq:LpLp}). Note that we cheated a little bit as Lemma \ref{lem:DiscreteWeightedYoungInequality} actually requires $\Phi\in \Sww(\RR^d)$ which is not true, inspecting however the proof of Lemma \ref{lem:DiscreteWeightedYoungInequality} we see that all we used was a suitable decay behavior which is still given. 

We will now show \eqref{eq:SchauderKeyBound}. 
Using Lemma \ref{lem:PolynomialBoundGivesSubexponential} below we can reduce this task to the simpler problem of proving the polynomial bound for $i=1,\ldots,d$ and $n\in \NN$
\begin{align}
\label{eq:SchauderKeyBoundPolynomial}
t^{\beta/2} |x_i|^n |\mathscr{K}_{(j)}(t,x)|\lesssim_\delta \delta^n C^n (n!)^{1/\sigma} 2^{-j\beta },\qquad \delta > 0,
\end{align}
with a constant $C>0$ that does not depend on $\delta$. To show~\eqref{eq:SchauderKeyBoundPolynomial} we assume that $2^j\eps\leq 1$. Otherwise we are dealing with the scale $2^j\approx \eps^{-1}$ and the arguments below can be easily modified. Integration by parts gives
\begin{align*}
	|x_i|^n |\mathscr{K}_{(j)}(t,x)|& = C^n \left| \int_{\RR^d}  \dd y \, e^{2\pi \imath x \scl y} \,\partial^{n\cdot e_i}\left(\varphi_{(j)}\, e^{-t l^\eps_\mu(2^{j} \cdot )} \right)(y) \right| \\
	& \leq C^n \int_{\RR^d} \dd y \left| \partial^{n\cdot e_i}\left(\varphi_{(j)} e^{-t 2^{2j} l_\mu^{2^j\eps}   } \right)(y) \right|  \,,
\end{align*}
where we used that $l^\eps_\mu(2^{j} y )=2^{2j} l_\mu^{2^j\eps}(y)$ by \eqref{eq:MultiplierOperator}. Now we have the following estimates for $k\in \NN$
\begin{align*}
	|\partial^{k\cdot e_i} \varphi_{(j)}(y)| \lesssim_\delta \delta^k (k!)^{1/\sigma} ,\,\,
	|\partial^{k\cdot e_i} l^{\eps 2^j}_\mu(y)| \lesssim_\delta \delta^{k} \,(k!)^{1/\sigma},\,\,
	\left|(2^{2j} t)^{\beta/2} \partial^k \left(e^{t2^{2j} \cdot} \right)\big(l^{2^j \eps}_\mu(y)\big)  \right| \lesssim_\delta k^{k/\sigma} \delta^k\,,
\end{align*}
where we used that $\varphi_{(j)}\in \Dww(\RR^d)$ (with bounds that can be chosen independently of $j$ by definition) and we applied Lemma \ref{lem:SchauderPrelude} with the assumption $2^j\eps\leq 1$ (which we need because we only defined $l^{\eps'}_\mu$ for $\eps' \leq 1$). Together with Leibniz's and Fa\`a-di Bruno's formula and a lengthy but elementary calculation \eqref{eq:SchauderKeyBoundPolynomial} follows, which finishes the proof of \eqref{eq:SemigroupHoelder} and \eqref{eq:SemigroupLp}.

The last estimate~\eqref{eq:SemigroupContinuity} can be obtained as in the proof of Lemma \cite[Lemma 6.6]{Gubinelli2017KPZ} by using Lemma \ref{lem:SchauderLp} below. 
\end{proof}

\begin{lemma}
\label{lem:PolynomialBoundGivesSubexponential}
Let $g\colon\RR^d\rightarrow \RR$, $\sigma>0 $ and $B>0$. Suppose for any $\delta>0$ there is a $C_\delta>0$ such that for all $z\in \RR^d$, $l\geq 0$ and $i=1,\ldots,d$ 
\begin{align*}
	|z_i^l  g(z)|\lesssim_\delta \delta^l C_\delta^l (l!)^{1/\sigma}  B\,.
\end{align*}
It then holds for any $\lambda>0$ and $z\in \RR^d$ 
\begin{align*}
|g(z)|\lesssim_\lambda B  e^{-\lambda |z|^\sigma}\,.
\end{align*}
\end{lemma}
\begin{proof}
This follows ideas from \cite[Proposition A.2]{WeberMourrat}. Without loss of generality we can assume $|z|>1$ (otherwise we get the required estimate by taking $l=0$). Recall that we have $|z|^l \leq C^l \sum_{i=1}^d |z_i|^l$, where $C>0$ denotes a constant that changes from line to line and is independent of $l$. Consequently,  Stirling's formula gives 
\begin{align*}
|e^{\lambda |z|^\sigma} g(z)| &=\left|\sum_{k=0}^\infty \frac{\lambda^k}{k!} |z|^{\sigma k} g(z) \right|  
\lesssim \sum_{k=0}^\infty \frac{\lambda^k C^k}{k^k} |z|^{\lceil k\sigma \rceil} \, |g(z)|
\lesssim \sum_{k=0}^\infty \frac{\lambda^k C^k}{k^k} \sum_{i=1}^d |z_i^{\lceil k\sigma \rceil} g(z)| \\
&\lesssim B \sum_{k=0}^\infty \frac{\lambda^k C^k \delta^{k\sigma}}{k^k} \lceil k\sigma \rceil^{\lceil k\sigma\rceil /\sigma }\lesssim B \sum_{k=0}^\infty \frac{\lambda^k C^k \delta^{k\sigma}}{k^k}  k^{k} = B\,\sum_{k=0}^\infty \lambda^k C^k \delta^{k\sigma}\lesssim_\lambda B\,,
\end{align*}
where we used $\lceil k\sigma \rceil\leq k\lceil \sigma \rceil$ so that $\lceil k\sigma \rceil^{\lceil k\sigma\rceil /\sigma }\leq (\lceil\sigma\rceil k)^{\frac{k\sigma+1}{\sigma}} \lesssim C^k k^k$ and where we chose $\delta<(C\,\lambda)^{-\frac{1}{\sigma}}$ in the last step. 
\end{proof}

\subsection{Schauder estimates}
We will follow here closely \cite{Gubinelli2017KPZ} and introduce time-weighted parabolic spaces $\mathscr{L}^{\gamma,\alpha}_{p,T}$ that interplay nicely with the semigroup $e^{t L^\eps_\mu}$.
\begin{definition}
\label{def:WeightedSpaces}
Given $\gamma\geq  0$, $T>0$ and an increasing family of normed spaces $X=(X(s))_{s\in [0,T]}$ we define the space
\begin{align*}
	\mathcal{M}^\gamma_T X:=\left\{ f\colon[0,T]\rightarrow X(T) \,\middle\vert\, \|f\|_{\mathcal{M}^\gamma_T X}=\sup_{t\in [0,T]} \|t^\gamma f(t)\|_{X(t)}<\infty  \right\}\,,
\end{align*}
and for $\alpha>0$ 
\begin{align*}
	C^\alpha_T X:=\left\{f\in C([0,T],X(T))\,\middle\vert \, \|f\|_{C^\alpha_T X}<\infty  \right\}\,,
\end{align*}
where
\[
   \|f\|_{C^\alpha_T X} :=\sup_{t\in [0,T]} \|f(t)\|_{X(t)} + \sup_{0\leq s\leq t\leq T}\frac{\|f(s)-f(t)\|_{X(t)}}{|s-t|^\alpha}\,.
\]
For a lattice $\GG$, parameters $\gamma\geq 0,T>0,\alpha \geq  0,\,p\in [1,\infty]$ and a pointwise decreasing map $\rho\colon[0,T]\ni t \mapsto \rho(t) \in \boldsymbol{\rho}(\omega)$ we set
\begin{align*}
	\mathscr{L}^{\gamma,\alpha}_{p,T}(\GG, \rho) := \left\{f\colon[0,T] \rightarrow \mathcal{S}_\omega'(\GG)\, \middle\vert \, \|f\|_{\mathscr{L}^{\gamma,\alpha}_{p,T}(\GG,\rho)}<\infty \right\}\,,
\end{align*}
where
\[
   \|f\|_{\mathscr{L}^{\gamma,\alpha}_{p,T}(\GG,\rho)} := \|t\mapsto t^\gamma f(t)\|_{C^{\alpha/2}_T L^p(\GG,\rho)}+\|f\|_{\mathcal{M}^\gamma_T \mathcal{C}^\alpha_p(\GG,\rho)}\,.
\]
\glsadd{mathscrLgammaalphapT}
\glsadd{mathcalMgammaTX}
\end{definition}

\begin{remark}
\label{rem:ContinousSpacesParabolic}
As in Remark \ref{rem:ContinousSpaces} the definition of the continuous version $\mathscr{L}^{\gamma,\alpha}_{p,T}(\RR^d, \rho)$ is analogous. 
\end{remark}
Standard arguments show that if $X$ is a sequence of increasing Banach spaces with decreasing norms, all the spaces in the previous definition are in fact complete in their (semi-)norms.

The Schauder estimates for the operator 
\begin{align}
	I^\eps_\mu f(t)=\int_0^t e^{(t-s) L^\eps_\mu}f(s) \,\dd s
\end{align}
and the semigroup $(e^{tL^\eps_\mu})$ in the time-weighted setup are  summarized in the following lemma, for which we introduce the weights
\glsadd{pkappa} 
\glsadd{esigmal}
\begin{align}
\label{eq:parabolicweights}
p^\kappa(x) &=(1+|x|)^{-\kappa} \\
e^\sigma_{l+t}(x)&=e^{-(l+t)(1+|x|)^\sigma}
\end{align}
with $\kappa>0$ and $l,t\in \RR$. The parameter $t$ should be thought of as \emph{time}. The notation $\mathscr{L}^{\gamma,\alpha}_{p,T}(\GG,e^\sigma_l)$ means therefore that we take the time-dependent weight $(e^\sigma_{l+t})_{t \in [0,T]}$, while $e^\sigma_l p^\kappa$ stands for the time-dependent weight $(e^\sigma_{l+t} p^\kappa)_{t \in [0,T]}$.

\begin{lemma}
\label{lem:Schauder}
Let $\GG^\eps$ be as in Definition \ref{def:LatticeSequence}, $\alpha\in (0,2),\gamma\in [0,1), p\in [1,\infty]$, $\sigma \in (0,1)$ and $T>0$. If $\beta\in \RR$ is such that $(\alpha+\beta)/2\in [0,1)$, then we have uniformly in $\eps$
\begin{align}
	\|s\mapsto e^{s{L}^\eps_\mu }f_0\|_{\mathscr{L}^{(\alpha+\beta)/2,\alpha}_{p,T}(\GG^\eps,e^\sigma_l)} \lesssim \|f_0\|_{\mathcal{C}^{-\beta}_p(\GG^\eps,e^\sigma_l)}\,,
	\label{eq:Schauder1}
\end{align}
and if $\kappa\geq 0$ is such that $\gamma+\kappa/\sigma\in [0,1)$, $\alpha+2\kappa/\sigma\in (0,2)$ also
\begin{align}
\label{eq:Schauder2}
	\|I^\eps_\mu f\|_{\mathscr{L}^{\gamma,\alpha}_{p,T}(\GG^\eps,e^\sigma_l)} 
	\lesssim
	\|f\|_{\mathcal{M}^{\gamma}_T\mathcal{C}_p^{\alpha+2\kappa/\sigma-2}(\GG^\eps,e^\sigma_l p^\kappa)}\,.
\end{align}
\end{lemma}

\begin{proof}
The proof is along the lines of Lemma 6.6 in \cite{Gubinelli2017KPZ} with the use of the simple estimate
\begin{align*}
	  p^\kappa e^\sigma_{l+s}\lesssim \frac{e^\sigma_{l+t}}{|t-s|^{\kappa/\sigma}} ,\qquad t \ge s,
\end{align*} 
which is similar to an inequality from the proof of Proposition 4.2 in \cite{HairerLabbeR2} and the reason for the appearance of the term $2\kappa/\sigma$ in~\eqref{eq:Schauder2} (the factor 2 comes from parabolic scaling). We need $\gamma+\kappa/\sigma\in [0,1)$ so that the singularity $|t-s|^{-\gamma - \kappa/\sigma}$ is integrable on $[0,t]$.
\end{proof}

For the comparison of the parabolic spaces $\mathscr{L}^{\gamma,\alpha}_{p,T}$ the following lemma will be convenient. 

\begin{lemma}
\label{lem:ComparisonParabolicSpaces}
Let $\GG^\eps$ be as in Definition \ref{def:LatticeSequence}.
For $\alpha\in (0,2),\,\gamma\in (0,1),\,\eps\in [0,\alpha \wedge 2\gamma),\,p\in [1,\infty],\,T>0$ and a pointwise decreasing $\RR_+\ni  s \mapsto \rho(s)\in  \boldsymbol{\rho}(\omega)$ we have 
\begin{align*}
	\|f\|_{\mathscr{L}^{\gamma-\eps/2,\alpha-\eps}_{p,T}(\GG^\eps,\rho)} \lesssim 	\|f\|_{\mathscr{L}^{\gamma,\alpha}_{p,T}(\GG^\eps,\rho)}\,,
\end{align*}
and for $\gamma\in [0,1)$ and $\eps \in (0,\alpha)$
\begin{align*}
	\|f\|_{\mathscr{L}^{\gamma,\alpha-\eps}_{p,T}(\GG^\eps,\rho)}\lesssim\mathbf{1}_{\gamma=0} \|f(0)\|_{\mathcal{C}^{\alpha-\eps}_p(\GG^\eps,\rho)}+ T^{\eps/2} \|f\|_{\mathscr{L}^{\gamma,\alpha}_{p,T}(\GG^\eps,\rho)}\,.
\end{align*} 
All involved constants are independent of $\eps$.
\end{lemma}
\begin{proof}
The first estimate is proved as in \cite[Lemma 6.8]{Gubinelli2017KPZ}. For $\gamma=0$ the proof of the second inequality works as in Lemma 2.11 of \cite{Gubinelli2017KPZ}. The general case follows from the fact that $f\in\mathscr{L}^{\gamma,\alpha}_{p,T}$ if and only if $t\mapsto t^\gamma f\in \mathscr{L}^{0,\alpha}_{p,T}$.
\end{proof}

\section{Paracontrolled analysis on Bravais lattices}
\label{sec:ParacontrolledAnalysisonBravaisLattices}

\subsection{Discrete Paracontrolled Calculus}

Given two distributions $f_1,f_2\in \Sw'(\RR^d)$, Bony \cite{Bony} defines their \emph{paraproduct} as
\begin{align*}
	f_1\para f_2 := \sum_{1\leq j_2}\,\, \sum_{-1\leq j_1<j_2-1} \varDelta_{j_1} f_1 \cdot \varDelta_{j_2} f_2
	= \sum_{1\leq j_2} S_{j_2-1}  f_1 \cdot \varDelta_{j_2} f_2\,,
\end{align*}
which turns out to always be a well-defined expression. However, to make sense of the product $f_1 f_2$ it is not sufficient to consider $f_1 \para f_2$ and $f_1 \lpara f_2 := f_2 \para f_1$, we also have to take into account the \emph{resonant term}~\cite{GIP}
\begin{align*}
	f_1\reso f_2 := \sum_{-1\leq j_1,\,j_2:\, |j_1-j_2|\leq 1} \varDelta_{j_1} f_1 \cdot \varDelta_{j_2} f_2\,,
\end{align*}
which can in general only be defined under compatible regularity conditions such as $f_1\in \mathcal{C}^{\alpha}_\infty(\RR^d)$, $f_2\in \mathcal{C}_\infty^{\beta}(\RR^d)$ with $\alpha+\beta>0$ (see e.g. \cite{Bahouri} or \cite[Lemma 2.1]{GIP}). If these conditions are satisfied we decompose
$f_1 f_2=f_1 \para f_2+ f_1 \lpara f_2 +f_1 \reso f_2$. Bony's construction can easily be adapted to a discrete and weighted setup, where of course we have no problem in making sense of pointwise products but we are interested in uniform estimates.
\begin{definition}
Let $\GGe$ be a Bravais lattice, $\omega\in \ww$ and $f_1, \,f_2 \in \Sw_\omega'(\RR^d)$. We define the \emph{discrete paraproduct}
\begin{align}
\label{eq:Gparaproduct}
	f_1 \para^\GG f_2 := \sum_{1\leq j_2\leq j_{\GG}}\,\, \sum_{-1\leq j_1<j_2-1} \varDelta_{j_1}^\GG f_1 \cdot \varDelta_{j_2}^\GG f_2 =\sum_{1\leq j \leq j_{\GG}} S_{j-1}^{\GG} f_1 \cdot \varDelta_j f_2 \,,
\end{align}
where the discrete Littlewood-Paley blocks $\varDelta_j^\GG$ are constructed as in Section \ref{sec:BravaisLattices}. 
We also write $f_1 \lpara^\GG f_2 : = f_2 \para^\GG f_1$. The \emph{discrete resonant term} is given by
\begin{align}
\label{eq:Gresonanceproduct}
	f_1 \reso^\GG f_2 := \sum_{1\leq j_1,j_2 \leq j_{\GG},\,|j_1-j_2|\leq 1} \varDelta_{j_1}^\GG f_1 \cdot \varDelta_{j_2}^\GG f_2 \,.
\end{align}
If there is no risk for confusion we may drop the index $\GG$ on $\para$, $\lpara$, and $\reso$. 
\end{definition}
\glsadd{para} \glsadd{reso}
In contrast to the continuous theory $\reso^\GG$ is well defined without any further restrictions since it only involves a finite sum. All the estimates that are known from the continuous theory carry over.

\begin{lemma}
\label{lem:ParaproductEstimates}
Given $\GG^\eps$ as in Definition \ref{def:LatticeSequence}, $\rho_1,\rho_2\in \boldsymbol{\rho}(\ww)$ and $p\in [1,\infty]$ we have the bounds:
\begin{enumerate}[(i.)]
	\item  For any $\alpha_2\in\RR$ 
\begin{align*}
	\|f_1 \para f_2\|_{\mathcal{C}_p^{\alpha_2}(\GG^\eps,\rho_1\cdot\rho_2)} \lesssim \|f_1\|_{L^\infty(\GG^\eps,\rho_1)} \, \|f_2\|_{\mathcal{C}_p^{\alpha_2}(\GG^\eps,\rho_2)} \wedge \|f_1\|_{L^p(\GG^\eps,\rho_1)}\,\|f_2\|_{\mathcal{C}^{\alpha_2}_\infty(\GG^\eps,\rho_2)}\,,
\end{align*}
\item for any $\alpha_1<0$, $\alpha_2\in\RR$
\begin{align*}
	\|f_1 \para f_2\|_{\mathcal{C}_p^{\alpha_1+\alpha_2}(\GG^\eps,\rho_1\cdot\rho_2)} \lesssim \|f_1\|_{\mathcal{C}_p^{\alpha_1}(\GG^\eps,\rho_1)} \,\|f_2\|_{\mathcal{C}_\infty^{\alpha_2}(\GG^\eps,\rho_2)}\wedge\|f_1\|_{\mathcal{C}_\infty^{\alpha_1}(\GG^\eps,\rho_1)} \,\|f_2\|_{\mathcal{C}_p^{\alpha_2}(\GG^\eps,\rho_2)}\,,
\end{align*}
\item for any $\alpha_1,\alpha_2\in\RR$ with $\alpha_1+\alpha_2>0$
\begin{align*}
		\|f_1 \reso f_2\|_{\mathcal{C}_p^{\alpha_1+\alpha_2}(\GG^\eps,\rho_1\cdot\rho_2)} \lesssim \|f_1\|_{\mathcal{C}_p^{\alpha_1}(\GG^\eps,\rho_1)} \, \|f_2\|_{\mathcal{C}_\infty^{\alpha_2}(\GG^\eps,\rho_2)}\wedge \|f_1\|_{\mathcal{C}_p^{\alpha_1}(\GG^\eps,\rho_1)} \, \|f_2\|_{\mathcal{C}_\infty^{\alpha_2}(\GG^\eps,\rho_2)}\,,
\end{align*}
\end{enumerate} 
where all involved constants only depend on $\GG$ but not on $\eps$. All estimates have the property \eqref{eq:EProperty} if the regularity on the left hand side is lowered by an arbitrary $\kappa>0$.
\end{lemma}

\begin{proof}
Similarly as in the continuous case $S_{j-1}^{\GGe} f_1 \cdot \varDelta_j^{\GGe} f_2$ is spectrally supported on a set of the form $2^j\rA\cap \wGGe$, where $\rA$ is an annulus around $0$. Similarly, we have for $i,j$ with $i\sim j$ that the function $\varDelta_i^{\GGe}f_1\cdot \varDelta^{\GGe}_{j} f_2$ is spectrally supported in a set of the form $2^j \B\cap \wGGe$, where $\B$ is a ball around $0$. We give a proof of these two facts in the appendix (Lemma \ref{lem:SupportParaproduct}). Using these two observations the proof of the estimates in \textit{(i.)}-\textit{(iii.)} follows along the lines of \cite[Lemma 2.1]{GIP}) (which in turn is taken from~\cite[Theorem~2.82, Theorem 2.85]{Bahouri}).

We are left with the task of proving the property \eqref{eq:EProperty}. We show in Lemma \ref{lem:ExtensionDecay} below that there is an $N\in \NN$ (independent of $\eps$ and $j$) such that for $-1\leq i\leq j \leq j_{\GGe}-N$
\begin{align}
\label{eq:precExtensionDecay}
\EE^\eps(\varDelta_i^{\GGe} f_1 \cdot \varDelta_j^{\GGe} f_2) = \varDelta_i \EE^\eps f_1\cdot \varDelta_j \EE^\eps f_2\,.
\end{align}

Consequently we can write
\begin{align*}
\EE^\eps(f_1 \para^{\GGe} f_2)=\short \sum_{1\leq j\leq j_{\GGe}} \short  \EE^\eps\big( S_{j-1}^{\GG^\eps} f_1\cdot \varDelta_j^{\GGe} f_2 \big)= \short \short \sum_{1\leq j\leq j_{\GGe}-N}  \short \short S_{j-1} \EE^\eps f_1 \cdot \varDelta_j \EE^\eps f_2+ \xshort \sum_{j_{\GGe}-N< j\leq j_{\GGe}} \short \short  \EE^\eps\big( S_{j-1}^{\GG^\eps} f_1\cdot \varDelta_j^{\GGe} f_2 \big)\,,
\end{align*}
where we used \eqref{eq:precExtensionDecay} and $S_{j-1}^{\GGe}=\sum_{-1\leq i <j-1} \varDelta_i^{\GGe},\,S_{j-1}=\sum_{-1\leq i< j-1} \varDelta_i$. On the other hand we can write
\begin{align*}
\EE^\eps f_1 \para \EE^\eps f_2=\sum_{ 1\leq j} S_{j-1} \EE^\eps f_1 \cdot \varDelta_j  \EE^\eps f=  \short \short \sum_{1\leq j\leq j_{\GGe}-N}  \short \short S_{j-1} \EE^\eps f_1 \cdot \varDelta_j \EE^\eps f_2  + \sum_{j\sim j_{\GGe}} S_{j-1}\EE^\eps f_1 \cdot \varDelta_j \EE^\eps f_2\,,
\end{align*}
where we used in the second step that $\EE^\eps f_2=\FFr(\psi(\eps\cdot)\,(\FFge f_2)_{\mathrm{ext}})$ is spectrally supported in a ball of size $\eps^{-1}\approx 2^{j_{\GG^\eps}}$ to drop all $j$ with $j\gsim j_{\GGe}$. In total we obtain
\begin{align*}
\EE^\eps(f_1 \para^{\GGe} f_2)-\EE^\eps f_1 \para \EE^\eps f_2=\sum_{j\sim j_{\GGe}}    \EE^\eps\big( S_{j-1}^{\GG^\eps} f_1\cdot \varDelta_j^{\GGe} f_2 \big) - \sum_{j\sim j_{\GGe}} S_{j-1}\EE^\eps f_1 \cdot \varDelta_j \EE^\eps f_2\,.
\end{align*}

Note that the two sums on the right hand side are spectrally supported in an annulus of size $2^{j_{\GGe}}$. 
Using Lemma \ref{lem:ExtensionOperator}, the fact $\varDelta_i\colon L^p(\RR^d,\rho)\rightarrow L^p(\RR^d,\rho)$ (by \eqref{eq:LpLpRRd}) and that $\EE^\eps\colon L^p(\GG^\eps,\rho)\rightarrow L^p(\RR^d,\rho)$ (due to \eqref{eq:ExtensionAsConvolution} and Lemma \ref{lem:DiscreteWeightedYoungInequality}), with uniform bounds, we can thus estimate
\begin{align*}
\|\varDelta_i\left(\EE^\eps(f_1 \para^{\GG^\eps} f_2)-\EE^\eps f_1 \para \EE^\eps f_2\right)\|_{L^p(\RR^d,\rho)} &\lesssim
\mathbf{1}_{i\sim j_{\GG^\eps}} \Big( \sum_{j\sim j_{\GGe} }{\| S_{j-1}^{\GG^\eps} f_1 \cdot\varDelta_j^{\GG^\eps} f_2\|_{L^p(\GG^\eps,\rho)}} \\
 &\quad+ \sum_{j\sim j_{\GGe} }{\|S_{j-1} \EE^\eps f_1 \cdot \varDelta_j \EE^\eps  f_2\|_{L^p(\RR^d,\rho)}} \Big).
\end{align*} 
Assume without loss of generality that the right hand side of estimate \textit{(i.)} is bounded by $1$. We then have using $S_{j-1}^{\GGe}\colon L^q(\GGe,\rho)\rightarrow L^q(\GGe,\rho)$ (by Lemma \ref{lem:ScalingBoundary} and Lemma \ref{lem:DiscreteWeightedYoungInequality}) and $S_{j-1}\colon L^q(\RR^d,\rho)\rightarrow L^q(\RR^d,\rho)$ (by \eqref{eq:ScalingSumVarphij} and Young's inequality) for $q\in [1,\infty]$, both with uniform bounds,
\begin{align*}
\|\varDelta_i\left(\EE^\eps(f_1 \para^{\GG^\eps} f_2)-\EE^\eps f_1 \para \EE^\eps f_2\right)\|_{L^p(\RR^d,\rho)} \lesssim \mathbf{1}_{i\sim j_{\GG^\eps}} \sum_{j\sim j_{\GGe} } 2^{-j\alpha_2}\lesssim \mathbf{1}_{i\sim j_{\GG^\eps}} 2^{-j_{\GGe} \alpha_2}\lesssim   2^{-i(\alpha_2-\kappa)} \eps^\kappa\,.
\end{align*}
In the last step we used that $2^{-j_{\GGe}}\approx \eps$ by definition of $j_{\GGe}$. This shows the property \eqref{eq:EProperty} for estimate \textit{(i.)}. If the right hand side of estimate \textit{(ii.)} is uniformly bounded by $1$ we obtain the bound
\begin{align*}
\|\varDelta_i\left(\EE^\eps(f_1 \para^{\GG^\eps} f_2)-\EE^\eps f_1 \para \EE^\eps f_2\right)\|_{L^p(\RR^d,\rho)} & \lesssim \mathbf{1}_{i\sim j_{\GG^\eps}} \sum_{j\sim j_{\GGe} }\,\, \sum_{-1\leq j'< j-1} 2^{-j'\alpha_1} 2^{-j\alpha_2} \\
& \lesssim \mathbf{1}_{i\sim j_{\GG^\eps}} 2^{-j_{\GG^\eps}(\alpha_1+\alpha_2)}\lesssim 2^{-i(\alpha_1+\alpha_2-\kappa)} \eps^\kappa
\end{align*}
 and the property \eqref{eq:EProperty} for \textit{(ii.)} follows. Considering case \textit{(iii.)} assume once more that the right hand side is bounded by 1. We get, by once more applying \eqref{eq:precExtensionDecay},
 \begin{align*}
 \EE^\eps (f_1 \reso^{\GG^\eps} f_2)-\EE^\eps f_1 \reso \EE^\eps f_2 & = \xshort  \sum_{j,\,j'\sim j_{\GG^\eps}:\,|j-j'|\leq 1} \xshort  \EE^{\eps}(\varDelta_{j}^{\GG^\eps} f_1 \cdot  \varDelta_{j'}^{\GG^\eps} f_2) - \xshort \sum_{j,\,j'\gsim j_{\GG^\eps}:\,|j-j'|\leq 1} \xshort  \varDelta_j \EE^\eps f_1 \cdot   \varDelta_{j'} \EE^\eps f_2 \\
  &= \xshort  \sum_{j,\,j'\sim j_{\GG^\eps}:\,|j-j'|\leq 1} \xshort  \Big( \EE^{\eps}(\varDelta_{j}^{\GG^\eps} f_1 \cdot  \varDelta_{j'}^{\GG^\eps} f_2) -  \varDelta_j \EE^\eps f_1 \cdot   \varDelta_{j'} \EE^\eps f_2 \Big)\,,
 \end{align*}
 where we used in the second line that the spectral support of $\EE^\eps f_1$ and of $\EE^\eps f_2$ is contained in a ball of size $\eps^{-1} \approx 2^{j_{\GGe}}$ to reduce the sum in the second term to $j,\,j'\sim j_{\GG^\eps}$. Using then that the terms on the right hand side are spectrally supported in a ball of size $2^j$ we get for $i\geq -1$
\begin{align*}
 \varDelta_i( \EE^\eps (f_1 \reso^{\GG^\eps} f_2)-\EE^\eps f_1 \reso \EE^\eps f_2)= \xshort  \sum_{j,\,j'\sim j_{\GG^\eps}:\,|j-j'|\leq 1} \xshort \mathbf{1}_{i\lsim j}  \Big( \EE^{\eps}(\varDelta_{j}^{\GG^\eps} f_1 \cdot  \varDelta_{j'}^{\GG^\eps} f_2) -  \varDelta_j \EE^\eps f_1 \cdot   \varDelta_{j'} \EE^\eps f_2 \Big)\,,
\end{align*}
so that we obtain, using once more $\EE^\eps\colon L^p(\GG^\eps,\rho)\rightarrow L^p(\RR^d,\rho)$ and $\varDelta_i \colon L^p(\RR^d,\rho)\rightarrow L^p(\RR^d,\rho)$,
\begin{align*}
\|\varDelta_i( \EE^\eps (f_1 \reso^{\GG^\eps} f_2)-\EE^\eps f_1 \reso \EE^\eps f_2)\|_{L^p(\RR^d,\rho)} &\lesssim \xshort  \sum_{j,j'\sim j_{\GG^\eps}:\,|j-j'|\leq 1}  \xshort\mathbf{1}_{i\lsim j} \cdot  2^{-(j \alpha_1+j'\alpha_2)}   \\
 &\lesssim \mathbf{1}_{i\lsim j_{\GGe}} \cdot  2^{-j_{\GGe} (\alpha_1+\alpha_2-\kappa)} \eps^\kappa
 \lesssim  2^{-i(\alpha_1+\alpha_2-\kappa)} \eps^\kappa \,,
\end{align*}
where we chose $\kappa>0$ in the second line small enough so that $\alpha_1+\alpha_2-\kappa>0$. 
\end{proof}

\begin{lemma}
\label{lem:ExtensionDecay}
Let $\GG^\eps$ be as in Definition \ref{def:LatticeSequence}, $\omega\in \ww$ and construct Littlewood-Paley blocks as in Subsection \ref{subsec:discreteBesovspaces}. Let $\psi,\,\psi^\eps$ and $\EE^\eps$ be as in Subsection \ref{subsec:ExtensionOperator}. There is a $N=N(\GG,\psi)\in \NN$ such that for all $\eps$ and $-1\leq i\leq  j\leq j_{\GGe}-N$ and $f_1,\,f_2\in \Sww'(\GG^\eps)$
\begin{align*}
\EE^\eps(\varDelta_i^{\GGe} f_1 \cdot \varDelta_j^{\GGe} f_2) = \varDelta_i \EE^\eps f_1\cdot \varDelta_j \EE^\eps f_2\,.
\end{align*}
\end{lemma}
\begin{proof}
Let us fix $r_\eps:=\dist(\partial \wGGe,0)$ so that $B(0,r_\eps)\subseteq \wGGe$. From Lemma \ref{lem:SupportParaproduct} and the construction of our discrete partition of unity on page \pageref{DiscreteDyadicPartition} we know that the spectral support of $\varDelta_i^{\GGe} f_1 \cdot \varDelta_j^{\GGe} f_2$ and the support of $\varphi_i^{\GG^\eps} \cdot \FFge f_1 $ and $\varphi_j^{\GG^\eps} \cdot  \FFge f_2$ are contained in a set of the form $2^j \B\cap \wGGe$, where $\B$ is a ball around $0$. Choose $N\in \NN$ such that for $j$ with $-1\leq j \leq j_{\GGe}-N$ (if any) we have $2^j \B \subseteq 2^{j_{\GGe}-N} \B \subseteq B(0,r_\eps/4)$ (note that $N$ is independent of $\eps$ since $r_\eps=c\cdot 2^{j_{\GGe}}$ by the dyadic scaling of our lattice). In particular we have $2^j \B\subseteq \wGGe$, $2^j \B\cap \wGGe=2^j \B$. Choose $N$ further so big that we have for the smear function $\psi^\varepsilon$
\begin{align*}
\psi^\eps\vert_{2^j \B}=\psi(\eps\cdot )\vert_{2^j \B}=1\,,\qquad \supp \psi^\eps \cap (2^j \B+\CR^\eps\backslash\{0\}) =\varnothing
\end{align*}
for $-1\leq j\leq j_{\GGe}-N$ (independently of $\eps$). Choose a $\chi\in \Dww(\RR^d)$ such that $\chi\vert_{B(0,r_\eps/4)}=1$ and $\chi=0$ outside $B(0,r_\eps/2)$. We can then reshape
\begin{align*}
\FFr\EE^\eps(\varDelta_i^{\GGe} f_1 \cdot \varDelta_j^{\GGe} f_2)=  \psi^\eps \cdot  (\varphi_i^{\GG \eps} \FFge f_1 \ast_{\wGGe} \varphi_j^{\GGe} \FFge f_2 )_{\mathrm{ext}} = \chi \cdot  (\varphi_i^{\GG \eps} \FFge f_1 \ast_{\wGGe} \varphi_j^{\GGe} \FFge f_2 )_{\mathrm{ext}} \,,
\end{align*}
where we used the support properties above to replace $\psi^\eps$ by $\chi$. Now, note that (using formal notation to clarify the argument)
\begin{align}
\label{eq:ExtensionDecay1}
\chi(x) \cdot  (\varphi_i^{\GG \eps} \FFge f_1 \ast_{\wGGe} \varphi_j^{\GGe} \FFge f_2 )_{\mathrm{ext}}(x)=
\chi(x) \cdot  \int_{\wGGe} (\varphi_i^{\GG \eps} \FFge f_1)(z)\cdot (\varphi_j^{\GGe} \FFge f_2)([x -z])  \dd z\,.
\end{align}
Since only $x\in B(0,r_\eps/2)$ and $z\in B(0,r_\eps/4)$ contribute we have $x-z\in B(0,3/4 r)\subseteq \wGGe$ so that $[x-z]=x-z$ in \eqref{eq:ExtensionDecay1}. Using that $\supp \varphi_i^{\GGe}\cup\supp \varphi^{\GGe}_j\subseteq \wGGe$ we can replace $\varphi_i^{\GGe}$ and $\varphi_j^{\GGe}$ in \eqref{eq:ExtensionDecay1} by $\varphi_i,\,\varphi_j$ (the dyadic partition of unity on $\RR^d$ from which $\varphi_j^{\GGe}$ is constructed as on page \pageref{DiscreteDyadicPartition}), replace $\FFge f_1,\,\FFge f_2$ by their periodic extension and extend the integral to $\RR^d$ so that in total
\begin{align*}
 \FFr\EE^\eps(\varDelta_i^{\GGe} f_1 &\cdot \varDelta_j^{\GGe} f_2)(x) =\chi(x) \cdot  \int_{\RR^d} (\varphi_i (\FFge f_1)_{\mathrm{ext}})(z)\cdot (\varphi_j (\FFge f_2)_{\mathrm{ext}})(x -z)  \dd z\\
 &=\int_{\RR^d} (\varphi_i \psi^\eps (\FFge f_1)_{\mathrm{ext}})(z)\cdot (\varphi_j \psi^\eps (\FFge f_2)_{\mathrm{ext}})(x -z) \, \dd z \\
 &=\FFr(\varDelta_i\EE^\eps f_1\, \varDelta_j \EE^\eps f_2)(x) \,,
 \end{align*} 
 where we used in the second line that the support of the convolution is once more contained in $B(0,r_\eps/4)$ to drop $\chi$ and that $\psi^\eps\vert_{2^j \B}=1$ to introduce smear functions in the integral. The claim follows.
 \end{proof}

The main observation of~\cite{GIP} is that if the regularity condition $\alpha_1 + \alpha_2 > 0$ is not satisfied, then it may still be possible to make sense of $f_1 \reso f_2$ as long as $f_1$ can be written as a paraproduct plus a smoother remainder. The main lemma which makes this possible is an estimate for a certain ``commutator''. The discrete version of the commutator is defined as
\begin{align*}
C^{\GG}(f_1,f_2,f_3):=(f_1 \para^{\GG} f_2) \reso^{\GG} f_3 - f_1 (f_2\reso^{\GG} f_3)\,.
\end{align*}
If there is no risk for confusion we may drop the index $\GG$ on $C$.

\begin{lemma}(\cite[Lemma 14]{Gubinelli2015EBP})
\label{lem:CommutatorLemma}
Given $\rho_1,\,\rho_2,\,\rho_3\in \boldsymbol{\rho}(\omega)$, $p\in [1,\infty]$ and $\alpha_1,\alpha_2,\alpha_3\in \RR$ with $\alpha_1+\alpha_2+\alpha_3>0$ and $\alpha_2 + \alpha_3 \neq 0$ we have
	\begin{align*}
		\|C^\GG(f_1,f_2,f_3)\|_{\mathcal{C}_p^{\alpha_2+\alpha_3}(\GG^\eps,\rho_1 \rho_2 \rho_3)}\lesssim \|f_1\|_{\mathcal{C}^{\alpha_1}_p(\GG^\eps,\rho_1)} \|f_2\|_{\mathcal{C}^{\alpha_2}_\infty(\GG^\eps,\rho_2)} \|f_3\|_{\mathcal{C}_\infty^{\alpha_3}(\GG^\eps,\rho_3)} \,.
	\end{align*}
Further, property \eqref{eq:EProperty} holds for $C$ if the regularity on the left hand side is reduced by an arbitrary $\kappa>0$.
\end{lemma}
\begin{proof}
The proof of the estimates works line-by-line as in \cite[Lemma 14]{Gubinelli2015EBP} and the \eqref{eq:EProperty}-property follows as in Lemma~\ref{lem:ParaproductEstimates} via a modification of Lemma \ref{lem:ExtensionDecay} to three factors.
\end{proof}

\subsection{The Modified Paraproduct}

It will be useful to define a lattice version of the \emph{modified paraproduct} $\mpara$ that was introduced in \cite{GIP} and also used in \cite{Gubinelli2017KPZ,Chouk2017}.
\begin{definition}
\label{def:ModifiedParaproduct}
Fix a function $\varphi\in C_c^\infty((0,\infty);\RR_+)$ such that $\int_{\RR} \varphi(s) \dd s=1$ and define
\[
   Q_i f(t):=\int_{-\infty}^t 2^{2id} \varphi(2^{2i}(t-s)) f(s\vee 0) \dd s, \qquad  i\geq -1\,.
\]
We then set
\begin{align*}
 f_1 \mpara^\GG f_2  :=\sum_{-1\leq j_1,j_2\leq j_{\GG}:\, j_1< j_2-1} Q_{j_2}\varDelta_{j_1}^\GG f_1 \cdot \varDelta_{j_2}^\GG f_2
\end{align*}
for $f_1,\, f_2\colon \RR_+ \rightarrow \Sw_\omega'(\GG)$ where this is well defined. We may drop the index $\GG$ if there is no risk for confusion.
\glsadd{mpara}
\end{definition}

\begin{convention}
\label{con:ConventionModifiedParaproduct}
As in \cite{Gubinelli2017KPZ} we silently identify $f_1$ in $f_1 \mpara f_2$ with $t\mapsto f(t) \mathbf{1}_{t>0}$ if $f_1\in \mathcal{M}^\gamma_T \mathcal{C}^\alpha_p(\GG,\rho)$ with $\gamma>0$.
\end{convention}
 Once more the translation to the continuous case $f_1,f_2 \colon \RR_+\rightarrow  \Sw_\omega'(\RR^d)$ is analogous. The modified paraproduct allows for similar estimates  as in Lemma \ref{lem:ParaproductEstimates}.
\begin{lemma}
\label{lem:ModifiedParaproductEstimates}
Let $\beta \in \RR,\, p\in [1,\infty],\,\gamma\in [0,1),\,t>0,\alpha<0$ and let $ \rho_1,\rho_2\colon \RR_+ \to \boldsymbol{\rho}(\omega)$ with $\rho_1$ pointwise decreasing. Then
\begin{align*}
	& t^\gamma \|f \mpara g (t)\|_{\mathcal{C}^{\alpha+\beta}_p(\GG^\eps, \rho_1(t) \rho_2(t))}\lesssim \|f\|_{\mathcal{M}^\gamma_t \mathcal{C}^\alpha_p (\GG^\eps,\rho_1)} \|g(t)\|_{\mathcal{C}^\beta_\infty(\GG^\eps,\rho_2(t))} \wedge \|f\|_{\mathcal{M}^\gamma_t \mathcal{C}^\alpha_\infty(\GG^\eps,\rho_1)} \|g(t)\|_{\mathcal{C}^\beta_p(\GG^\eps,\rho_2(t))} 
\end{align*}
and 
\begin{align*}
	&t^\gamma \|f \mpara g (t)\|_{\mathcal{C}^{\beta}_p(\GG^\eps,\rho_1(t) \rho_2(t))}  \lesssim  \|f\|_{\mathcal{M}^\gamma_t L^p (\GG^\eps,\rho_1)} \|g(t)\|_{\mathcal{C}^\beta_\infty(\GG^\eps,\rho)} \wedge \|f\|_{\mathcal{M}^\gamma_t L^\infty(\GG^\eps,\rho_1)} \|g(t)\|_{\mathcal{C}^\beta_p(\GG^\eps,\rho_2(t))} \,.
\end{align*}
Both estimates have the property \eqref{eq:EProperty} if the regularity on the left hand side is decreased by an arbitrary $\kappa>0$.
\end{lemma}
\begin{proof}
The proof is the same as for \cite[Lemma 6.4]{Gubinelli2017KPZ}. Property \eqref{eq:EProperty} is shown as in Lemma~\ref{lem:ParaproductEstimates}.
\end{proof}

We further have an estimate in terms of the parabolic spaces $\mathscr{L}^{\gamma,\alpha}_{p,T}(\GG,\rho)$ that were introduced in Definition \ref{def:WeightedSpaces}.

\begin{lemma}
\label{lem:ParaproductEstimateParabolicSpace}
We have for $\alpha\in (0,2), \,p\in [1,\infty],\,\gamma\in [0,1)$ and $\rho_1,\rho_2\colon \RR_+ \to \boldsymbol{\rho}(\omega)$, pointwise decreasing in $s$, the estimate
\begin{align*}
	\|f\mpara g\|_{\mathscr{L}^{\gamma,\alpha}_{p,T}(\GG^\eps,\rho_1\rho_2)} \lesssim \|f\|_{\mathscr{L}^{\gamma,\delta}_{p,T}(\GG^\eps,\rho_1)} \,(\|g\|_{C_T\mathcal{C}_\infty^\alpha(\GG^\eps,\rho_2)}+\|\mathscr{L}^\eps g\|_{C_T\mathcal{C}^{\alpha-2}_\infty(\GG^\eps,\rho_2)})
\end{align*}
for any $\delta>0$ and any diffusion operator $\mathscr{L}_\mu^\eps$ as in Definition \ref{def:DiffusionOperator}. This estimate has the property \eqref{eq:EProperty} if the regularity $\alpha$ on the left hand side is lowered by an arbitrary $\kappa>0$.
\end{lemma}
\begin{proof}
The proof is as in \cite[Lemma 6.7]{Gubinelli2017KPZ} and uses Lemma \ref{lem:CommuteModifiedParaproduct} below. The proof of the property~\eqref{eq:EProperty} is as in Lemma~\ref{lem:ParaproductEstimates}.
\end{proof}

The main advantage of the modified paraproduct $\mpara$ on $\RR^d$ is its commutation property with the heat kernel $\partial_t -\Delta$ (or $\mathscr{L}_\mu=\partial_t -L_\mu$) which is essential for the Schauder estimates for paracontrolled distributions, compare also Subsection \ref{subsec:ConvergenceOfTheLatticeModel} below. In the following we state the corresponding results for Bravais lattices. 

\begin{lemma}
\label{lem:CommuteModifiedParaproduct}
For $\alpha\in (0,2)$, $\beta\in \RR$, $p\in [1,\infty]$, $\gamma\in [0,1)$ and $\rho_1,\rho_2\colon \RR_+ \to \boldsymbol{\rho}(\omega)$, with $\rho_1$ pointwise decreasing, we have for $t>0$
\begin{align*}
	t^\gamma \|(f\mpara g - f\para g)(t) \|_{\mathcal{C}^{\alpha+\beta}_p(\GG^\eps,\rho_1(t)\rho_2(t))} \lesssim \|f\|_{\mathscr{L}^{\gamma,\alpha}_{p,t}(\GG^\eps,\rho_1)}  \|g(t)\|_{\mathcal{C}^\beta_\infty(\GG^\eps,\rho_2(t))} 
\end{align*}
and 
\begin{align*}
	t^\gamma \|(\mathscr{L}_\mu^\eps(f\mpara g) - f\mpara \mathscr{L}_\mu^\eps g)(t) \|_{\mathcal{C}^{\alpha+\beta-2}_p(\GG^\eps,\rho_1(t)\rho_2(t))} \lesssim \|f\|_{\mathscr{L}^{\gamma,\alpha}_{p,t}(\GG^\eps,\rho_1)}  \|g(t)\|_{\mathcal{C}^\beta_\infty(\GG^\eps,\rho_2(t))}\,. 
\end{align*}
where $\mathscr{L}_\mu^\eps=\partial_t-L^\eps_\mu$ is a discrete diffusion operator as in 
Definition \ref{def:DiffusionOperator}. These estimates have the property \eqref{eq:EProperty} if the regularity on the left hand side is lowered by an arbitrary $\kappa>0$.
\end{lemma}

\begin{proof}
Again we can almost follow along the lines of the proof in \cite[Lemma 6.5]{Gubinelli2017KPZ} with the only difference that in the derivation of the second estimate the application of the ``product rule'' of $\mathscr{L}^\varepsilon_\mu$ does not yield a term $-2\nabla f \mpara \nabla g$ but a more complex object, namely 
\begin{align}
\label{eq:NablaModifiedParaproduct}
\int_{\RR^d}\frac{\dd \mu(y)}{\varepsilon^2} \,\, D^\eps_y f \mpara D^\eps_y g \,,
\end{align}
where $D^\eps_y f(t,x)=f(t,x+\varepsilon y)-f(t,x)$ and similarly for $g$. The bound for \eqref{eq:NablaModifiedParaproduct} follows from Lemma \ref{lem:ModifiedParaproductEstimates} once we show
\begin{align}
\label{eq:ProofDerivativeModifiedParaproductCore}
	\|D^\eps_y \varphi \|_{\mathcal{C}^{\gamma-1}_p(\GG^\varepsilon,\rho_1)}\lesssim 	\|\varphi\|_{\mathcal{C}^\gamma_p(\GG^\varepsilon,\rho_1)}\, |y|\cdot  \varepsilon
\end{align}
for any $\gamma\in \RR$. Note that due to Lemma \ref{lem:ScalingBoundary} we can write
\begin{align*}
	\varDelta_j D^y_\eps \varphi=\left(\tilde{\Psi}^{\eps,j}(\cdot+\varepsilon y) - \tilde{\Psi}^{\eps,j} \right)\aste \varphi \,,
\end{align*}
where $\tilde{\Psi}^{\eps,j}=\mathcal{E}^\eps \Psi^{\GG^\eps,j}=2^{jd}\phi_{\je}(2^j\cdot )$ with $\phi_{\je}\in \Sw_\omega(\RR^d)$.
With
\begin{align*}
	\tilde{\Psi}^{\eps,j}(x+\varepsilon y) - \tilde{\Psi}^{\eps,j}(x)=2^{j} \int_0^1  2^{jd} \phi_{\je}(2^j (x+\zeta \varepsilon y)) \,\dd \zeta\cdot y \varepsilon
\end{align*}
we get \eqref{eq:ProofDerivativeModifiedParaproductCore} by applying Lemma \ref{lem:DiscreteWeightedYoungInequality}. The proof of the property~\eqref{eq:EProperty} is as in Lemma~\ref{lem:ParaproductEstimates} and it uses Lemma~\ref{lem:LaplaceRegularity}.
\end{proof}

\section{Weak universality of PAM on $\RR^2$} 
\label{sec:PAM}

With the theory from the previous sections at hand we can analyze stochastic models on unbounded lattices using paracontrolled techniques. As an example, we prove the weak universality result for the linear parabolic Anderson model that we discussed in the introduction.
 For $F\in C^2(\RR;\RR)$ with $F(0)=0$ and bounded second derivative we consider the equation
\begin{align}
\label{eq:PAMonNonScaledLattice}
	\mathscr{L}^1_\mu v^\eps=F(v^\eps)\cdot \eta^\eps,\qquad v^\eps(0)=|\GG|^{-1}\mathbf{1}_{\cdot=0}
\end{align}
on $\RR_+ \times \GG$, where $\GG\subseteq \RR^2$ is a two-dimensional Bravais lattice, $\mathscr{L}^1_\mu=\partial_t-L^1_\mu$ is a discrete diffusion operator on the lattice $\GG$ as described in Definition \ref{def:DiffusionOperator}, induced by $\mu\in \mm(\omega)$ with $\omega=\wexps$ for $\sigma\in(0,1)$. The upper index ``$1$'' indicates that we \emph{did not scale} the lattice $\GG$ yet. The family $(\eta^\eps(z))_{z\in \GG}\in \Sw_\omega'(\GG)$ consists of independent (not necessarily identically distributed) random variables satisfying for $z\in\GG$
\begin{align*}
\mathbb{E}[\eta^\eps(z)]=-F'(0) c^\eps_\mu\eps^2\,,\qquad \mathrm{Var}\big(\eta^\eps(z)\big)=\frac{1}{|\GG^\eps|}= \frac{1}{|\GG|} \,\eps^2\,,
\end{align*}
where $c^\eps_\mu>0$ is a constant of order $O(|\log \eps|)$ which we will fix in equation~\eqref{eq:renorm-const} below. We further assume that for every $\eps$ and $z\in \GG$ the variable $\eta^\eps(z)$ has moments of order $p_\xi>14$ such that
\begin{align*}
	\sup_{z \in \GG^\varepsilon} \mathbb{E}\big[|\eta^\eps(z)-\mathbb{E}[\eta^\eps(z)]|^{p_\xi}\big]\lesssim \eps^{p_\xi}\,.
\end{align*}
The lower bound $14$ for $p_\xi$ might seem quite arbitrary at the moment, we will explain this choice in Remark \ref{rem:fourteen} below. Note that $\eta^\eps$ is of order $O(\eps)$ while its expectation is of order $O(\eps^2 |\log \eps|)$, so we are considering a small shift away from the ``critical'' expectation $0$. 

We are interested in the behavior of \eqref{eq:PAMonNonScaledLattice} for large scales in time and space. Setting $u^\eps(t,x):=\eps^{-2} v^\eps(\eps^{-2} t,\eps^{-1} x)$ and $\xi^\eps(x):=\eps^{-2} (\eta^\eps(\eps^{-1} x )+F'(0) c^\eps_\mu\eps^2   )$ modifies the problem to
\begin{align}
\label{eq:DiscretePAM}
	\mathscr{L}^\eps_\mu u^\eps=F^\eps(u^\eps)(\xi^\eps -F'(0) c^\eps_\mu),\qquad u^\eps(0)=|\GG^\eps|^{-1}\mathbf{1}_{\cdot=0}\,,
\end{align}
where $u^\eps \colon \RR_+ \times \GG^\eps \to \RR$ is defined on refining lattices $\GG^\eps$ in $d=2$ as in Definition \ref{def:LatticeSequence} and where $F^\eps:=\eps^{-2} F(\eps^2\cdot )$. The potential $(\xi^\eps(x))_{x\in \GG^\eps}$ is scaled so that it satisfies for $z\in \GG^\eps$
\begin{itemize}
	\item $\mathbb{E}[\xi^\eps(z)]=0$,
	\item $ \mathbb{E}\left[ |\xi^\eps(z)|^2  \right]= |\GG^\eps|^{-1} = |\GG|^{-1} \eps^{-2}$,
	\item $\sup_{z \in \GG^\eps} \mathbb{E}\left[ |\xi^\eps(z)|^{p_\xi} \right]\lesssim \eps^{- p_\xi}$ for some $p_\xi > 14$.
\end{itemize}
We consider $\xi^\eps$ as a discrete approximation to white noise in dimension $2$.  In particular, we expect $\EE^\eps \xi^\eps$ to converge in distribution to white noise on $\RR^2$, and we will see in Lemma \ref{lem:LimitArea} below that this is indeed the case. In Theorem \ref{thm:ConvergencePAM} we show that $\EE^\eps u^\eps$ converges in distribution to the solution $u$ of the linear parabolic Anderson model on $\RR^2$,
\begin{align}
\label{eq:GlobalPAM}
	\mathscr{L}_\mu u=F'(0) u (\xi - F'(0)\infty),\qquad u(0) = \delta,
\end{align}
where $\xi$ is white noise on $\RR^2$, $\delta$ is the Dirac delta distribution, ``$-\infty$'' denotes a renormalization and $\mathscr{L}_\mu$ is the limiting operator from Definition \ref{def:DiffusionOperator}. The existence and uniqueness of a solution to \eqref{eq:GlobalPAM} were first established in~\cite{HairerLabbeR2} (for more regular initial conditions) by using a ``partial Cole-Hopf transformation'' which turns the equation into a well-posed PDE. Using the continuous versions of the objects defined in the Sections above we can modify the arguments of~\cite{GIP} to give an alternative proof of their result, see Corollary~\ref{cor:ContinuousSolution} below. The limit of \eqref{eq:DiscretePAM} only sees $F'(0)$ and forgets the structure of the non-linearity $F$, so in that sense the linear parabolic Anderson model arises as a universal scaling limit.

Let us illustrate this result with a (far too simple) model: Suppose $F$ is of the form $F(v)=v(1-v)$ and let us first consider the following ordinary differential equation on $[0,T]$:
\begin{align*}
\partial_t v=\eta \cdot F(v),\qquad v(0)\in (0,1)\,,
\end{align*}
for some $\eta\in \RR$. If $\eta>0$, then $v$ describes the evolution of the concentration of a growing population in a pleasant environment, which however shows some saturation effects represented by the factor $(1-v)$ in the definition of $F$. For $\eta<0$ the individuals live in unfavorable conditions,  say in competition with a rival species. From this perspective equation \eqref{eq:PAMonNonScaledLattice} describes the dynamics of a population that migrates between diverse habitats. The meaning of our universality result is that if we tune down the random potential $\eta^\eps$ and counterbalance the growth of the population with some renormalization (think of a death rate), then from far away we can still observe its growth (or extinction) without feeling any saturation effects.

The analysis of~\eqref{eq:DiscretePAM} and the study of its convergence are based on the lattice version of paracontrolled distributions that we developed in the previous sections and it will be given in Subsection~\ref{sec:convergence} below. In that analysis it will be important to understand the limit of $\EE^\eps \xi^\eps$ and a certain bilinear functional built from it, and we will also need uniform bounds in suitable Besov spaces for these objects. In the following subsection we discuss this convergence.

\subsection{Discrete Wick calculus and convergence of the enhanced noise}
\label{subsec:ConvergenceArea}

We develop here a general machinery for the use of discrete Wick contractions in the renormalization of discrete, singular SPDEs with i.i.d. noise which is completely analogous to the continuous Gaussian setting. Moreover, we build on the techniques of~\cite{Caravenna} to provide a criterion that identifies the scaling limits of discrete Wick products as multiple Wiener-It\^o integrals. Our results are summarized in Lemma~\ref{lem:DiscreteItoIsometry} and Lemma~\ref{lem:ChaosExpansionConvergence} below and although the use of these results is illustrated only on the discrete parabolic Anderson model, the approach extends in principle to any discrete formulation of popular singular SPDEs such as the KPZ equation or the $\Phi^4_d$ models. In order to underline the general applicability of these methods we work in this subsection in a general dimension $d$.

Take a sequence of scaled Bravais lattices $\GG^\eps$ in dimension $d$ as in Definition \ref{def:LatticeSequence}. As a \emph{discrete approximation to white noise} we take independent (but \emph{not necessarily identically distributed}) random variables $\big(\xi^\eps(z)\big)_{z\in \GG^\eps}$ that satisfy
\begin{itemize}
	\item $\mathbb{E}[\xi^\eps(x)]=0$,
	\item $ \mathbb{E}\left[ |\xi^\eps(x)|^2  \right]= |\GG^\eps|^{-1} = |\GG|^{-1} \eps^{-d}$,
	\item $\sup_{z \in \GG^\eps} \mathbb{E}\left[ |\xi^\eps(z)|^{p_\xi} \right]\lesssim \eps^{- d/2\cdot  p_\xi}$ for some $p_\xi \geq 2$.
\end{itemize}
\label{ApproximationToWhiteNoise}
Note that the family $(\xi^\eps(z))_{z\in \GG^\eps}$ we defined in the introduction of this Section  fits into this framework (with $d=2$ and $p_\xi>14$). 

Let us fix a symmetric $\chi\in \Dw_\omega(\RR^d)$, independent of $\eps$, which is $0$ on $\frac{1}{4}\cdot \widehat{\GG}$ and $1$ outside of $\frac{1}{2} \cdot \widehat{\GG}$ and define
\begin{align*}
	X_\mu^\eps := \frac{\chi}{l^\eps_\mu}(D_{\GG^\eps})  \xi^\eps:=\FFge^{-1} \left( \frac{\chi}{l^\eps_\mu} \cdot \FFge \xi^\eps \right)\,.
\end{align*}   
Let us point out that the $\chi$ used in the construction of $X_\mu^\eps$ does \emph{not} depend on $\eps$ and only serves to erase the ``zero-modes'' of $\xi^\eps$ to avoid integrability issues. Note that $\mathscr{L}_\mu^\eps X_\mu^\eps = -L^\eps_\mu X_\mu^\eps=\chi(D_{\GG^\eps})\xi^\eps=\FFge^{-1}(\chi \cdot \FFge \xi^\eps)$ so that $X_\mu^\eps$ is a time independent solution to the heat equation on $\GG^\eps$ driven by $\chi(D_{\GG^\eps})\xi^\eps$. Our first task will be to measure the regularity of the sequences $(\xi^\eps)$, $(X_\mu^\eps)$ in terms of the discrete Besov spaces introduced in Subsection~\ref{subsec:discreteBesovspaces}. For that purpose we need to estimate moments of sufficiently high order. For  discrete multiple stochastic integrals with respect to the variables $(\xi^\eps(z))_{z \in \GG^\eps}$, that is for sums $\sum_{z_1,\ldots,z_n \in \GG^\eps}  f(z_1,\ldots,z_n)\, \xi^\eps(z_1)  \ldots  \xi^\eps(z_n)$ with $f(z_1,\dots, z_n) = 0$ whenever $z_i = z_j$ for some $i \neq j$ it was shown in~\cite[Proposition~4.3]{Chouk2017} that all moments can be bounded in terms of the $\ell^2$ norm of $f$ and the corresponding moments of the $(\xi^\eps(z))_{z\in \GG^\eps}$. However, typically we will have to bound such expressions for more general $f$ (which do not vanish on the diagonals) and in that case we first have to arrange our random variable into a finite sum of discrete multiple stochastic integrals, so that then we can apply~\cite[Proposition~4.3]{Chouk2017} for each of them. This arrangement can be done in several ways, here we follow~\cite{Hairer2015Central} and regroup in terms of Wick polynomials.

Given random variables $(Y(j))_{j \in J}$ over some index set $J$ and $I = (j_1,\dots,j_n) \in J^n$ we set
\[
   Y^I = Y(j_1) \dots Y(j_n) = \prod_{k=1}^n Y(j_k)
\]
as well as $Y^\varnothing=1$. According to Definition~3.1 and Proposition~3.4 of \cite{LukkarinenMarcozzi}, the \emph{Wick product} $Y^{\diamond I}$ 
 can be defined recursively by $Y^{\diamond \varnothing} := 1$ and
\begin{equation}\label{eq:wick-def}
   Y^{\diamond I} := Y^I - \sum_{\varnothing \neq E \subset I} \mathbb{E}[Y^E] \cdot  Y^{\diamond \,I \setminus E}\,.
\end{equation}
For $I = (j_1,\dots,j_n) \in J^n$ we also write 
\begin{align*}
Y(j_1)\diamond \dots \diamond Y(j_n):= Y^{\diamond I}\,.
 \end{align*} 
 By induction one easily sees that this product is commutative. In the case $j_1=\ldots=j_n$ we may write instead
 \begin{align*}
 Y(j_1)^{\diamond n}\,.
 \end{align*}

\glsadd{diamond}

\begin{lemma}[see also Proposition~4.3 in~\cite{Chouk2017}]
\label{lem:DiscreteItoIsometry}
Let $\GGe$ be as in Definition \ref{def:LatticeSequence} and let $\big(\xi^\eps(z)\big)_{z\in \GG^\eps}$ be a discrete approximation to white noise as above, $n\geq 1$ and assume $p_\xi\geq 2 n$.
For $f \in L^2((\GG^\varepsilon)^n)$ define the \emph{discrete multiple stochastic integral} w.r.t $\big(\xi^\eps(z)\big)$ by
\begin{align*}
\mathscr{I}_nf:=\sum_{z_1,\ldots,z_n \in \GG^\varepsilon} |\GG^\varepsilon|^n \, f(z_1,\ldots,z_n) \, \xi^\varepsilon(z_1)\diamond  \ldots  \diamond \xi^\varepsilon(z_n)\,.
\end{align*}
It then holds for $2\leq p\leq p_\xi/n$

\[
	\left\|  \mathscr{I}_n f\right\|_{L^{p}(\mathbb{P})}   \lesssim \|f\|_{L^2((\GG^\varepsilon)^n)}\,.
\]
\end{lemma} 
\begin{proof}
In the following we identify $\GG^\varepsilon$ with an enumeration by $\NN$ so that we can write
\begin{align*}
\mathscr{I}_n f =\sum_{1\leq r\leq n,\,a\in A_r^n} r! \sum_{z_1<\ldots<z_r}  |\GG^\varepsilon|^n \tilde{f}_a(z_1,\ldots,z_r)\cdot\xi^\varepsilon(z_1)^{\diamond {a_1}} \times \ldots \times \xi^\varepsilon(z_r)^{\diamond {a_r}}\,,
\end{align*}
where $A_r^n:=\{a\in\mathbb{N}^r \vert \sum_i a_i=n  \}$, $\tilde{f}_a$ denotes the symmetrized version of 
\begin{align*}
f_a(z_1,\ldots,z_r):=f(\overbrace{z_1,\ldots,z_1}^{a_1 \times},\ldots,\overbrace{z_r,\ldots,z_r}^{a_r \times})\cdot \mathbf{1}_{z_i\neq z_j \, \forall i\neq j}\,,
\end{align*}
and where we used the independence of $\xi^\varepsilon(z_1),\ldots,\xi^\varepsilon(z_r)$ to decompose the Wick product (we did not show this property, but it is not hard to derive it from the definition of $\diamond$ we gave above). The independence and the zero mean of the Wick products allow us to see this as a sum of nested martingale transforms so that an iterated application of the Burkholder-Davis-Gundy inequality and Minkowski's inequality as in \cite[Proposition 4.3]{Chouk2017} gives the desired estimate
\begin{align*}
\left\|\mathscr{I}_n f \right\|^2_{L^{p}(\mathbb{P})}  &\lesssim \sum_{1\leq r\leq n,\,a\in A_r^n} \left\| \sum_{z_1<\ldots<z_r} |\GG^\varepsilon|^n \cdot \tilde{f}_a(z_1,\ldots,z_r)\cdot \xi^\varepsilon(z_1)^{\diamond{a_1}}\times \ldots \times \xi^\varepsilon(z_r)^{\diamond{a_r}}\right\|_{L^p(\mathbb{P})}^2 \\
&\lesssim \sum_{1\leq r\leq n,\,a\in A_r^n}  \sum_{z_1<\ldots<z_r}  |\GG^\varepsilon|^{2n}\cdot 
 |\tilde{f}_a(z_1,\ldots,z_r)|^2 \cdot \prod_{j=1}^r \|\xi^\varepsilon(z_j)^{\diamond{a_j}}\|_{L^p(\mathbb{P})}^2 \\
 &\lesssim \sum_{1\leq r\leq n,\,a\in A_r^n}  \sum_{z_1,\ldots,z_r} |\GG^\varepsilon|^n |\tilde{f}_a(z_1,\ldots,z_r)|^2 \leq \|f\|_{L^2((\GG^\varepsilon)^{n})}^2\,,
\end{align*}
where we used the bound $\|\xi^\varepsilon(z_r)^{\diamond{a_j}}\|_{L^p(\mathbb{P})}^2\lesssim |\GG^\varepsilon|^{-a_j}$ which follows from~\eqref{eq:wick-def} and our assumption on $\xi^\eps$.
\end{proof}

As a direct application we can bound the moments of $\xi^\eps$ and $X_\mu^\eps$ in Besov spaces. We also need to control the resonant term $X_\mu^\eps \reso \xi^\eps$, for which we introduce the renormalization constant
\begin{equation}\label{eq:renorm-const}
	c^\eps_\mu:=\int_{\widehat{\GG^\eps}} \frac{\chi(x)}{l^\eps_\mu(x)} \,\dd x\,,
\end{equation}
which is finite for all $\eps > 0$ because $\widehat{\GG^\eps}$ is compact and $\chi$ is supported away from $0$. We define a \emph{renormalized} resonant product by
\glsadd{bullet}
\begin{align*}
X^\eps_\mu\bullet \xi^\varepsilon:=X^\eps_\mu \reso \xi^\varepsilon-c^\varepsilon_\mu\,.
\end{align*}

\begin{remark}
Since $l^\varepsilon_\mu \approx \lvert\cdot\rvert^2$ (Lemma \ref{lem:SchauderPrelude} together with the easy estimate $l^\varepsilon_\mu \lesssim \lvert\cdot\rvert^2$) we have $c^\varepsilon_\mu \approx -\log \varepsilon$ in dimension 2.  
\end{remark}

Using Lemma \ref{lem:DiscreteItoIsometry} we can derive the following bounds. 

\begin{lemma}
\label{lem:RegularityNoise}
Let $\xi^\eps,\,X^\eps$ and $X^\eps_\mu\bullet \xi^\varepsilon$ be defined on $\GG^\eps$ as above with $p_\xi\geq 4$ (where $p^\xi$ is as on page \pageref{ApproximationToWhiteNoise}) and let $d<4$. For $\mu\in \mm(\ww),\,\zeta<2-d/2-d/p_\xi$ and $\kappa>d/p_\xi$ we have
\begin{align}
\label{eq:RegularityNoise}
	\mathbb{E} \left[ \,\|\xi^\varepsilon\|^{p_\xi}_{\mathcal{C}^{\zeta-2}(\GG^\varepsilon,\w^{\kappa})} \right] +	\mathbb{E} \left[ \,\|X^\eps_\mu\|^{p_\xi}_{\mathcal{C}^{\zeta}(\GG^\varepsilon,\w^{\kappa})} \right] + \mathbb{E}\left[ \| X^\eps_\mu \bullet \xi^\varepsilon\|^{p_\xi/2}_{\mathcal{C}^{2\zeta -2}(\GG^\varepsilon,\w^{2\kappa})}  \right] &\lesssim 1 \,.
\end{align}
The implicit constant is independent of $\eps$.
\end{lemma}
\begin{proof}
Let us bound the regularity of $X^\eps_\mu$. Recall that by Lemma~\ref{lem:DiscreteEmbedding} we have the continuous embedding (with norm uniformly bounded in $\varepsilon$) $\mathcal{B}^{\zeta+d/p_\xi}_{p_\xi,p_\xi}(\GG^\varepsilon,\w^{\kappa})\subseteq \mathcal{C}^{\zeta}(\GG^\varepsilon,\w^{\kappa})$. To show \eqref{eq:RegularityNoise} it is therefore sufficient to bound for $\beta<2-d/2$
\begin{align*}
	\mathbb{E}\left[ \|X^\eps_\mu\|^{p_\xi}_{\mathcal{B}^\beta_{p_\xi,p_\xi}(\GG^\varepsilon, \w^{\kappa})} \right]=\sum_{-1\le j \le j_{\GG^\varepsilon}} 2^{j{p_\xi} \beta } \sum_{z \in \GG^\varepsilon} |\GG^\varepsilon| \,\mathbb{E}[|\varDelta_j^{\GG^\varepsilon} X^\eps_\mu(z)|^{p_\varepsilon} ] \frac{1}{(1+|z|)^{\kappa p_\xi}}\,.
\end{align*}
By assumption we have $\kappa p_\xi > d$ and can bound $\sum_{z \in \GG^\varepsilon} |\GG^\varepsilon| (1+|z|)^{-\kappa p_\xi}\lesssim 1$ uniformly in $\varepsilon$ (for example by Lemma \ref{lem:IntegralTest}). It thus suffices to derive a bound for $\mathbb{E}[|\varDelta_j^{\GG^\varepsilon} X^\eps_\mu(x)|^{p_\varepsilon} ]$, uniformly in $\varepsilon$ and $x$. Note that by \eqref{eq:ConvolutionLattice} $	\varDelta_j^{\GG^\varepsilon} X^\eps_\mu(x)=\sum_{z\in\GG^\varepsilon} |\GG^\varepsilon|\, \mathscr{K}^{\varepsilon}_j(x-z) \xi^\varepsilon(z)$ with $\mathscr{K}^\varepsilon_j=\FFge^{-1} (\varphi^{\GG^\varepsilon}_j \chi/l^\varepsilon_\mu)$ so that Lemma \ref{lem:DiscreteItoIsometry}, Parseval's identity  \eqref{eq:ParsevalsIdentity} and $l^\varepsilon_\mu\gtrsim \lvert\cdot\rvert^2$ on $\wGGe$ (from Lemma \ref{lem:SchauderPrelude}) imply
\begin{align*}
\mathbb{E}[ |\varDelta_j^{\GG^\eps} X^\eps_\mu (x)|^{p_\xi}] \lesssim \|\mathscr{K}_j^\varepsilon(x - \cdot) \|_{L^2(\GG^\varepsilon)}^{p_\xi} \lesssim 2^{jp_\xi (d/2-2)}\,,
\end{align*}
which proves the bound for $X^\eps_\mu$. The bound for $\xi^\varepsilon$ follows from the same arguments or with Lemma \ref{lem:LaplaceRegularity}.

Now let us turn to $X^\eps_\mu \bullet \xi^\varepsilon$. A short computation shows that
\[
   \mathbb{E}[(X^\eps_\mu\reso \xi^\varepsilon)(x)] = \mathbb{E}[(X^\eps_\mu \cdot  \xi^\varepsilon)(x)] = c^\varepsilon_\mu, \qquad x \in \GG^\varepsilon\,,
\]
and, by a similar argument as above, it suffices to bound $X^\eps_\mu \bullet \xi^\varepsilon$ in $\mathcal{B}^{\beta}_{p_\xi/2,p_\xi/2}(\RR^d,\w^{2\kappa})$ for $\beta<2-d$. We are therefore left with the task of bounding the $p_\xi/2$-th moment of 
\begin{align*}
	&\varDelta^{\GG^\varepsilon}_k \left( \sum_{|i-j|\le 1} \varDelta^{\GG^\eps}_i X^\eps_\mu \varDelta^{\GG^\eps}_j\xi^\varepsilon -\mathbb{E}[\varDelta^{\GG^\eps}_i X_\mu^\varepsilon \varDelta^{\GG^\eps}_j\xi^\varepsilon] \right)(x) \\
	&\hspace{15pt} =\sum_{z_1,z_2,y} |\GG^\varepsilon|^3\,\sum_{|i-j|\le 1} \Psi^{\GG^\varepsilon, k}(x-y) \mathscr{K}_i^\varepsilon(y-z_1) \Psi^{\GG^\varepsilon,j} (y-z_2) \left(\xi^\varepsilon(z_1)\xi^\varepsilon(z_2)-\mathbb{E}[\xi^\varepsilon(z_1)\xi^\varepsilon(z_2)]\right) \\
	&\hspace{15pt} = \sum_{z_1,z_2} |\GG^\varepsilon|^2 \left(  \sum_{|i-j|\le1}  \sum_y |\GG^\varepsilon|\, \Psi^{\GG^\varepsilon, k}(x-y) \mathscr{K}_i^\varepsilon(x-z_1) \Psi^{\GG^\varepsilon,j}(x-z_2) \right)\xi^\varepsilon(z_1)\diamond \xi^\varepsilon(z_2) \,,
\end{align*}
which with Lemma~\ref{lem:DiscreteItoIsometry} and Parseval's identity~\eqref{eq:ParsevalsIdentity} can be estimated by
\begin{align*}
	&\mathbb{E}\left[\left|\sum_{z_1,z_2} |\GG^\varepsilon|^2 \left( \sum_{|i-j|\le1} |\GG^\varepsilon|\, \Psi^{\GG^\varepsilon,k}(x-y) \mathscr{K}_i^\varepsilon(x-z_1) \Psi^{\GG^\varepsilon,j}(x-z_2) \right)\xi^\varepsilon(z_1)\diamond \xi^\varepsilon(z_2)\right|^{p_\xi/2}\right]^{2/p_\xi} \\
	&\hspace{20pt} \lesssim \left\|  \sum_{|i-j|\le1}  \sum_y|\GG^\varepsilon|\, \Psi^{\GG^\varepsilon,k}(x-y) \mathscr{K}_i^\varepsilon(x-z_1) \Psi^{\GG^\varepsilon,j}(x-z_2) \right\|_{L^2_{z_1,z_2}((\GG^\varepsilon)^2)} \\
	&\hspace{20pt} = \left\| \sum_{|i-j|\le1} \sum_y |\GG^\varepsilon|\, \Psi^{\GG^\varepsilon,k}(x-y)\FF_{(\GG^\varepsilon)^2} \big( \mathscr{K}_i^\varepsilon(x- \cdot)\otimes \Psi^{\GG^\varepsilon,j}(x-\cdot) \big)(\ell_1, \ell_2) \right\|_{L^2_{\ell_1,\ell_2}((\widehat{\GG}^\varepsilon)^2)} \\
	&\hspace{20pt} = \left\| e^{-2\pi \imath (\ell_1 + \ell_2) \cdot x} \sum_{|i-j|\le1} \FF_{\GG^\varepsilon} \Psi^{\GG^\varepsilon,k}(-(\ell_1+\ell_2)) \FF_{\GG^\varepsilon} \mathscr{K}_i^\varepsilon(-\ell_1) \FF_{\GG^\varepsilon}\Psi^{\GG^\varepsilon,j}(-\ell_2) \right\|_{L^2_{\ell_1,\ell_2}((\widehat{\GG}^\varepsilon)^2)} \\
	&\hspace{20pt} = \left\| \sum_{|i-j|\le1}  \varphi^{\GG^\varepsilon}_k (\ell_1+\ell_2) \frac{\varphi^{\GG^\varepsilon}_i(\ell_1) \chi(\ell_1)}{l^\varepsilon_\mu(\ell_1)} \varphi^{\GG^\varepsilon}_j(\ell_2) \right\|_{L^2_{\ell_1,\ell_2}((\widehat{\GG}^\varepsilon)^2)}\,,
\end{align*}
where in the last step we used that all considered functions are even. Since $\varphi^{\GG^\varepsilon}_k (\ell_1+\ell_2) = 0$ unless $|\ell_m| \gtrsim 2^k$ for $m=1$ or $m=2$ and since $\| \varphi^{\GG^\varepsilon}_m\|_{L^2(\GG^\varepsilon)}\lesssim 2^{m d/2} $, we get
\begin{align*}
	&\left\| \sum_{|i-j|\le1}  \varphi^{\GG^\varepsilon}_k (\ell_1+\ell_2) \frac{\varphi^{\GG^\varepsilon}_i(\ell_1) \chi(\ell_1)}{l^\varepsilon_\mu(\ell_1)} \varphi^{\GG^\varepsilon}_j(\ell_2) \right\|_{L^2_{\ell_1,\ell_2}((\widehat{\GG}^\varepsilon)^2)} \\
	&\hspace{40pt} \lesssim \sum_{|i-j|\ge 1, j \gtrsim k} 2^{-2i} \left\|  \varphi^{\GG^\varepsilon}_k (\ell_1+\ell_2)  \varphi^{\GG^\varepsilon}_j(\ell_2) \right\|_{L^2_{\ell_1,\ell_2}((\widehat{\GG}^\varepsilon)^2)}  \\
	&\hspace{40pt} \lesssim \sum_{|i-j|\ge 1, j \gtrsim k} 2^{-2i} 2^{kd/2} 2^{jd/2}   \lesssim 2^{k (d-2)},
\end{align*}
using $d/2-2 < 0$ in the last step.
%
%
%
%
%
\end{proof}

By the compact embedding result in Lemma \ref{lem:ContinuousEmbedding} together with Prohorov's theorem we see that the sequences $(\EE^\varepsilon\xi^\varepsilon)$, $(\EE^\varepsilon X^\eps_\mu)$, and $(\EE^\varepsilon(X^\eps_\mu\bullet \xi^\varepsilon))$ have convergent subsequences in distribution -- note that while the H\"older space $\mathcal{C}^\zeta(\RR^d, p^\kappa)$ is not separable, all the processes above are supported on the closure of $\mathcal{C}^{\zeta'}(\RR^d, p^{\kappa'})$ for $\zeta' > \zeta$ and $\kappa' < \kappa$, which is a separable subspace and therefore we can indeed apply Prohorov's theorem. We will see in Lemma~\ref{lem:LimitArea} below that $\EE^\varepsilon\xi^\varepsilon$ converges to the white noise $\xi$ on $\RR^d$. Consequently, the solution $X^\eps_\mu$ to $-L^\varepsilon_\mu X^\eps_\mu=\chi(D_{\GG^\eps})\xi^\varepsilon$ should approach the solution of $-L_\mu X_\mu=\chi(D_{\RR^d})\xi :=\FFr^{-1}\big( \chi\, \FFr \xi\big)$, i.e.
\begin{align}
\label{eq:X}
	X_\mu=\frac{\chi(D_{\RR^d})}{(2\pi)^2\|D_{\RR^d}\|_\mu^2}\xi=\FFr^{-1}\left(  \frac{\chi}{(2\pi)^2\|\cdot\|_\mu^2} \FFr \xi\right)     =\mathscr{K}^0_\mu\ast\xi,\qquad\mathscr{K}^0_\mu:=\FF^{-1}_{\RR^d}\frac{\chi}{(2\pi)^2\lVert\cdot\rVert_\mu^2}\,.
\end{align}
where $\|\cdot\|_\mu$ is defined as in Definition \ref{def:Measure}. The limit of $\EE^\varepsilon(X^\eps_\mu\bullet \xi^\varepsilon)$ will turn out to be the distribution
\begin{align}
\label{eq:Xxi}
	X_\mu \bullet \xi(\varphi):=\int_{\RR^d}\int_{\RR^2} \mathscr{K}^0_\mu(z_1-z_2) \varphi(z_1)\xi(\dd z_1)\diamond\xi(\dd z_2) -(X_\mu\para\xi +\xi\para X_\mu)(\varphi)
\end{align} 
for $\varphi\in \Sw_\omega(\RR^d)$, where the right hand side denotes the second order Wiener-It\^o integral with respect to the Gaussian stochastic measure $\xi(\dd z)$ induced by the white noise $\xi$, compare \cite[Section 7.2]{Janson}. Note that $X_\mu\bullet \xi$ is not a continuous functional of $\xi$, so the last convergence is not a trivial consequence of the convergence for $\EE^\eps \xi^\eps$. To identify the limit of $\EE^\eps(X_\mu^\eps\bullet \xi^\eps)$ we could use a diagonal sequence argument that first approximates the bilinear functional by a continuous bilinear functional as in~\cite{Mourrat2017, Hairer2015Central, Chouk2017}. Here prefer to go another route and instead we follow~\cite{Caravenna} who provide a general criterion for the convergence of discrete multiple stochastic integrals to multiple Wiener-It\^o integrals, and we adapt their results to the Wick product setting of Lemma~\ref{lem:DiscreteItoIsometry}.


\begin{lemma}[see also \cite{Caravenna}, Theorem~2.3]
\label{lem:ChaosExpansionConvergence}
  Let $\GG^\eps,n\in \NN$ and $\big(\xi^\eps(z)\big)_{z\in \GG^\eps}$ be as in Lemma \ref{lem:DiscreteItoIsometry}. For $k = 0, \dots, n$ let $f_k^\varepsilon \in L^2((\GG^\varepsilon)^k)$. We identify $(\GG^\varepsilon)^k$ with a Bravais lattice in $k \cdot d$ dimensions via the orthogonal sum $(\GG^\varepsilon)^k=\bigoplus_{i=1}^k \GG^\varepsilon \subseteq \bigoplus_{i=1}^k\RR^d= (\RR^d)^k$  to define the Fourier transform $\mathcal{F}_{(\GG^\varepsilon)^k}f^\varepsilon_k\in L^2((\widehat{\GG^\varepsilon})^k)$ of $f_k^\varepsilon$. Assume that there exist $g_k \in L^2((\RR^d)^k)$ with $|\1_{(\widehat{\GG^\varepsilon})^k}\mathcal{F}_{(\GG^\varepsilon)^k}f^\varepsilon_k  |\leq g_k$ for all $\varepsilon$ and $f_k\in L^2((\RR^d)^k)$ such that $\lim_{\varepsilon \to 0} \|\1_{(\widehat{\GG^\varepsilon})^k} \mathcal{F}_{(\GG^\varepsilon)^k}f^\varepsilon_k  - \Ff_{(\RR^d)^k} f_k\|_{L^2((R^d)^k)} = 0$ for all $k \leq n$. Then the following convergence holds in distribution
\begin{align*}
   \lim_{\varepsilon \to 0} \sum_{k=0}^n \mathscr{I}_k f_k^\varepsilon = \sum_{k=0}^n \int_{(\RR^d)^k} f_k(z_1, \dots, z_k) \, \xi(\dd z_1)\diamond \dots \diamond\xi(\dd z_k)\,,
\end{align*}
where $\xi(\dd z_1)\diamond \dots \diamond\xi(\dd z_k)$ denotes the Wiener-It\^o integral against the Gaussian stochastic measure induced by the white noise $\xi$ on $\RR^d$. 
\end{lemma}

\begin{proof}
The proof is contained in the appendix.
\end{proof}

The identification of the limits of the extensions of $\xi^\eps,X_\mu^\eps$ and $X_\mu^\eps\bullet\xi^\eps$ is then an application of Lemma \ref{lem:ChaosExpansionConvergence}. 

\begin{lemma}
\label{lem:LimitArea}
In the setup of Lemma \ref{lem:RegularityNoise}  with $\xi,\,X_\mu$ and $ X_\mu\bullet\xi$ defined as above and with $\zeta,\kappa$ as in Lemma \ref{lem:RegularityNoise} we have for $d<4$
\begin{align*}
	(\EE^\varepsilon\xi^\varepsilon,\EE^\varepsilon X^\eps_\mu,\EE^\varepsilon (X^\eps_\mu\bullet\xi^\varepsilon))\overset{\varepsilon\rightarrow 0}{\longrightarrow}(\xi,X_\mu,X_\mu\bullet\xi)
\end{align*}
in distribution in $\mathcal{C}^{\zeta-2}(\RR^d,\w^{\kappa})\times\mathcal{C}^{\zeta}(\RR^d,\w^{\kappa})\times\mathcal{C}^{2\zeta-2}(\RR^d,\w^{2\kappa} )$.
\end{lemma}

\begin{proof}
Recall that the extension operator $\EE^\eps$ is constructed from $\psi^\eps=\psi(\eps\cdot )$ where the smear function $\psi \in \Dww(\RR^d)$ is symmetric and satisfies $\psi=1$ on some ball around 0. Since from Lemma~\ref{lem:RegularityNoise} we already know that the sequence $(\EE^\varepsilon\xi^\varepsilon,\EE^\varepsilon X^\eps_\mu,\EE^\varepsilon (X^\eps_\mu\bullet\xi^\varepsilon))$ is tight in $\mathcal{C}^{\zeta-2}(\RR^d,\w^{\kappa})\times\mathcal{C}^{\zeta}(\RR^d,\w^{\kappa})\times\mathcal{C}^{2\zeta-2}(\RR^d,\w^{2\kappa} )$, it suffices to prove the convergence after testing against $\varphi\in \Sw_\omega(\RR^d)$:
\begin{align*}
	&(\EE^\varepsilon\xi^\varepsilon(\varphi_1),\dots, \EE^\varepsilon\xi^\varepsilon(\varphi_n),\EE^\varepsilon X^\eps_\mu(\psi_1),\dots,\EE^\varepsilon X^\eps_\mu(\psi_n) ,\EE^\varepsilon (X^\eps_\mu \bullet\xi^\varepsilon)(f_1), \dots,\EE^\varepsilon (X^\eps_\mu \bullet\xi^\varepsilon)(f_n) ) \\
	&\hspace{20pt} \overset{\varepsilon\to0}{\rightarrow}	(\xi(\varphi_1), \dots, \xi(\varphi_n), X_\mu(\psi_1),\dots, X_\mu(\psi_n),X_\mu\bullet\xi(f_1),\dots, X_\mu\bullet\xi(f_n))\,,
\end{align*}
and by taking linear combinations and applying Lemma~\ref{lem:ChaosExpansionConvergence} we see that it suffices to establish each of the following convergences:
\begin{align}
\label{eq:ChaosConvergenceInDistribution}
	\EE^\varepsilon\xi^\varepsilon(\varphi) \overset{\varepsilon\to 0}{\longrightarrow} \xi(\varphi),\quad \EE^\varepsilon X^\eps_\mu(\varphi) \overset{\varepsilon\to0}{\longrightarrow} X_\mu(\varphi),\quad \EE^\varepsilon (X^\eps_\mu \bullet\xi^\varepsilon)(\varphi)) \overset{\varepsilon \to 0}{\rightarrow}	X_\mu\bullet\xi(\varphi)
\end{align}
for all $\varphi \in \Sw_\omega(\RR^d)$. We can even restrict ourselves to those $\varphi\in\Sw_\omega(\RR^d)$ with $\FFr \varphi \in \Dww(\RR^d)$, which implies  $\supp \FFr\varphi \subseteq \wGGe$ and $\FFr^{-1}(\psi^\eps  \FFr\varphi )=\varphi$ for $\varepsilon$ small enough, which we will assume from now on. Note that $\supp \FFr\varphi\subseteq \wGGe$ implies
\begin{align}
\label{eq:FiniteBandwithRelation}
 \FFge (\varphi\vert_{\GGe}) =(\FFr \varphi)\vert_{\wGGe}
 \end{align} 
since by definition of $\FFge^{-1}$
\begin{align*}
\FFge^{-1} ((\FFr \varphi)\vert_{\wGGe}  )=(\FFr^{-1} \FFr  \varphi ) \vert_{\GG^\eps}= \varphi \vert_{\GG^\eps}\,.
\end{align*}
 
To show the convergence of $\EE^\varepsilon\xi^\varepsilon(\varphi)$ to $\xi(\varphi)$ note that we have from \eqref{eq:ExtensionAsConvolution}
\begin{align*}
\EE^\varepsilon \xi^\varepsilon(\varphi)=\sum_{z\in \GG^\varepsilon} |\GG^\varepsilon| \, (\FFr^{-1}{\psi}^\eps\ast \varphi)(z)  \xi^\varepsilon(z)= \sum_{z\in \GG^\varepsilon} |\GG^\varepsilon| \, \FFr^{-1}({\psi}^\eps \,\FFr\varphi)(z)  \xi^\varepsilon(z)    =\sum_{z\in \GG^\varepsilon} |\GG^\varepsilon| \varphi(z) \xi^\varepsilon(z)
\end{align*}
where we used in the first step that $\psi^\eps$ is symmetric and in the last step that $\FFr^{-1}({\psi}^\eps  \FFr\varphi )=\varphi$ by our choice of $\varphi$ and $\eps$. Using Lemma \ref{lem:ChaosExpansionConvergence} and relation \eqref{eq:FiniteBandwithRelation} the convergence of $\EE^\varepsilon \xi^\varepsilon(\varphi)$ to $\xi(\varphi)$ follows. 

For the limit of $\EE^\varepsilon X^\eps_\mu$ we use the following formula, which is derived by the same argument as above:
\begin{align*}
\EE^\varepsilon X^\eps_\mu(\varphi)=\sum_{z_1,\,z_2\in\GG^\varepsilon}\,|\GG^\varepsilon|^2\,\varphi(z_1) \mathscr{K}^\varepsilon_\mu(z_2-z_1) \xi^\varepsilon(z_2)
\end{align*}
with $\mathscr{K}_\mu^\varepsilon=\FFge^{-1} (\chi/l^\varepsilon_\mu)$. In view of Lemma \ref{lem:ChaosExpansionConvergence} it then suffices to note that
\[
	\hat{f}^\eps:=\FFge(\varphi \aste \mathscr{K}_\mu^\varepsilon)=\FFge \varphi \cdot\frac{\chi}{l^\varepsilon_\mu} \overset{\eqref{eq:FiniteBandwithRelation}}{=}\FFr \varphi \cdot \frac{\chi}{l^\eps_\mu}
\]
is dominated by a multiple of $\chi/\lvert\cdot\rvert^2$ on $\widehat{\GG^\varepsilon}$ due to Lemma \ref{lem:SchauderPrelude}, and it converges to 
\begin{align*}
\FFr\varphi \cdot \frac{\chi}{(2\pi)^2\lVert\cdot\rVert_\mu^2}
\end{align*}
by the explicit formula for $l^\eps_\mu$ in \eqref{eq:MultiplierOperator}. 

We are left with the convergence of the third component. Since $\EE^\varepsilon\xi^\varepsilon\rightarrow \xi$ and $\EE^\varepsilon X^\eps_\mu \rightarrow X_\mu$ we obtain via the \eqref{eq:EProperty}-Property of the paraproduct
\[
	\lim_{\varepsilon\to0} \EE^\varepsilon (X^\eps_\mu \para^{\GGe} \xi^\varepsilon) = \lim_{\varepsilon\to0} \EE^\varepsilon X^\eps_\mu\para \EE^\varepsilon \xi^\varepsilon = X_\mu \para \xi
\]
and similarly one gets $\EE^\varepsilon (\xi^\varepsilon \para^{\GGe} X^\eps_\mu) \to \xi \para X_\mu$. We can therefore show instead 
\begin{align}
\label{eq:ProofConvergenceAreaWickConvergence}
	\EE^\varepsilon\left(X^\eps_\mu\xi^\varepsilon-\mathbb{E}[X^\eps_\mu \xi^\varepsilon]\right)(\varphi)\rightarrow (X_\mu\bullet \xi +\xi \para X_\mu+ X_\mu \para \xi)(\varphi)\,.
\end{align}
Note that we have the representations 
\begin{align*}
	\EE^\varepsilon\left(X^\eps_\mu\xi^\varepsilon-\mathbb{E}[X^\eps_\mu \xi^\varepsilon]\right)(\varphi)&=\sum_{z_1,z_2\in \GG^\varepsilon}|\GG^\varepsilon|^2 \varphi(z_1) \mathscr{K}^\varepsilon_\mu(z_1-z_2) \, \xi^\varepsilon(z_1)\diamond \xi^\varepsilon(z_2) \,, \\
	(X_\mu\bullet \xi +\xi \para X_\mu+ X_\mu \para \xi)(\varphi)&=\int_{\RR^2}\int_{\RR^2} \varphi(z_1) \mathscr{K}^0_\mu(z_1-z_2)\, \xi(\dd z_1)\diamond \xi(\dd z_2)
\end{align*}
with $\mathscr{K}^\eps_\mu$ as above and $\mathscr{K}^0_\mu$ as in \eqref{eq:X}.
The $(\GG^\varepsilon)^2$-Fourier transform of $\varphi(z_1) \mathscr{K}^\varepsilon_\mu(z_1-z_2)$ is $\hat{\varphi}_{\mathrm{ext}}(x_1-x_2) \chi(x_2)/l^\varepsilon_\mu(x_2)$ for $x_1,x_2\in\widehat{\GG^\varepsilon}$, where $\hat{\varphi}_{\mathrm{ext}}$ denotes the periodic extension from \eqref{eq:PeriodicExtension2} for $\FFr \varphi\vert_{\wGGe}\in\Dww(\wGGe)$ (recall again that $\supp \FFr\varphi\subseteq \wGGe$). We can therefore apply Lemma~\ref{lem:ChaosExpansionConvergence} since for $d<4$ the function $(\chi(x_2)/l^\varepsilon(x_2))^2\lesssim \mathbf{1}_{|x|\gtrsim 1} /|x|^{4} $ is integrable on $\widehat{\GG^\varepsilon}$ and thus we obtain \eqref{eq:ProofConvergenceAreaWickConvergence}.

We have shown the convergence in distribution of all the components in \eqref{eq:ChaosConvergenceInDistribution}. By Lemma~\ref{lem:ChaosExpansionConvergence} we can take any linear combination of these components and still get the convergence from the same estimates, so that~\eqref{eq:ChaosConvergenceInDistribution} follows from the Cram\'er-Wold Theorem. 
\end{proof}

\subsection{Convergence of the lattice model}\label{sec:convergence}

We are now ready to prove the convergence of $\EE^\eps u^\eps$ announced at the beginning of this section. The key statement will be the a priori estimate in Lemma \ref{lem:Picard}. The convergence of $\EE^\eps u^\eps$ to the continuous solution on $\RR^2$, constructed in Corollary \ref{cor:ContinuousSolution}, will be proven in Theorem \ref{thm:ConvergencePAM}. We first fix the relevant parameters. 
\label{subsec:ConvergenceOfTheLatticeModel}
\subsubsection*{Preliminaries} 
Throughout this subsection we use the same $p\in [1,\infty]$, $\sigma \in (0,1),\,\mu \in \mm(\wexps)$, a polynomial weight $p^\kappa$ for some $\kappa>2/p_\xi > 1/7$ and a time dependent sub-exponential weight $(e^\sigma_{l+t})_{t\in [0,T]}$. We further fix an arbitrarily large time horizon $T>0$ and require $l\le -T$ for the parameter in the weight $e^\sigma_l$. Then we have $1\leq e^\sigma_{l+t}\leq (e^\sigma_{l+t})^2$ for any $t\leq T$, which will be used to control a quadratic term that comes from the Taylor expansion of the non-linearity $F^\eps$. We take $\xi^\eps$ as in the beginning of this section with $p_\xi>14$ (see Remark \ref{rem:fourteen} below) and construct $X^\eps_\mu$ as in Subsection \ref{subsec:ConvergenceArea}. We further fix a parameter 
\begin{align}
\label{eq:ParameterCondition1}
	\alpha\in (2/3-2/3\cdot \kappa/\sigma, 1- 2 /p_\xi-2\kappa/\sigma)
\end{align}
with $\kappa/\sigma\in (2/p_\xi,1)$ small enough such that the interval is non-empty, which (as we will discuss in the following remark) is possible since $2/p_\xi<1/7$. 

\begin{remark}[Why $14+$ moments]
\label{rem:fourteen}
Let us sketch where the boundaries of the interval \eqref{eq:ParameterCondition1} come from. The parameter $\alpha$ will measure the regularity of $u^\eps$ below. The upper boundary, that is  $1- 2 /p_\xi-2\kappa/\sigma$, arises due to the fact that we cannot expect $u^\eps$ to be better than $X^\eps$, which has regularity below $1-2/p_\xi$ due to Lemma \ref{lem:RegularityNoise}. The correction $-2\kappa/\sigma$ is just the price one pays in the Schauder estimate in Lemma \ref{lem:Schauder} for the ``weight change''. The lower bound $2/3-2/3\cdot \kappa/\sigma$ is a criterion for our paracontrolled approach below to work. We increase below the regularity $\alpha$ of our solutions, by subtraction of a paraproduct, to $2\alpha$. By Lemma \ref{lem:ParaproductEstimates} this allows us to uniformly control products with $\xi^\eps$ provided
\begin{align*}
   2\alpha+(\alpha+2\kappa/\sigma-2)>0\,,
\end{align*}  
because $\xi^\eps \in \mathcal{C}^{\alpha+2\kappa/\sigma-2}_{p^\kappa}$. This condition can be reshaped to $\alpha>2/3-2/3\cdot \kappa/\sigma$, explaining the lower bound. The interval \eqref{eq:ParameterCondition1} can only be non-empty if
\begin{align*}
2/3-2/3\cdot \kappa/\sigma<1- 2/p_\xi-2\kappa/\sigma \,\Leftrightarrow \,2/3<1- 2/p_\xi-4/3\cdot \kappa/\sigma
\end{align*}
Lemma \ref{lem:RegularityNoise} forces us to take $\kappa/\sigma>2/p_\xi$ so that the the right hand side can only be true provided $2/3<1- 2/p_\xi-4/3 \cdot 2/p_\xi$, which is equivalent to
\begin{align*}
p_\xi>14\,.
\end{align*}
\end{remark}

Let us mention the simple facts $2\alpha+2\kappa/\sigma,2\alpha+4\kappa/\sigma\in (0,2)$, $\alpha+\kappa/\sigma,\alpha+2\kappa/\sigma \in (0,1)$ and $3\alpha+2\kappa/\sigma-2>0$ which we will use frequently below. 

We will assume that the initial conditions $u_0^\eps$ are uniformly bounded in $\mathcal{C}^{0}_p(\GG^\eps,e^\sigma_l)$ and are chosen such that $\EE^\eps u_0^\eps$ converges in $\Sw_\omega'(\RR^2)$ to some $u_0$. For $u^\eps_0 = |\GG^\eps|^{-1}\mathbf{1}_{\cdot = 0}$ it is easily verified that this is indeed the case and the limit is the Dirac delta, $u_0 = \delta$.

Recall that we aim at showing that (the extension of) the solution $u^\eps$ to 
\begin{align}
\label{eq:DiscretePAM2}
\mathscr{L}_\mu^\eps u^\eps=F(u^\eps)(\xi^\eps-F'(0)c^\eps_\mu),\qquad u^\eps(0)=u^\eps_0=|\GG^\eps|^{-1}\mathbf{1}_{\cdot=0}
\end{align}
converges to the solution of 
\begin{align}
\label{eq:ContinuousPAM2}
\mathscr{L}_\mu u =F'(0) u \blacklozeng  \xi,\qquad u(0)=u_0=\delta\,,
\end{align}
where $u\blacklozeng \xi$ is a suitably renormalized product defined in Corollary \ref{cor:ContinuousSolution} below.
 
Our solutions will be objects in the parabolic space $\mathscr{L}^{\alpha,\alpha}_{p,T}$ which does not require continuity at $t=0$. A priori there is thus no obvious meaning for the Cauchy problems \eqref{eq:DiscretePAM2}, \eqref{eq:ContinuousPAM2} (although of course for \eqref{eq:DiscretePAM2} we could use the pointwise interpretation). We use the common interpretation of (\ref{eq:DiscretePAM2}, \ref{eq:ContinuousPAM2}) as equations for distributions $u^\eps,u\in \Dw_\omega'(\RR^{1+2})$ (compare for example \cite[Definition 3.3.4]{TriebelHigherAnalysis}) by requiring $\supp u^\eps,\,\supp u \subseteq \RR_{+}\times \RR^2$ and
\begin{align*}
\mathscr{L}_\mu^\eps u^\eps &=F(u^\eps)(\xi^\eps-F'(0)c^\eps_\mu)+\delta \otimes u^\eps_0\,, \\
\mathscr{L}_\mu u &=F'(0) u\blacklozeng \xi +\delta\otimes u_0\,,
\end{align*}
in the distributional sense on $(-\infty,T) \times \RR^2$, where $\otimes$ denotes the tensor product between distributions. Since we mostly work with the mild formulation of these equations the distributional interpretation will not play a crucial role. Some care is needed to check that the only distributional solutions are mild solutions, since the distributional Cauchy problem for the heat equation is not uniquely solvable \cite{Tychonoff}. However, under generous growth conditions for $u,u^\eps$ for $x\rightarrow \infty$ (compare \cite{Friedman}) there is a unique solution. In our case this fact can be checked by considering the Fourier transform of $u,u^\eps$ in space.
\subsubsection*{A priori estimates} 

We will work with the following space of paracontrolled distributions.
\begin{definition}[Paracontrolled distribution for 2d PAM]
\label{def:Paracontrolled}
We identify a pair \[(u^{\eps,X},u^{\eps,\sharp})\colon [0,T]\rightarrow \Sw_\omega'(\GG^\eps)^2\] with $u^\eps:=u^{\eps,X}\mpara X_\mu^\eps+u^{\eps,\sharp}$ and introduce a norm  
\begin{align}
\label{eq:ParacontrolledNorm}
	\|u^\eps\|_{\mathscr{D}^{\gamma,\delta}_{p,T}(\GG^\eps,e^{\tilde{\sigma}}_l)}:=\|(u^{\eps,X},u^{\eps,\sharp})\|_{\mathscr{D}^{\gamma,\delta}_{p,T}(\GG^\eps,e^{\tilde{\sigma}}_l)}:=\|u^{\eps,X}\|_{\mathscr{L}^{\gamma/2,\delta}_{p,T}(\GG^\eps,e^{\tilde{\sigma}}_l)}+\|u^{\eps,\sharp}\|_{\mathscr{L}^{\gamma,\delta+\alpha}_{p,T}(\GG^\eps,e^{\tilde{\sigma}}_l)}
\end{align}
for $\alpha$ as above, $\tilde{\sigma}\in (0,1)$ and $\gamma\geq 0$, $\delta \in (0, 2-\alpha)$. We denote the corresponding space by $\mathscr{D}^{\gamma,\delta}_{p,T}(\GG^\eps,e^{\tilde{\sigma}}_l)$. If the norm \eqref{eq:ParacontrolledNorm} is bounded for a sequence $(u^\eps=u^{\eps,X}\mpara X_\mu^\eps+u^{\eps,\sharp})_\varepsilon$ we say that $u^\eps$ is paracontrolled by $X_\mu^\eps$. 
\end{definition}
\begin{remark}
\label{rem:ContinousSpacesParacontrolled}
In view of Remark \ref{rem:ContinousSpacesParabolic} we can also define a \emph{continuous} version $\mathscr{D}^{\gamma,\delta}_{p,T}(\RR^d,e^{\tilde{\sigma}}_l)$ of the space above. 
\end{remark}
\glsadd{mathscrDgammaalphapT}
As in \cite{Gubinelli2017KPZ} it will be useful to have a common bound on the stochastic data: Let  
\begin{align}
\label{eq:BoundData}
	M_\eps:=\|\xi^\eps\|_{\mathcal{C}^{\alpha+2\kappa/\sigma-2}_\infty(\GG^\eps,p^\kappa)}
	\vee \|X_\mu^\eps\|_{\mathcal{C}^{\alpha+2\kappa/\sigma}_\infty(\GG^\eps,p^\kappa)} 
	\vee \|X_\mu^\eps\bullet \xi^\eps\|_{\mathcal{C}^{2\alpha+4\kappa/\sigma-2}_\infty(\GG^\eps,p^{2\kappa})}
\end{align}
(compared to Lemma~\ref{lem:RegularityNoise} we have $\zeta=\alpha+2\kappa/\sigma$). The following a priori estimates will allow us to set up a Picard iteration below. 
\begin{lemma}[A priori estimates]
\label{lem:Picard}
In the setup above consider $\gamma\in \{0,\,\alpha\}$ and $u_0\in \cC^{0}_p(\GGe)$. If $\gamma=0$ we require further that $u_0\in \cC^\alpha_p(\GG^\eps,\rho)$ and $u_0^{\sharp}:=u_0-F'(0)u_0\para X_\mu^\eps \in \cC^{2\alpha}_p(\GG^\eps,e^\sigma_l)$. Define a map
\begin{align*}
\mathscr{M}^\eps_{\gamma,u_0} \colon \qquad \,\mathscr{D}_{p,T}^{\gamma,\alpha}(\GG^\eps,e^\sigma_l) \ni (u^{\eps,X}, u^{\eps,\sharp})\,\, \longmapsto\,\, (v^{\eps,X}, v^{\eps,\sharp}) \in \mathscr{D}_{p,T}^{\gamma,\alpha}(\GG^\eps,e^\sigma_l)
\end{align*}
for $u^\eps=u^{\eps,X}\mpara X_\mu^\eps+ u^{\eps,\sharp}$ with $u^\eps(0)=u_0$
via $v^{\eps,X}:=F'(0)u^\eps$ and $v^{\eps,\sharp} := v^\eps-v^{\eps,X}\mpara X_\mu^\eps$,
where $v^\eps$ is the solution to the problem
\begin{align}
\mathscr{L}^\eps_\mu v^\eps &:=F^\eps(u^\eps)\xi^\eps - F^\eps(u^{\eps,X}/F'(0)) F'(0) c^\eps_\mu, \qquad v^\eps(0)=u_0\,.
\label{eq:PicardProblem1}
\end{align}
The map $\mathscr{M}^\eps_{\gamma,u_0}$ is well defined for $\gamma\in\{0,\alpha\}$ and we have the bound 
\begin{align*}
\|(v^{\eps,X}, v^{\eps,\sharp})\|_{\mathscr{D}_{p,T}^{\gamma,\alpha}(\GG^\eps,e^\sigma_l)} &\leq C_{u_0} 
+C_{M_\eps} \cdot T^{(\alpha-\delta)/2}\,\left(\|u^\eps\|_{\mathscr{D}_{p,T}^{\gamma,\alpha}(\GG^\eps,e^\sigma_l)}+\eps^{\nu} \|u^\eps\|_{\mathscr{D}_{p,T}^{\gamma,\alpha}(\GG^\eps,e^\sigma_l)}^2   \right) 
\end{align*}
for $\delta\in \left(2-2\alpha-2\kappa/\sigma,\alpha\right)$ and some $\nu>0$, where $C_{M_\eps}=c_0\,(1+M_\eps^2)$ and 
\begin{align} \nonumber
C_{u_0} & = \mathbf{1}_{\gamma=\alpha} \,c_0\,\|u_0\|_{\mathcal{C}_{p}^{0}(\GG^\eps,e^\sigma_l)}  \\
	&\quad +\mathbf{1}_{\gamma=0}\, c_0\, \left(\|u_0^{\sharp}\|_{\mathcal{C}_{p}^{2\alpha}(\GG^\eps,e^\sigma_l)}  + \|u^{\eps,X}(0)\|_{\mathcal{C}_{p}^{\alpha}(\GG^\eps,e^\sigma_l)}+ \|u^{\eps,\sharp}(0)\|_{\mathcal{C}_{p}^{2\alpha}(\GG^\eps,e^\sigma_l)} \right)\,,
\label{eq:Cu0}
\end{align}
for some $c_0>0$ that does not depend on $\xi^\varepsilon$, $c^\varepsilon_\mu$ or $u_0$.
\end{lemma}

\begin{remark}
\label{rem:FixPoint}
The complicated formulation of~\eqref{eq:PicardProblem1} is necessary because when we expand the singular product on the right hand side we get
\[
   F^\eps(u^\eps)\xi^\eps = F'(0) (C(u^{\eps,X},X_\mu^\eps,\xi^\eps) + u^{\eps,X} (X_\mu^\eps\reso \xi^\eps)) + \dots\,,
\]
so to obtain the right renormalization we need to subtract $F'(0) u^{\eps,X} c^\eps_\mu$, which is exactly what we get if we Taylor expand the second addend on the right hand side of~\eqref{eq:PicardProblem1}. 

If $u^\varepsilon = v^\varepsilon = \mathscr{M}^\varepsilon_{\gamma,u_0} u^\varepsilon$ is a fixed point, 
then $u^{\eps,X} = v^{\eps,X} = F'(0) u^\eps$ and the ``renormalization term'' is just $F^\eps(u^\eps) F'(0) c^\eps_\mu$. Moreover we have in this case
\begin{align*}
\mathscr{L}^\eps_\mu u^\eps = F^\eps(u^\eps)(\xi^\eps - F'(0) c^\eps_\mu)\,,\qquad u^\eps(0)=u_0\,.
\end{align*}
\end{remark}

\begin{proof}
We assume for the sake of shorter formulas $(1+M_\eps^2)\lesssim 1$, the general case can be easily included in the reasoning below. 
The solution to \eqref{eq:PicardProblem1} can be constructed using the Green's function $\FFge^{-1} e^{-t l^\eps_\mu}$ and Duhamel's principle. To uncluster the notation a bit, we will drop the upper index $\eps$ on $u,\,v,\,X_\mu,\,\mathscr{L}_\mu,\ldots$  in this proof.  We show both estimates at once by denoting by $\gamma$ either $0$ or $\alpha$. 

Throughout the proof we will use the fact that 
\begin{align}
\label{eq:UseParacontrolledStructure}
	\|u\|_{\mathscr{L}^{\gamma/2,\alpha}_{p,T}(\GG^\eps,e^\sigma_l)}=\|u^{X}\mpara X_\mu +u^\sharp\|_{\mathscr{L}^{\gamma/2,\alpha}_{p,T}(\GG^\eps,e^\sigma_l)}\lesssim \|u\|_{\mathscr{D}_{p,T}^{\gamma,\beta}(\GG^\eps,e^\sigma_l)}
\end{align}
for all $\beta\in (0,\alpha]$ which follows from Lemma \ref{lem:ParaproductEstimateParabolicSpace}. In particular (with $\beta=\delta$) we have
\begin{align}
   \|v^{X}\|_{\mathscr{L}^{\gamma/2,\alpha}_{p,T}(\GG^\eps,e^\sigma_l)} &= \|F'(0) u \|_{\mathscr{L}^{\gamma/2,\alpha}_{p,T}(\GG^\eps,e^\sigma_l)} \overset{\eqref{eq:UseParacontrolledStructure}}{\lesssim} \|u\|_{\mathscr{D}^{\gamma,\delta}_{p,T}(\GG^\eps,e^\sigma_l)} \nonumber \\
    &\overset{\xshort \short \tiny \mbox{Lem. \ref{lem:ComparisonParabolicSpaces}}}{\lesssim} \mathbf{1}_{\gamma=0} (\|u^{X}(0)\|_{\mathcal{C}_{p}^{\alpha}(\GG^\eps,e^\sigma_l)}+\|u^{\sharp}(0)\|_{\mathcal{C}_{p}^{2\alpha}(\GG^\eps,e^\sigma_l)})+ T^{\frac{\alpha-\delta}{2}} \|u\|_{\mathscr{D}^{\gamma,\alpha}_{p,T}(\GG^\eps,e^\sigma_l)}\,.
\label{eq:APrioriVX}
\end{align}

This leaves us with the task of estimating $\|v^\sharp\|_{\mathscr{L}^{\gamma,2\alpha}_{p,T}(\GG^\eps,e^\sigma_l)}$.
We split 
\begin{align}
\mathscr{L}_\mu v^\sharp &=\mathscr{L}_\mu(v-F'(0)u\mpara X_\mu) \label{eq:ParacontrolledSchauderToDecompose}
 \\
 &=F^\eps(u)\xi-F^\varepsilon(u^Y/F'(0)) F'(0)c_\mu-F'(0)  \mathscr{L}_\mu(u\mpara Y \nonumber)
 \\ 
 &= F'(0)u\xi- F'(0)u^X c_\mu -F'(0)\mathscr{L}_\mu(u\mpara X_\mu)+R(u)u^2 \xi - R(u^{X}/F'(0)) \frac{(u^{X})^2 }{F'(0)} c_\mu \nonumber\\
&=F'(0)[ u\para (\xi-\bar{\xi})+u\para \bar{\xi} -u\mpara \bar{\xi} +u \mpara \bar{\xi}- \mathscr{L}_\mu (u\mpara  X_\mu)+ u \lpara \xi  \tag{$\naivepara$} \label{eq:ParacontrolledSchauder1} \\
&\hspace{45pt} +C(u^{X},X_\mu,\xi)+u^{X}(X_\mu\bullet \xi)  \tag{$\naivereso$} \label{eq:ParacontrolledSchauder2} \\
&\hspace{45pt} + u^\sharp\reso \xi ] \tag{$\sharp$}  \label{eq:ParacontrolledSchauder3}\\
&\quad +R(u)\cdot u^2 \xi\nonumber 
\label{eq:ParacontrolledSchauder4}
\tag{$R_u$}\\
&\quad -R(u^{X}/F'(0))\frac{(u^{X})^2}{F'(0)} c_\mu
\tag{$R_{u^{X}}$}
\label{eq:ParacontrolledSchauder5}
\,,
\end{align}
where $\overline{\xi}=\chi(D)\xi$ so that $\mathscr{L}_\mu X_\mu=\bar{\xi}$ with $\xi-\bar{\xi}\in \bigcap_{\beta\in \RR} \mathcal{C}^\beta_{\infty}(\GG^\eps,p^\kappa)$ and where $R(x)=\eps^2\int_0^1 (1-\lambda) F''(\lambda \eps^2 x)\dd \lambda$.  We have by Lemmas \ref{lem:ParaproductEstimates}, \ref{lem:CommuteModifiedParaproduct}
\[
	\|\eqref{eq:ParacontrolledSchauder1}\|_{\mathcal{M}^\gamma_T\mathcal{C}^{2\alpha+2\kappa/\sigma-2}_p(\GG^\eps,e^\sigma_l p^\kappa)} 
	\lesssim \|u\|_{\mathscr{L}^{\gamma/2,\alpha}_{p,T}(\GG^\eps,e^\sigma_l)}
	\overset{\mbox{\tiny{\eqref{eq:UseParacontrolledStructure}}}}{\lesssim} \|u\|_{\mathscr{D}^{\gamma,\delta}_{p,T}(\GG^\eps,e^\sigma_l)}
\]
and further with Lemma \ref{lem:CommutatorLemma} and Lemma \ref{lem:ParaproductEstimates}
\[
	\|\eqref{eq:ParacontrolledSchauder2}\|_{\mathcal{M}^\gamma_T \mathcal{C}^{2\alpha+4\kappa/\sigma-2}(\GG^\eps,e^\sigma_l p^{2\kappa})}\lesssim \|u\|_{\mathscr{D}^{\gamma,\delta}_{p,T}(\GG^\eps,e^\sigma_l)}\,,
\]
while the term \eqref{eq:ParacontrolledSchauder3} can be bounded with Lemma \ref{lem:ParaproductEstimates} by 
\[
	\|u^\sharp\reso \xi\|_{\mathcal{M}^\gamma_T\mathcal{C}^{2\alpha+2\kappa/\sigma-2}_p(\GG^\eps,e^\sigma_l p^\kappa)} \lesssim \|u^\sharp\|_{\mathscr{L}^{\gamma,\alpha+\delta}_{p,T}(\GG^\eps,e^\sigma_l)}\leq \|u\|_{\mathscr{D}^{\gamma,\delta}_{p,T}(\GG^\eps,e^\sigma_l)}\,.
\]
To estimate \eqref{eq:ParacontrolledSchauder4} we  use the simple bounds $\|\eps^{\beta'} f\|_{\mathcal{C}_q^{\beta+\beta'}(\GG^\eps,\rho)}\lesssim \|f\|_{\mathcal{C}^{\beta}_q(\GG^\eps,\rho)}$ for $\beta\in \RR$, $\beta'>0$, $q\in [1,\infty]$, $\rho\in \boldsymbol{\rho}(\omega)$ and
\[
	\|\eps^{-\beta} f\|_{L^q(\GG^\eps,\rho)}\lesssim \eps^{-\beta} \sum_{j\lesssim j_{\GG^\eps} } 2^{-j\beta } \|f\|_{\mathcal{C}^{\beta}_q(\GG^\eps,\rho)} \lesssim \|f\|_{\mathcal{C}^{\beta}_q(\GG^\eps,\rho)}
\]
for $\beta<0$, $q\in [1,\infty]$, $\rho\in\boldsymbol{\rho}(\omega)$, together with the assumption $F''\in L^\infty$, and obtain for $\nu' > 0$
\begin{align*}
\|\eqref{eq:ParacontrolledSchauder4}\|_{\mathcal{M}^\gamma_T\mathcal{C}^{2\alpha+2\kappa/\sigma-2}_p(\GG^\eps,e^\sigma_l p^\kappa)} &\lesssim \|F''\|_\infty \|\eps^{\alpha+2\kappa/\sigma} u^2\|_{\mathcal{M}^\gamma L^p(\GG^\eps,e^\sigma_l)} \, \|\eps^{2-(\alpha+2\kappa/\sigma)} \xi\|_{L^\infty(\GG^\eps,p^\kappa)} 
\\&\lesssim\|\eps^{\alpha+2\kappa/\sigma} u^2\|_{\mathcal{M}^\gamma_T L^p(\GG^\eps,(e^\sigma_l)^2)}\, \|\xi\|_{\mathcal{C}^{\alpha+2\kappa/\sigma-2}_\infty(\GG^\eps, p^\kappa)}
\\& \lesssim \|\eps^{\alpha/2+\kappa/\sigma} u\|_{\mathcal{M}^{\gamma/2}_T L^{2p}(\GG^\eps,e^\sigma_l)}^2 \lesssim \|\eps^{\alpha/2+\kappa/\sigma} u\|_{\mathcal{M}^{\gamma/2}_T \mathcal{C}_p^{d/2p+\nu'}(\GG^\eps,e^\sigma_l)}^2 
\\& \le \|\eps^{\alpha/2+\kappa/\sigma} u\|_{\mathcal{M}^{\gamma/2}_T \mathcal{C}_p^{1+\nu'}(\GG^\eps,e^\sigma_l)}^2 \lesssim \| \eps^{\alpha/2+\kappa/\sigma-(1+\nu'-\alpha)} u\|^2_{\mathcal{M}^{\gamma/2}_T \mathcal{C}^\alpha_p(\GG^\eps,e^\sigma_l)} 
\\& \lesssim \eps^{3\alpha+2\kappa/\sigma-2(1+\nu')} \|u\|_{\mathscr{D}^{\gamma,\delta}_{p,T}(\GG^\eps,e^\sigma_l)}^2 \le \varepsilon^\nu \|u\|_{\mathscr{D}^{\gamma,\delta}_{p,T}(\GG^\eps,e^\sigma_l)}^2 
\end{align*}	
for all $\nu\in (0,3\alpha+2\kappa/\sigma-2(1+\nu')]$ (which is nonempty if $\nu'$ is sufficiently small). Similarly we get for $\nu'\in (0,\delta)$
\begin{align*}
\|\eqref{eq:ParacontrolledSchauder5}\|_{\mathcal{M}^\gamma_T\mathcal{C}^{2\alpha+2\kappa/\sigma-2}_p(\GG^\eps,e^\sigma_l p^\kappa)} &\lesssim \|F''\|_{L^\infty(\RR)} \cdot c_\mu \| \eps u^{X}\|_{\mathcal{M}^{\gamma/2}_T L^{2p}(\GG^\eps,e^\sigma_l)}^2 \lesssim c_\mu \|\eps u^{X}\|_{\mathcal{M}^{\gamma/2}_T \mathcal{C}^{1+\nu'}_{p}(\GG^\eps,e^\sigma_l)}^2  \\
&\lesssim \eps^{2(\delta-\nu')} |\log(\eps)| \|u^{X}\|_{\mathcal{M}^{\gamma/2}_T \mathcal{C}^{\delta}_p(\GG^\eps,e^\sigma_l)}^2 \lesssim \eps^{\nu} \|u\|_{\mathscr{D}^{\gamma,\delta}_{p,T}(\GG^\eps,e^\sigma_l)}^2
\end{align*}
for all $\nu\in (0,\delta-\nu']$. 
In total we have 
\begin{align*}
\|\mathscr{L}_\mu v^{\sharp}\|_{\mathcal{M}^\gamma_T\mathcal{C}^{2\alpha+2\kappa/\sigma-2}_p(\GG^\eps,e^\sigma_l p^\kappa)}\lesssim \|u\|_{\mathscr{D}^{\gamma,\delta}_{p,T}(\GG^\eps,e^\sigma_l)} + \eps^{\nu} \|u\|_{\mathscr{D}^{\gamma,\delta}_{p,T}(\GG^\eps,e^\sigma_l)}^2\,,\quad v^{\eps,\sharp}(0)=\mathbf{1}_{\gamma=0} u_0^{\sharp}+\mathbf{1}_{\gamma=\alpha} u_0\,,
\end{align*}
where we used for the initial condition that by Definition \ref{def:ModifiedParaproduct} and Convention \ref{con:ConventionModifiedParaproduct} we have $(F'(0)u\mpara X_\mu)(0)=F'(0)u_0\para X$ for $\gamma=0$ and $(F'(0)u\mpara X_\mu)(0)=0$ for $\gamma=\alpha>0$.  The Schauder estimates of Lemma~\ref{lem:Schauder} yield on these grounds
\begin{align*}
	\|v^\sharp\|_{\mathscr{L}^{\gamma,2\alpha}_{p,T}(\GG^\eps,e^\sigma_l)}&\lesssim  \mathbf{1}_{\gamma=\alpha} \|u_0\|_{\mathcal{C}^0_p(\GG^\eps,e^\sigma_l)}
	+ \mathbf{1}_{\gamma=0} \|u_0^\sharp\|_{\mathcal{C}^{2\alpha}_p(\GG^\eps,e^\sigma_l)}+
	\|u\|_{\mathscr{D}^{\gamma,\delta}_{p,T}(\GG^\eps,e^\sigma_l)} + \eps^{\nu} \|u\|_{\mathscr{D}^{\gamma,\delta}_{p,T}(\GG^\eps,e^\sigma_l)}^2 \\
	&\lesssim  
	\mathbf{1}_{\gamma=\alpha} \|u_0\|_{\mathcal{C}^0_p(\GG^\eps,e^\sigma_l)}+\mathbf{1}_{\gamma=0}\left(\|u_0^\sharp\|_{\mathcal{C}^{2\alpha}_p(\GG^\eps,e^\sigma_l)}
+\|u^{\sharp}(0)\|_{\mathcal{C}_{p}^{2\alpha}(\GG^\eps,e^\sigma_l)}
+\|u^{X}(0)\|_{\mathcal{C}^{\alpha}_p(\GG^\eps,e^\sigma_l)}\right)  \\	
&\quad +	T^{(\alpha-\delta)/2} (\|u\|_{\mathscr{D}^{\gamma,\alpha}_{p,T}(\GG^\eps,e^\sigma_l)}	
	+ \eps^{\nu} \|u\|_{\mathscr{D}^{\gamma,\alpha}_{p,T}(\GG^\eps,e^\sigma_l)}^2) 
	\,,
\end{align*}
where in the last step we used Lemma \ref{lem:ComparisonParabolicSpaces}. Together with \eqref{eq:APrioriVX} the claim follows. 
\end{proof}

As we mentioned in Remark \ref{rem:FixPoint} we aim at finding fixed points of $\mathscr{M}^\eps_{\gamma,a_0}$ which is achieved by the following Corollary.

\begin{corollary}
\label{cor:FixPoint}
With the notation of Lemma \ref{lem:Picard} choose $T_\eps^{\mathrm{loc}}:=\frac{1}{2}\, (C_{M_\eps}+C_{M_\eps}\eps^{\nu} r( u_0))^{-2/(\alpha-\delta)} $ for a sufficiently large $r(u_0)>0$, depending on $u_0$. Then the map $\mathscr{M}^{\eps}_{\gamma,u_0}$ from Lemma \ref{lem:Picard} has a unique fixed point $u^\eps=u^{\eps,X}\mpara X_\mu^\eps+u^{\eps,\sharp}$ on $\mathscr{D}^{\gamma,\alpha}_{p,T_\eps^{\mathrm{loc}}}(\GGe,e^\sigma_l)$. This fixed point solves
\begin{align}
\label{eq:FixPointProblem}
\mathscr{L}^\eps_\mu u^\eps = F^\eps(u^\eps)(\xi^\eps - F'(0) c^\eps_\mu)\,,\qquad u^\eps(0)=u_0\,,
\end{align}
and $u^{\varepsilon,X} = F'(0) u^\varepsilon$. Moreover, we have
\begin{align*}
\|u^\eps\|_{\mathscr{D}^{\gamma,\alpha}_{p,T_\eps^{\mathrm{loc}}}(\GGe,e^\sigma_l)} \leq r(u_0)\,.
\end{align*}
\end{corollary}
\begin{proof}
We construct the fixed point $u^\eps$ by a Picard type iteration. To avoid notational clashes with the initial condition $u_0$, we start the iteration with $n=-1$ for which we define $u_{-1}^{\eps}:=F'(0) u_0\mpara X_\mu^\eps+u^\sharp_0=u_0\para X_\mu^\eps+u^\sharp_0=u_0$ for $\gamma=0$ and $u_{-1}^\eps:=0\mpara X_\mu^\eps+e^{tL^\eps_\mu}u_0$ for $\gamma=\alpha$ (which is in $\mathscr{D}_{p,T}^{\gamma,\alpha}(\GG^\eps,e^\sigma_l)$ due to Lemma \ref{lem:Schauder}).  Define recursively for $n\geq 0$ the sequence $u_n^{\eps}:=\mathscr{M}^{\eps}_{\gamma,u_0} u^\eps_{n-1}$ (with $u^\eps_{n}=u^{\eps,X}_{n}\mpara X_\mu^\eps+u_{n}^{\eps,\sharp}$ to be read as a pair as in Definition \ref{def:Paracontrolled}). Choose now $r(u_0)$ so big that $\|u_{-1}^\eps\|_{\mathscr{D}_{p,1}^{\gamma,\alpha}(\GG^\eps,e^\sigma_l)}\leq r(u_0)$ and such that
\begin{align*}
C_{u_0}\leq \frac{1}{2} r(u_0)
\end{align*}
with $C_{u_0}$ as in Lemma \ref{lem:Picard}. Note that for $u_{n+1}^\varepsilon$ the constant $C_{u_0}$ in principle depends on $u^\eps_n(0)$, but in fact we can choose it independently of $n$ since $u_n^{\eps,X}(0)=F'(0)u_0$ for all $n \ge -1$ (by definition of $\mathscr{M}^\eps_{\gamma,u_0}$) and $u_n^{\eps,\sharp}(0)=\mathbf{1}_{\gamma=0} u_0^\sharp+\mathbf{1}_{\gamma=\alpha} u_0$ (by Definition \ref{def:ModifiedParaproduct} and Convention \ref{con:ConventionModifiedParaproduct}) in the second term of \eqref{eq:Cu0}.

 Since $T_\eps^{\mathrm{loc}}\leq 1$ we already know for $n=-1$ that
\begin{align}
\label{eq:PicardEstimate}
\|u_{n}^\eps\|_{\mathscr{D}_{p,T_\eps^{\mathrm{loc}}}^{\gamma,\alpha}(\GG^\eps,e^\sigma_l)}\leq r(u_0) \,.
\end{align}
We show recursively that \eqref{eq:PicardEstimate} is in fact true for any $n\geq -1$. Suppose we have already shown the statement for $n-1$, we then obtain by Lemma \ref{lem:Picard}
\begin{align*}
\|u_n^\eps\|_{\mathscr{D}^{\gamma,\alpha}_{p,T_\eps^{\mathrm{loc}}}(\GG^\eps,e^\sigma_l)} &\leq C_{u_0}+(T_\eps^{\mathrm{loc}})^{\frac{\alpha-\delta}{2}}\cdot C_{M_\eps} (r(u_0)+\eps^\nu\,(r(u_0))^2)  \\
&\leq \frac{r(u_0)}{2}  + (T_\eps^{\mathrm{loc}})^{\frac{\alpha-\delta}{2}}\,(C_{M_\eps}+C_{M_\eps} \eps^\nu\,r(u_0)) \cdot r(u_0) =\frac{r(u_0)}{2}+\frac{r(u_0)}{2}=r(u_0)\,.
\end{align*}
By Lemma \ref{lem:Fatou} in the appendix inequality~\eqref{eq:PicardEstimate} implies that for $\alpha'\in (0,\alpha)$ and $\sigma'\in (0,\sigma)$ there is a subsequence $(u_{n_k}^\varepsilon)_{k\geq 0}$, convergent in $\mathscr{D}^{\gamma,\alpha'}_{p,T_\eps^{\mathrm{loc}}}(\GG^\eps,e^{\sigma'}_l)$ to some $u^\varepsilon\in \mathscr{D}^{\gamma,\alpha}_{p,T_\eps^{\mathrm{loc}}}(\GG^\eps,e^\sigma_l)$, and 
\begin{align*}
\|u^\eps\|_{\mathscr{D}^{\gamma,\alpha}_{p,T_\eps^{\mathrm{loc}}}(\GG^\eps,e^\sigma_l)}\leq \liminf_{k\rightarrow	\infty} \|u^\eps_{n_k}\|_{\mathscr{D}^{\gamma,\alpha}_{p,T_\eps^{\mathrm{loc}}}(\GG^\eps,e^\sigma_l)}\leq r(u_0)\,.
\end{align*}
In particular $u^\varepsilon$ is a fixed point of $\mathscr{M}^\eps_{\gamma,u_0}$ that satisfies \eqref{eq:FixPointProblem}. It remains to check uniqueness. Choose two fixed points $u^\eps,\,v^\eps$, which then satisfy
\begin{align*}
	\mathscr{L}_\mu^\eps(u^\eps-v^\eps)& {=}(F^\eps(u^\eps)-F^\eps(v^\eps))(\xi^\eps-c^\eps_\mu F'(0)) {=}\underbrace{\int_0^1 F'(u^\eps+\lambda(v^\eps-u^\eps))\dd \lambda}_{=:\mathscr{F}}\cdot (v^\eps-u^\eps)(\xi^\eps-c^\eps_\mu F'(0))\,.
\end{align*}
\vspace{-0.5 cm}

We will use that for $\rho\in \rr(\ww)$ and $\zeta,\,\zeta'\in \RR$ with $\zeta'\geq \zeta$
\begin{align}
\label{eq:PayEps}
\|f\|_{\cC^{\zeta'}_p(\GG^\eps,\rho)} \lesssim \eps^{-(\zeta'-\zeta)} \|f\|_{\cC^{\zeta}_p(\GG^\eps,\rho)}\,,
\end{align}
which is an easy consequence of Definition \ref{def:DiscreteBesov} and which we essentially already used in the proof of Lemma~\ref{lem:Picard}. In other words, we can consider our objects as arbitrarily ``smooth'' if we are ready to accept negative powers of $\eps$. In particular, we can consider the initial condition $u_0$ as paracontrolled, that is $u_0\in \cC^\alpha_p(\GG^\eps,e^\sigma_l),\,u_0^\sharp\in \cC^{2\alpha}_p(\GG^\eps,e^\sigma_l)$ (and thus $u^{\eps,X}(0)=v^{\eps,X}(0)=F'(0)u_0\in \cC_p^{\alpha}(\GG^\eps,e^\sigma_l)$), so that with Lemma \ref{lem:Picard}  we obtain $u^\eps,v^\eps\in \mathscr{D}^{0,\alpha}_{p,T_\eps^{\mathrm{loc}}}(\GG^\eps,e^\sigma_l)$. Consequently, since also $e^\sigma_l \ge 1$, we get $u^\eps,\,v^\eps\in C_{T_\eps^{\mathrm{loc}}} L^\infty(\GG^\eps)$ which implies that the integral term $\mathscr{F}$ is in $C_{T_\eps^{\mathrm{loc}}} L^\infty(\GG^\eps)$ and, by using once more \eqref{eq:PayEps}, we can consider it as an element of $C_{T_\eps^{\mathrm{loc}}}\mathcal{C}^{\beta}_{\infty}(\GG^\eps)$ for any $\beta\in\RR$. The product $(v^\eps-u^\eps)(\xi^\eps-c^\eps_\mu F'(0))$ can then be estimated as in the proof of Lemma \ref{lem:Picard}. Since multiplication by $\mathscr{F}$ only contributes an ($\eps$-dependent) factor we obtain for $T'\leq T_\eps^{\mathrm{loc}}$ a bound of the form
\begin{align*}
	\|u^\eps-v^\eps\|_{\mathscr{D}^{0,\alpha}_{p,T'}(\GG^\eps,e^\sigma_l)}
	\lesssim_\eps (T')^{\frac{\alpha-\delta}{2}} \|u^\eps-v^\eps\|_{\mathscr{D}^{0,\alpha}_{p,T'}(\GG^\eps,e^\sigma_l)}\,,
\end{align*}
which shows $\|u^\eps-v^\eps\|_{\mathscr{D}^{0,\alpha}_{p,T'}(\GG^\eps,e^\sigma_l)}=0$ for $T'$ small enough. Iterating this argument gives $u^\eps=v^\eps$ on all of $[0,T_\eps^{\mathrm{loc}}]$.
\end{proof}

\subsubsection*{Convergence to the continuum}
It is straightforward to redo our computations in the continuous linear case (i.e. $F(x) = c x$), which leads to the existence of a solution to the continuous linear parabolic Anderson model on $\RR^2$, a result which was already established in \cite{HairerLabbeR2}. Since the continuous analogue of our approach is a one-to-one translation of the discrete statements and definitions above from $\GG^\varepsilon$ to $\RR^d$ we do not provide the details.

\begin{corollary}
\label{cor:ContinuousSolution}
Let $u_0\in \mathcal{C}^0_p(\RR^d,e^\sigma_l)$. Let $\xi$ be a white noise on $\RR^2$, and let $\mathscr{L}_\mu$ be defined as in Section~\ref{sec:DiffusionOperators}. Then there is a unique solution $u=F'(0)u\mpara X_\mu+u^\sharp\in \mathscr{D}^{\alpha,\alpha}_{p,T}(\RR^d,e^\sigma_l)$ to 
\begin{align}
\label{eq:LinearPAM}
	\mathscr{L}_\mu u=F'(0)u\blacklozeng \xi,\qquad u(0)=u_0\,,
\end{align} 
\glsadd{blacklozenge}
on $[0,T]$, where
\begin{align*}
	u\blacklozeng \xi:=  \xi \para u+ u\para\xi + F'(0) C(u,X_\mu,\xi)+ F'(0)u\,(X_\mu\bullet \xi)+ u^\sharp\reso \xi
\end{align*}
with $X_\mu,\,X_\mu\bullet \xi$ as in \eqref{eq:X}, \eqref{eq:Xxi}. 
\end{corollary}	
	
\begin{proof}[Sketch of the proof]
As in Lemma \ref{lem:Picard} we can build a map $\mathscr{M}_{\alpha,u_0}\colon\,\mathscr{D}^{\alpha,\alpha}_{p,T}(\RR^d,e^\sigma_l)\rightarrow \mathscr{D}^{\alpha,\alpha}_{p,T}(\RR^d,e^\sigma_l):\,u=u^{X_\mu}\mpara X_\mu + u^{\sharp} \mapsto v= F'(0)u\mpara X_\mu + v^\sharp$ via
\begin{align}
\label{eq:LinearFixedPointEquation}
\mathscr{L}_\mu v:= F'(0) u \blacklozeng \xi\,,\qquad v(0)=u_0\,.
\end{align}
As in Corollary \ref{cor:FixPoint} there is a time $T^{\mathrm{loc}}$ such that $\mathscr{M}_{\alpha,u_0}$ has a (unique) fixed point $u^{(0)}=F'(0)u^{(0)}\mpara X_\mu+u^{(0),\sharp}$ in $\mathscr{D}^{\alpha,\alpha}_{p,T^{\mathrm{loc}}}(\RR^d,e^\sigma_l)$ that solves 
\begin{align*}
\mathscr{L}_\mu u^{(0)} = F'(0) u^{(0)}\blacklozeng \xi\,,\qquad u^{(0)}(0)=u_0 \,.
\end{align*}
on $[0,T^{\mathrm{loc}}]$. Since the right hand side of \eqref{eq:LinearFixedPointEquation} is linear, this time can be chosen of the form $T^{\mathrm{loc}}=\frac{1}{2}\, K^{-2/(\alpha-\delta)}$, where $K>0$ is a (random) constant that only depends on $\xi, X_\mu, X_\mu \bullet \xi$, but not on the initial condition. Proceeding as above but starting in $u^{(0)}(T^{\mathrm{loc}})$ we can construct  
a map $\mathscr{M}_{0,u^{(0)}(T^{\mathrm{loc}})} \colon \,\mathscr{D}^{0,\alpha}_{p,T^{\mathrm{loc}}}(\RR^d,e^\sigma_l) \to \mathscr{D}^{0,\alpha}_{p,T^{\mathrm{loc}}}(\RR^d,e^\sigma_l)$ by (the continuous version of) Lemma \ref{lem:Picard} and Lemma \ref{lem:CommuteModifiedParaproduct}. The map $\mathscr{M}_{0,u^{(0)}(T^{\mathrm{loc}})}$ has again a fixed point on $[0,T^{\mathrm{loc}}]$ which we call $u^{(1)}$. Starting now in $u^{(1)}(T^{\mathrm{loc}})$ we can construct $u^{(2)}$ as the fixed point of $\mathscr{M}_{0,u^{(1)}(T^{\mathrm{loc}})}$ on $[0,T^{\mathrm{loc}}]$ and so on. As in \cite[Theorem 6.12]{Gubinelli2017KPZ}) the sequence of local solutions $u^{(0)},\,u^{(1)},u^{(2)},\ldots$ can be concatenated to a paracontrolled solution $u=F'(0)u\mpara X_\mu+u^\sharp\in \mathscr{D}^{\alpha,\alpha}_{p,T}(\RR^d,e^\sigma_l)$ on $[0,T]$. 

To see uniqueness take two solutions $u,\,v$ in $\mathscr{D}^{\alpha,\alpha}_{p,T}(\RR^d,e^\sigma_l)$ and consider $h=u-v$. Using that $h(0)=0$ and $\mathscr{L}_\mu h = h\blacklozeng \xi$ one derives as in Lemma \ref{lem:Picard}
\begin{align*}
\|h\|_{\mathscr{D}_{p,T}^{\alpha,\alpha}(\RR^d,e^\sigma_l)}\leq  C \cdot T^{(\alpha-\delta)/2}\,\|h\|_{\mathscr{D}_{p,T}^{\alpha,\alpha}(\RR^d,e^\sigma_l)}
\end{align*}
so that choosing $T$ first small enough and then proceeding iteratively yields $h=0$. 
\end{proof}

We can now deduce the main theorem of this section. The parameters are as on page \pageref{eq:ParameterCondition1}. 

\begin{theorem}
\label{thm:ConvergencePAM}
Let $u^\eps_0$ be a uniformly bounded sequence in $\mathcal{C}^0_p(\GG^\eps,e^\sigma_l)$ such that $\EE^\eps u^\eps_0$ converges to some $u_0$ in $\Sw_\omega'(\RR^2)$. Then there are unique solutions $u^\eps\in \mathscr{D}^{\alpha,\alpha}_{p,T^\eps}(\GG^\eps,e^\sigma_l)$ to 
\begin{align}
	\label{eq:ConvergencePAMDiscreteEquation}
	\mathscr{L}_\mu^\eps u^\eps=F^\eps(u^\eps)(\xi^\eps - c^\eps_\mu F'(0)),\qquad u^\eps(0)=u^\eps_0,
\end{align}
on $[0,T^\eps]$ with random times $T^\eps\in (0,T]$ that satisfy $\mathbb{P}(T^\eps=T)\overset{\eps \rightarrow 0}{\longrightarrow} 1$. The sequence $u^\eps=F'(0)u^\eps\mpara X_\mu+u^{\eps,\sharp}\in\mathscr{D}^{\alpha,\alpha}_{p,T^\eps}(\GG^\eps,e^\sigma_l)$ is uniformly bounded and the extensions $\EE^\eps u^\eps$ converge in distribution in $\mathscr{D}^{\alpha,\alpha'}_{p,T}(\RR^d,e^{\sigma'}_l)$, $\alpha'<\alpha,\,\sigma'<\sigma$, to the solution $u$ of the linear equation in Corollary \ref{cor:ContinuousSolution}.
\end{theorem}
\begin{remark}
\label{rem:ConvergenceInDistribution}
Since $T^\eps$ is a random time for which it might be true that $P(T^\eps<T)>0$ the convergence in distribution has to be defined with some care: We mean by $\EE^\eps u^\eps\rightarrow u$ in distribution that for any $f\in C_b (\mathscr{D}^{\alpha,\alpha'}_{p,T}(\GG^\eps,e^\sigma_l);\RR)$, we have $\mathbb{E}[f(\EE^\eps u^\eps)\mathbf{1}_{T^\eps = T}]\rightarrow \mathbb{E}[f(u)]$ and further $\mathbb{P}(T^\eps < T)\rightarrow 0$. 
\end{remark}

\begin{proof}
The \emph{local} existence of a solution to \eqref{eq:ConvergencePAMDiscreteEquation} is provided by Corollary \ref{cor:FixPoint}. Proceeding as in the proof of Corollary \ref{cor:ContinuousSolution} we can in fact construct a sequence of local solutions $(u^{\eps,(n)})_{n\geq 0}$ on intervals $[0,T_\eps^{\mathrm{loc},(n)}]$ with $u^{\eps,(n)}(0)=u^{\eps,(n-1)}(T_\eps^{\mathrm{loc},(n-1)})$, where we set $T_\eps^{\mathrm{loc},(-1)}:=0$ and $u^{\eps,(-1)}:=u_0$. Due to Corollary \ref{cor:FixPoint} the time $T_\eps^{\mathrm{loc},(n)}$ is given by 
\begin{align}
\label{eq:ConcatenationTime}
T_\eps^{\mathrm{loc},(n)}:=\frac{1}{2}\, \left(C_{M_\eps}+C_{M_\eps}\eps^{\nu} r\big( u^{\eps,(n-1)}(T_\eps^{\mathrm{loc},(n-1)})\big)\right)^{-2/(\alpha-\delta)} \,.
\end{align}
Note that, in contrast to the proof of Corollary \ref{cor:ContinuousSolution}, $T_\eps^{\mathrm{loc},(n)}$ now really depends on $n$ and we might have $\sum_{n\geq 0} T_\eps^{\mathrm{loc},(n)}<\infty$. As in \cite[Theorem 6.12]{Gubinelli2017KPZ} we can concatenate the sequence $u^{\eps,(0)},u^{\eps,(1)},\ldots$ to a solution $u^\eps$ to \eqref{eq:ConvergencePAMDiscreteEquation} which is defined up to its ``blow-up'' time
\begin{align*}
 T_\eps^{\mathrm{blow-up}}=\sum_{n\geq 0} T_\eps^{\mathrm{loc},(n)}
 \end{align*} 
(which might be larger than $T$ or even infinite). Let us set
\begin{align}
T^\eps:= T\wedge \frac{T_\eps^{\mathrm{blow-up}}}{2}\,.
\end{align}
To show $\mathbb{P}(T^\eps=T)\overset{\eps \rightarrow 0}{\longrightarrow} 1$ we prove that for any $t>0$ we have $\mathbb{P}(T_\eps^{\mathrm{blow-up}}<t)\rightarrow 0$. By inspecting the definition of $r(\ldots)$ in the proof of Corollary \ref{cor:FixPoint} we see that given the (bounded) sequence of initial condition $u_0^\eps$ the size of $T_\eps^{\mathrm{blow-up}}$ can be controlled by the quantity $M^\eps$. More precisely there is a deterministic, decreasing function $T_\eps^{\mathrm{det}}:\,\RR_+\rightarrow \RR_+$ such that
\begin{align*}
T_\eps^{\mathrm{blow-up}} \geq T_\eps^{\mathrm{det}}(M^\eps)
\end{align*}
and such that for any $K>0$ (due to the presence of the factor $\eps^\nu$ in \eqref{eq:ConcatenationTime})
\begin{align}
\label{eq:TdetExpl}
T_\eps^{\mathrm{det}}(K)\overset{\eps\rightarrow 0}{\longrightarrow} \infty\,.
\end{align}
Let $t>0$ and $K^\eps_{t}:=\sup\{K>0\,\vert\,T_\eps^{\mathrm{det}}(K) \geq  t\}$. Note that we must have $K^\eps_{t}\overset{\eps\rightarrow 0}{\longrightarrow} \infty$ since otherwise we contradict \eqref{eq:TdetExpl}. But this already implies the desired convergence:
\begin{align*}
\mathbb{P}(T_\eps^{\mathrm{blow-up}}< t)\leq \mathbb{P}(T_\eps^{\mathrm{det}}(M^\eps)< t)\leq \mathbb{P}(M^\eps\geq K^\eps_{t})\overset{K^\eps_{t} \rightarrow \infty}{\longrightarrow} 0\,,
\end{align*}
where we used in the last step the boundedness of the moments of $M^\eps$ due to Lemma \ref{lem:RegularityNoise}.

It remains to show that the extensions $\EE^\eps u^\eps$ converge to $u$. By Skohorod representation we know that $\EE^\eps \xi^\eps,\,\EE^\eps X_\mu^\eps,\,\EE^\eps (X_\mu^\eps\bullet \xi^\eps)$ in Lemma \ref{lem:LimitArea} converge almost surely on a suitable probability space. We will work on this space from now on. The application of the Skohorod representation theorem is indeed allowed since the limiting measure of these objects has support in the closure of smooth compactly supported functions and thus in a separable space. We can further assume by Skohorod representation that (a.s.) $T_\eps^{\mathrm{blow-up}}\rightarrow \infty$ so that almost surely we have $T^\eps=T$ for all $\varepsilon \le \varepsilon_0$ with some (random) $\varepsilon_0$. Having proved that the sequence $u^\eps$ is uniformly bounded in $\mathscr{D}^{\alpha,\alpha}_{p,T^\eps}(\GG^\eps,e^\sigma_l)$ we know, by Lemma \ref{lem:ExtensionOperator}, that $\EE^\eps u^\eps$ is uniformly bounded in $\mathscr{D}^{\alpha,\alpha}_{p,T^{\eps}}(\RR^d,e^\sigma_l)$. Due to (the continuous version of) Lemma \ref{lem:Fatou} there is at least a subsequence of $
 \EE^{\eps_n} u^{\eps_n}$ that converges to some $u\in \mathscr{D}^{\alpha,\alpha}_{p,T}(\RR^d,e^\sigma_l)(\RR^d)$ in the topology of $\mathscr{D}^{\alpha,\alpha'}_{p,T}(\RR^d,e^{\sigma'}_l)$. If we can show that this limit solves \eqref{eq:LinearPAM} we can argue by uniqueness that (the full sequence) $\EE^\eps u^\eps$ converges to $u$.  
We have
\begin{align}
\label{eq:ProofDiscreteEquationExtended}
\mathscr{L}^{\eps_n}_\mu \EE^{\eps_n} u^{\eps_n}=\EE^{\eps_n} \mathscr{L}^{\eps_n}_\mu u^{\eps_n} =\EE^{\eps_n}(F^{\eps_n}(u^{\eps_n})(\xi^{\eps_n} -c^{\eps_n}_\mu F'(0)))\,,
\end{align}
where $\mathscr{L}^\eps_\mu \EE^\eps u^\eps$ should be read as in \eqref{eq:DiscretizedActionDiffusionOperator}. Note that the left hand side of \eqref{eq:ProofDiscreteEquationExtended} converges as
\begin{align*}
\mathscr{L}^\eps_\mu \EE^{\eps_n} u^{\eps_n} = (\mathscr{L}^{\eps_n}_\mu - \mathscr{L}_\mu)\EE^{\eps_n} u^{\eps_n}+\mathscr{L}_\mu\EE^{\eps_n} u^{\eps_n} \,\, \overset{{\eps_n}\rightarrow 0}{\longrightarrow} \,\, 0+\mathscr{L}_\mu u=\mathscr{L}_\mu u
\end{align*}
due to Lemma \ref{lem:LaplaceRegularity}. For the right hand side of \eqref{eq:ProofDiscreteEquationExtended} we apply the same decomposition as in \eqref{eq:ParacontrolledSchauderToDecompose}=\eqref{eq:ParacontrolledSchauder1}+\eqref{eq:ParacontrolledSchauder2}+\eqref{eq:ParacontrolledSchauder3}+\eqref{eq:ParacontrolledSchauder4}+\eqref{eq:ParacontrolledSchauder5}. While (the extensions of) the terms \eqref{eq:ParacontrolledSchauder4},\eqref{eq:ParacontrolledSchauder5} of \eqref{eq:ParacontrolledSchauderToDecompose} vanish as $\eps$ tends to $0$, we can use the property \eqref{eq:EProperty} of the operators acting in the terms \eqref{eq:ParacontrolledSchauder1}, \eqref{eq:ParacontrolledSchauder2}, \eqref{eq:ParacontrolledSchauder3} to identify their limits. Consider for example the product $u^{\eps,X^\eps_\mu}(X_\mu^\eps \bullet \xi^\eps)= F'(0) u^\eps (X_\mu^\eps \bullet \xi^\eps)$ in \eqref{eq:ParacontrolledSchauder2} whose extension we can rewrite as 
\begin{align*}
& \EE^{\eps_n}\big( F'(0) u^{\eps_n} (X_\mu^{\eps_n} \bullet \xi^{\eps_n}) \big) = F'(0)\, \EE^{\eps_n} \big(u^{\eps_n} \para (X_\mu^{\eps_n} \bullet \xi^{\eps_n}) +u^{\eps_n} \lpara (X_\mu^{\eps_n} \bullet \xi^{\eps_n}) +u^{\eps_n} \reso (X_\mu^{\eps_n} \bullet \xi^{\eps_n}) \big) \\
   &\overset{\eqref{eq:EProperty}}{=}F'(0)\big[\EE^{\eps_n} u^{\eps_n} \para \EE^{\eps_n}(X_\mu^{\eps_n} \bullet \xi^{\eps_n}) +\EE^{\eps_n} u^{\eps_n} \lpara \EE^{\eps_n}(X_\mu^{\eps_n} \bullet \xi^{\eps_n}) + \EE^{\eps_n} u^{\eps_n} \reso \EE^{\eps_n}(X_\mu^{\eps_n} \bullet \xi^{\eps_n})\big]+o_{\eps_n}(1)\,,
\end{align*}
where we applied the property \eqref{eq:EProperty} of $ \para ,\,\lpara ,\,\reso $ (Lemma \ref{lem:ParaproductEstimates}) in the second step. By continuity of the involved operators and Lemma \ref{lem:LimitArea} we thus obtain 
\begin{align*}
\lim_{\eps_n\rightarrow 0}\EE^{\eps_n}\big( F'(0) u^{\eps_n} (X_\mu^{\eps_n} \bullet \xi^{\eps_n}) \big) = F'(0) \big[ u\para (X\bullet \xi)+  u\lpara (X\bullet \xi)+ u\reso (X\bullet \xi\big)]=F'(0)u (X\bullet \xi)\,.
\end{align*}
Proceeding similarly for all terms in the decomposition of the right hand side of \eqref{eq:ProofDiscreteEquationExtended} one arrives at 
\begin{align*}
\mathscr{L}_\mu u=\lim_{\eps_n\rightarrow 0}\EE^{\eps_n} \mathscr{L}^{\eps_n}_\mu u^{\eps_n} = \lim_{\eps_n\rightarrow 0}\EE^{\eps_n}(F^{\eps_n}(u^{\eps_n})(\xi^{\eps_n} -c^{\eps_n}_\mu F'(0)))=F'(0) u\blacklozeng\xi\,,
\end{align*}
which finishes the proof. 
\end{proof}

Since the weights we are working with are increasing, the solutions $u^\eps$ and the limit $u$ are actually classical tempered distributions. However, since we need the $\Sw_\omega$ spaces to handle convolutions in $e^\sigma_l$ weighted spaces it is natural to allow for solutions in $\Sw'_\omega$. 
In the linear case, $F=\mathrm{Id}$, we can allow for sub-exponentially growing initial conditions $u_0$ since the only reason for choosing the parameter $l$ in the weight $e_{l+t}^\sigma$ smaller than $-T$ was to be able to estimate $e^\sigma_{l+t}\leq (e^\sigma_{l+t})^2$ to handle the quadratic term. In this case the solution will be a genuine ultra-distribution.

\appendix
\section{Appendix}

\subsection*{Results related to Section \ref{sec:BravaisLattices}}

\begin{lemma}
\label{lem:FourierAnalysisUltraDistributions}
The mappings $(\FF_\GG,\FF^{-1}_\GG)$ defined in Subsection~\ref{subsec:LatticeUltraDistributions} map the spaces $(\Sw_\omega(\GG),\,\Sw_\omega(\widehat{\GG}))$ and $(\Sw'_\omega(\GG),\,\Sw'_\omega(\widehat{\GG}))$ to each other.  
\end{lemma}

\begin{proof}
We only consider the non-standard case $\omega=|\cdot|^\sigma$. Given  $f \in \Sw_\omega(\widehat{\GG})$ the sequence
\begin{align*}
	\FF_\GG f(x)=|\GG|\sum_{k\in \GG} f(k) e^{2\pi \imath k x} 
\end{align*}
obviously converges to a smooth function that is periodic on $\widehat{\GG}$. We estimate on $\widehat{\GG}$ (and thus by periodicity uniformly on $\RR^d$)
\begin{align*}
	\left|\partial^\alpha \sum_{k\in \GG} |\GG| f(k) e^{2\pi \imath k x} \right| &\lesssim_\lambda  
	\sum_{k\in \GG} |\GG|  |k|^{|\alpha|} e^{-\lambda |k|^\sigma}	 
\end{align*}
We can use Lemma~\ref{lem:IntegralTest} for $|\cdot|^{|\alpha|} e^{-\lambda |\cdot|^{\sigma}}$ with $\Omega=\GG$ and $c>0$ of the form $c=C(\lambda) \cdot C^{|\alpha|}$ ($C$ denoting a positive constant that may change from line to line) which yields   
\begin{align*}
	\left|\partial^\alpha \sum_{k\in \GG} |\GG| f(k) e^{2\pi \imath k x} \right| &\lesssim_\lambda C^{|\alpha|}   \int_{\RR^d} |x|^{|\alpha|} e^{-\lambda|x|^{\sigma}} \dd x
\end{align*}
We now proceed as in \cite[Lemma 12.7.4]{HoermanderII} and estimate the integral by the $\Gamma-$function 
\begin{align*}
	&\int_{\RR^d} |x|^{|\alpha|} e^{-\lambda|x|^{\sigma}} \dd x \lesssim \int_0^\infty r^{|\alpha|+d-1} e^{-\lambda r^{\sigma}} \dd r \lesssim_\lambda \lambda^{- |\alpha|/\sigma} \int_0^\infty r^{|\alpha|+d-1} e^{-r^{\sigma}} \dd r \\
	&\lesssim \lambda^{- |\alpha|/\sigma} \Gamma((|\alpha|+d-1)/\sigma) \overset{\mbox{Stirling}}{\lesssim} \lambda^{-|\alpha|/\sigma} C^{|\alpha|} |\alpha|^{|\alpha|/\sigma}\,.
\end{align*}
Since we can choose $\lambda>0$ arbitrarily large we see that indeed $f\in C^\infty_\omega(\widehat{\GG})$.

For the opposite direction, $f\in \Sw_\omega(\widehat{\GG})$, we use that by integration by parts $|z_i^l \cdot  \FF^{-1}_\GG f(z) | \lesssim C^l \sup_{\widehat{\GG}} (\partial^i)^l f
	\lesssim C^l \eps^l l^{l/\sigma}$ for all $z\in \GG,\,l\geq 0,\,i=1,\ldots,d$. With Stirling's formula and Lemma~\ref{lem:PolynomialBoundGivesSubexponential} we then obtain $|\FF_\GG^{-1} f(z)|\lesssim  e^{\lambda |z|^\sigma}$.
This shows the statement for the pair $(\Sw_\omega(\GG),\,\Sw_\omega(\widehat{\GG}))$. The estimates above show that $\FF_\GG,\FF^{-1}_\GG$ are in fact continuous w.r.t to the corresponding topologies so that the statement for the dual spaces $(\Sw'_\omega(\GG),\,\Sw'_\omega(\widehat{\GG}))$ immediately follows.
\end{proof}

\begin{lemma}
\label{lem:IntegralTest}
Given a lattice $\GG$ as in \eqref{eq:Lattice} we denote the translations of the closed parallelotope $G:=[0,1]a_1 +\ldots + [0,1]a_d$ by $	\mathbb{G}:=\{g+G\,\vert \, g\in \GG\}$. Let $\Omega\subseteq \GG$ and set $
 \overline{\Omega}:=\bigcup_{G'\in \mathbb{G},\, G'\cap \Omega \neq \emptyset} G'\,.
$
If for a measurable function $f:\overline{\Omega} \rightarrow \RR_+$ there exists $c\geq 1$ such that for any $g\in \Omega$ there is a $G'(g)\in \mathbb{G},\, g\in G'(g)$ with $f(g)\leq c \cdot \mathrm{ess\, inf\,}_{x\in G'} f(x)$ 
then it also holds
\begin{align*} 
	\sum_{g\in \Omega} |\GG| f(g) \leq c\cdot 2^{d} \int_{\overline{\Omega}} f(x) \dd x\,.
\end{align*}
\end{lemma}
\begin{proof}
Indeed
	\begin{align*}
		&\sum_{g\in \Omega} |\GG| f(g)  \leq c \sum_{g\in \Omega} \int_{G'(g)} f(x) \dd x \leq c \sum_{g\in \Omega} \sum_{G'\in \mathbb{G}:\,g\in G'} \int_{G'(g)} f(x) \dd x \\
			 &\le  c \sum_{G' \subseteq \overline{\Omega}} \,\,\sum_{g\in \Omega: g\in G'} \int_{G'} f(x) \dd x \overset{(\triangle)}{=} 2^d c \sum_{G' \in \overline\Omega} \int_{G'} f(x) \dd x = 2^d c \int_{\overline{\Omega}} f(x) \dd x\,,
	\end{align*}
where we used in $(\triangle)$ that the $d$-dimensional parallelotope has $2^d$ vertices. 
\end{proof}


\begin{lemma}[Mixed Young inequality]
\label{lem:MixedYoungInequality}

For $f\colon\RR^d\rightarrow \mathbb{C}$ and $g\colon \GG\rightarrow \mathbb{C}$ we set for $x\in \RR^d$
\begin{align*}
	f\ast_\GG g(x):=\sum_{k\in \GG} |\GG| f(x-k) g(k)
\end{align*}
Then for $r,p,q\in [1,\infty]$ with $1+1/r=1/p+1/q$ 
\begin{align*}
\|f\ast_\GG g\|_{L^r(\RR^d)}\le \sup_{x\in \RR^d} \| f(x-\cdot) \|^{1-\frac{p}{r}}_{L^p(\GG)}\cdot \|f\|_{L^p(\RR^d)}^{\frac{p}{r}} \|g\|_{L^q(\GG)}
\end{align*}
(with the convention $1/\infty=0,\,\infty/\infty=1$). 
\end{lemma}
\begin{proof}
We assume $p,q,r\in (1,\infty)$. The remaining cases are easy to check. The proof is based on Hölder's inequality on $\GG$ with $\frac{1}{r}+\frac{1}{\frac{rp}{r-p}}+\frac{1}{\frac{rq}{r-q}}=1$
\begin{align*}
	|f\ast_\GG g(x)| &\leq \sum_{k\in \GG} |\GG| \left( |f(x-k)|^p |g(k)|^q \right)^{1/r} \cdot |f(x-k)|^{\frac{r-p}{r}} |g(k)|^{\frac{r-q}{r}} \\
	&\overset{\text{Hölder}}{\leq} \left\|  \left( |f(x-\cdot)|^p |g(\cdot)|^q  \right)^{1/r}  \right\|_{L^r(\GG)} \cdot \| |f(x-\cdot)|^{\frac{r-p}{r}}  \|_{L^{\frac{rp}{r-p}}(\GG)}\cdot \||g(\cdot)|^{\frac{r-q}{r}}\|_{L^{\frac{rq}{r-q}}(\GG)} \\
	&\le \left( \sum_{k\in\GG} |\GG| (|f(x-k)|^p |g(k)|^q \right)^{1/r} \sup_{x'\in \RR^d} \|f(x'-\cdot)\|^{\frac{r-p}{r}}_{L^p(\GG)} \|g\|^{\frac{r-q}{r}}_{L^q(\GG)}\,.
\end{align*}
Raising this expression to the $r$th power and integrating it shows the claim. 
\end{proof}

\subsection*{Results related to Section \ref{sec:DiffusionOperators}}

\begin{lemma}
\label{lem:SchauderLp}
For $T\geq 0$, $p\in [1,\infty]$, $\rho\in \boldsymbol{\rho}(\omega)$ we have uniformly in $t \in [0,T]$ and $\varepsilon \in (0,1]$
\begin{align*}
	\|e^{tL^\eps_\mu} f \|_{L^p(\GG^\eps,\rho)} \lesssim \|f\|_{L^p(\GG^\eps,\rho)}\,,
\end{align*}
and for $\beta>0$
\begin{align*}
	\|e^{t L^\eps_\mu} f\|_{L^p(\GG^\eps,\rho)} \lesssim t^{-\beta/2} \|f\|_{\mathcal{C}^{-\beta}_p(\GG^\eps,\rho)}\,.
\end{align*}
\end{lemma}
\begin{proof}
With the random walk $(X_t^\eps)_{t\in \RR_+}$ which is generated by $L^\eps_\mu$ on $\GG^\eps$ we can express the semigroup as $
	e^{t L^\eps_\mu} f(x)=\mathbb{E}[f(x+X_t^\eps)]$, so that
\begin{align*}
	&\| \rho e^{t L^\eps_\mu}f \|_{L^p(\GG^\varepsilon)}  = \left\| \mathbb{E}\left[\frac{\rho(\cdot)}{\rho(\cdot+X^\varepsilon_t)} \rho(\cdot + X^\varepsilon_t) f(\cdot +X_t^\eps)\right] \right\|_{L^p(\GG^\varepsilon)} \\
	&\hspace{20pt} \le \mathbb{E} \left[ \sup_{x \in \GG^\varepsilon} \frac{\rho(x)}{\rho(x+X^\varepsilon_t)} \| f \|_{L^p(\GG^\varepsilon, \rho)}\right]  \lesssim \mathbb{E}[e^{\lambda \omega(X^\varepsilon_t)}] \| f \|_{L^p(\GG^\varepsilon, \rho)}
\end{align*}
%
%
An application of the next lemma finishes the proof of the first estimate. 
The second estimate follows as in Lemma 6.6. of \cite{Gubinelli2017KPZ}.
\end{proof}

\begin{lemma}
\label{lem:MomentsRandomWalk}
The random walk generated by $L^\eps_\mu$ on $\GG^\eps$ satisfies for any $\lambda >0$ and $t\in [0,T]$ 
\begin{align*}
\mathbb{E}[e^{\lambda \omega(X_t^\eps)}]\lesssim_{\lambda,T} 1\,.
\end{align*}
\end{lemma}
\begin{proof}
We assume $\omega=\wexps$, if $\omega$ is of the polynomial form the proof follows by similar, but simpler arguments. 
In this proof we write shorthand $s=1/\sigma$.  By the L\'evy-Khintchine-formula we have $\mathbb{E}[e^{\imath \theta X_{t}^\eps}]=e^{-t/\eps^2 \int_{\GG} (1-e^{\imath \theta \eps x})\dd \mu(x)} =e^{-t l^\eps_\mu(\theta)}$ for all $\theta\in\RR$. 
We want to bound first for $k\geq 1$
\begin{align*}
	\mathbb{E}[|X_{t,1}^\eps|^k+\ldots+|X_{t,d}^\eps|^k]=\sum_{j=1}^d \left|\partial_{\theta_j}^k \vert_{\theta=0} \mathbb{E}[e^{\imath \theta X_t^\eps}]\right|\,.
\end{align*}
To this end we apply Fa\'a-di-Brunos formula with $u(v)=e^{-tv},\,v(\theta)=l^\eps_\mu(\theta)$.
Note that with Lemma \ref{lem:SchauderPrelude} for $m \in \NN$ and $j=1,\ldots,d$
\begin{align*}
	u^{(m)}(0) &=(-t)^m \\ 
	|\partial^{m}_{\theta_j}v(0)| &\lesssim_\delta \delta^{m} (m!)^s.
\end{align*}
Thus with $A_{m,k}=\{(\alpha_1,\ldots,\alpha_m)\in \NN^m\,\vert\, \sum_{i=1}^m \alpha_i \cdot  i =k\}$ we get for any $\delta\in (0,1]$
\begin{align*}
	&\left|\partial_{\theta_j}^k \vert_{\theta=0} \mathbb{E}[e^{\imath \theta X_t^\eps}]\right| =\left| \sum_{1\leq m\leq k,\,\alpha\in A_{m,k} } \frac{k!}{\alpha!} u^{(m)}(0) \prod_{i=1}^m \left(\frac{1}{i!}\partial^{i}_{\theta_j} v(0)\right)^{\alpha_i}  \right| \\
& \lesssim_\delta \sum_{1\leq m\leq k ,\,\alpha \in A_{m,k} } \frac{k!}{\alpha!} t^m \prod_{i=1}^m (i!)^{\alpha_i(s-1)} \delta^{i\cdot \alpha_i} \overset{\mathrm{Stirling}}{\leq} \delta^k C^k \sum_{1\leq m\leq k,\,\alpha \in A_{m,k}} \frac{k!}{\alpha!} t^m 
\prod_{i=1}^m C^{i\alpha_i} i^{i\alpha_i(s-1)} \\
&\overset{i\leq m\leq k}{\leq} \delta^k C^k \sum_{1\leq m\leq k,\,\alpha \in A_{m,k}} \frac{k!}{\alpha!} t^m 
 k^{k(s-1)} \overset{\mathrm{Stirling}}{\leq }  \delta^k C^k \sum_{1\leq m\leq k,\,\alpha \in A_{m,k}} \frac{(k!)^s}{\alpha!} t^m \\ 
&\overset{(\alpha!)^{-1}\leq 1}{\leq}  \delta^k C^k (k!)^s  \sum_{1\leq m\leq k} |A_{m,k}|\, t^m = \delta^k C^k (k!)^s  \sum_{1\leq m\leq k} \binom{k-1}{m-1} t^m \\ &= \delta^k C^k (k!)^s t (1+t)^{k-1} \leq \delta^k C^k (k!)^s (1+t)^k \,,
\end{align*}
where $C>0$ denotes as usual a generic constant that changes from line to line.  
With $|x|_k^k:=|x_1|^k+\ldots+|x_d|^k$ we get
\begin{align*}
	\mathbb{E}[|X^\eps_t|_k^k] \lesssim \delta^k C^k (k!)^s (1+t)^k
\end{align*}
and therefore, using once more Stirling's formula and $|x|^k \lesssim C^k \cdot |x|_k^k$,
\begin{align*}
	\mathbb{E}[e^{\lambda |X_t^\eps|^{\sigma}}] &\lesssim 1+ \mathbb{E}[e^{\lambda|X_t^\eps|^{\sigma}} \mathbf{1}_{|X_t^\eps|\geq 1}] \leq 1+\sum_{k=0}^\infty \frac{\lambda^k}{k!} \mathbb{E}[|X^\eps_t|^{\lceil k \sigma \rceil}] \\
	&\lesssim 1+\sum_{k=0}^\infty \frac{C^k (1+t)^{\lceil k \sigma \rceil}}{k^k} \delta^{\lceil k \sigma\rceil} \lceil k \sigma \rceil^{\lceil k \sigma \rceil s} \lesssim 1+ (1+t)\sum_{k=0}^\infty \frac{C^k \delta^{k \sigma} (1+t)^{k \sigma}}{k^k} k^k \lesssim 1\,,
\end{align*}
where in the last step we chose $\delta>0$ small enough to make the series converge. 
\end{proof}

\subsection*{Results related to Section \ref{sec:ParacontrolledAnalysisonBravaisLattices}}

\begin{lemma}
\label{lem:SupportParaproduct}
Let $\GG^\eps$ as in Definition \ref{def:LatticeSequence}, let $\omega \in \ww$, and let $(\varphi_j^{\GG^\eps})_{j=-1,\ldots,j_{\GGe}}$ be a partition of unity as on page~\pageref{DiscreteDyadicPartition}.
For $-1\leq i\leq j\leq j_{\GGe}$ the function
\begin{align*}
\varDelta_i^{\GGe} f_1 \cdot \varDelta_j^{\GGe} f_2 \in \Sww'(\GG^\eps)
\end{align*}
is spectrally supported in a set of the form $2^j \B \cap \wGGe$, where $\B$ is a ball around $0$ that can be chosen independently of $i,\,j$ and $\eps$. For $f_1,\,f_2\in \Sww'(\GGe)$ and $0<j\leq j_{\GG^\eps}$ the function
\begin{align*}
S_{j-1}^{\GGe} f_1 \cdot \varDelta_j^{\GGe} f_2 \in \Sww'(\GG^\eps)\,, 
\end{align*}
is spectrally supported in a set of the form $2^j \rA \cap \wGGe$, where $\rA$ is an annulus around $0$ that can be chosen independently of $j$ and $\eps$. 
\end{lemma}
\begin{proof}
We can rewrite 
\begin{align*}
\FFge \big( \varDelta^{\GGe}_i f_1 \cdot \varDelta_j^{\GGe} f_2 \big) &= (\varphi_{i}^{\GGe} \FFge f)\ast_{\wGGe}(\varphi_j^{\GGe} \FFge f_2) \\
&= \int_{\wGGe} (\varphi_{i}^{\GGe} \FFge f)(z) \cdot  (\varphi_j^{\GGe} \FFge f_2)([\cdot -z]_{\wGGe})     \dd z \,, 
\end{align*}
where we used formal notation in the last step and $[\cdot]_{\wGGe}$ as in \eqref{eq:ConvolutionBandwidth}. From this one sees that the spectral support of $\varDelta^{\GGe}_i f_1 \cdot \varDelta_j^{\GGe} f_2$ is contained in
\begin{align}
\label{eq:SupportProductBall}
(\supp \varphi_{i}^{\GGe}\, +\,\supp \varphi_j^{\GG^\eps} + \CR^\eps) \cap \wGGe \,,
\end{align}
where we recall that $\supp \varphi^{\GGe}_i=\overline{\{x\in \wGGe \, \vert\, \varphi_i^{\GGe}(x)\neq 0 \}}$ is a subset of (the closure of) $\wGGe\subseteq \RR^d$, while the sum of sets in the parentheses should be read as a subset of $\RR^d$.
Now, by the dyadic scaling of $\varphi_j^{\GGe}$ we have for all $i \le j$
\begin{align*}
\supp \varphi_{i}^{\GGe}\, +\,\supp \varphi_j^{\GG^\eps} \subseteq B(0,2^{j}\,b)
\end{align*}
for some $b>0$, independent of $\eps$ and $j$. Set: $\B_1:=B(0,b)$ and consider first the case $2^j \mathcal{B}_1=B(0,2^j b)\subseteq \wGGe$. In this case we have 
\begin{align*}
(\supp \varphi_{i}^{\GGe}\, +\,\supp \varphi_j^{\GG^\eps} + \CR^\eps) \cap \wGGe \subseteq (2^j \B_1 + \CR^\eps) \cap \wGGe = 2^j\B_1 \cap \wGG^\eps = 2^j \B_1\,.
\end{align*}
On the other hand, if $2^j\B_1=B(0,2^j b)\subsetneq \wGGe$ we are in the regime $j\sim j_{\GGe}$ and take a ball $\B_2$ around $0$ such that $2^{j} \B_2 \supseteq \wGGe$ and hence $2^{j} \B_2\cap \wGGe =\wGGe$ for all $j\sim j_{\GGe}$ (by the dyadic scaling of $\GGe$ from Definition \ref{def:LatticeSequence} we have $2^{j_{\GGe}}=c\cdot  \eps^{-1}$ so that we can choose $\B_2$ independently of $\eps$). Choosing then $\B=\B_1\cup \B_2$ shows the first part of the claim.

Let us now consider $S_{j-1}^{\GGe} f_1 \cdot \varDelta_j^{\GGe} f_2$. With $\varphi_{<j-1}^{\GG^\eps}:= \sum_{j'<j-1} \varphi_{j'}^{\GGe}$ we see as above that the spectral support of $S_j^{\GGe} f_1 \cdot \varDelta_j^{\GGe} f_2$ is contained in
\begin{align}
\label{eq:SupportParaproduct}
(\supp \varphi_{<j-1}^{\GGe}\, +\,\supp \varphi_j^{\GG^\eps} + \CR^\eps) \cap \wGGe \,, 
\end{align}
We already know from above that this set is contained in a ball of size $2^j$ so that is enough to show that \eqref{eq:SupportParaproduct} is bounded away from $0$. Since $\supp \varphi_{<j-1}^{\GGe}$ and $\supp \varphi_j^{\GGe}$ are symmetric and disjoint, we have due to the scaling from \eqref{eq:ScalingVarphij} and \eqref{eq:ScalingSumVarphij}, which we observed in the proof of Lemma \ref{lem:ScalingBoundary}, that 
\begin{align*}
 \dist (\supp \varphi_{<j-1}^{\GGe} + \supp \varphi_j^{\GGe},0) \geq 2^j a
 \end{align*} 
 for some $a>0$ and 
 \begin{align}
 \label{eq:SupportParaproductb}
\supp \varphi_{<j-1}^{\GGe} + \supp \varphi_j^{\GGe} \subseteq B(0,2^j \cdot b')\,,
 \end{align}
 for some $b'>0$.
Note, that we can choose $b'>0$ small enough such that $B(0,2^{j_{\GGe}} b')\cap \CR^\eps=\{ 0 \}$. Indeed, otherwise there are $x_1\in \supp \varphi_{<j_{\GGe}-1}^{\GGe}$, $x_2\in \supp \varphi_{j_{\GGe}}^{\GGe}$ such that $x_1+x_2=r$ for some $r\in \CR^\eps\backslash \{0\}$. But from $|x_1|<\dist(\partial \wGGe,0)$ one sees that $|x_2|=|r-x_1|> \diam(\wGGe)/2$ which contradicts $x_2\in \supp \varphi_j^{\GGe}\subseteq \wGGe$. This choice of the parameter $b'$ can be done independently of $\eps$ due to the dyadic scaling of our lattice (Definition \ref{def:LatticeSequence}).

Consequently, there exists $r>0$ such that $\dist(B(0,2^j b')+\CR^\eps \backslash\{0\},0)= 2^j r$ (to see that $r>0$ is independent of $\eps$, use once more the dyadic scaling of the sequence $\GG^\eps$). But then we have 
\begin{align*}
\dist\big((\supp \varphi_{<j-1}^{\GGe}\, +\,\supp \varphi_j^{\GG^\eps} + \CR^\eps) \cap \wGGe,\,0\big)\geq (a \wedge r) \cdot 2^j\,,
\end{align*}
which closes the proof. 

\end{proof}

\subsection*{Results related to Section \ref{sec:PAM}}
\begin{proof}[Proof of Lemma \ref{lem:ChaosExpansionConvergence}]
We will write shorthand $\widehat{f^\varepsilon_k}:=\Ff_{(\GG^\varepsilon)^k}  f^\varepsilon_k$ and $\widehat{f_k}:=\Ff_{(\RR^d)^k} f_k$. The claimed convergence is a consequence of the results in~\cite{Caravenna}. For $z\in\GG^\varepsilon$ let $G^\varepsilon(z)=z+[-\varepsilon/2,\varepsilon/2)a_1+\ldots+[-\varepsilon/2,\varepsilon/2)a_d$, where $a_1,\ldots,a_d$ denote the vectors that span $\GG$. For $x\in\RR^d$ let $[x]_\eps$ be the (unique) element in $ \GG^\varepsilon$ such that $x\in G^\varepsilon([x]_\eps)$ and for $x \in (\RR^d)^k$ set $[x]_\eps = ([x_1]_\eps, \dots, [x_k]_\eps)$.  We will start by showing 
   \begin{align}
   \label{eq:ProofChaosConvergenceCentralStatement}
   \lim_{\varepsilon \to 0} \| f^\varepsilon_k([ \cdot]_\eps) - f_k\|_{L^2((\RR^d)^k)} = 0
   \end{align} for all $k$.

By Parseval's identity we have $\| f^\varepsilon_k([ \cdot]_\eps) - f_k\|_{L^2((\RR^d)^k)} = \| \mathcal{F}_{(\RR^{d})^k}(f^\varepsilon_k([ \cdot]_\eps)) - \widehat{f_k}\|_{L^2((\RR^d)^k)}$, where $\mathcal{F}_{(\RR^{d})^k}$ denotes the Fourier transform on $(\RR^d)^k$ for which one easily checks that
\[
   \mathcal{F}_{(\RR^{d})^k}(f^\varepsilon_k([ \cdot]_\eps)) = (\widehat{f^\varepsilon_k})_{\mathrm{ext}} \cdot p^\varepsilon_k,
\]
where we recall that $(\widehat{f^\varepsilon_k})_{\mathrm{ext}}$ is the periodic extension of the discrete Fourier transform of $f^\varepsilon_k$ (on $(\RR^d)^k$) as in \eqref{eq:PeriodicExtension2} and where
\begin{align*}
	p^\varepsilon_k(y_1, \dots, y_k) = \int_{G^1(0)^k} \frac{\dd z_1 \dots \dd z_k}{|\GG^1|^k} \,e^{-2\pi \imath \varepsilon (y_1 \scl z_1 + \dots + y_k \scl z_k)}.
\end{align*}
The function $p^\varepsilon_k$ is uniformly bounded and tends to $1$ as $\varepsilon$ goes to $0$. Now we apply Parseval's identity, once on $(\RR^{d})^k$ and once on $(\widehat{\GG^\varepsilon})^k$, and obtain
\begin{align*}
	\int_{(\RR^d)^k} \dd x_1 \dots \dd x_k \,\left| \big((\widehat{f^\varepsilon_k})_{\mathrm{ext}} \,p^\varepsilon\big)(x_1, \dots, x_k) \right|^2 & = \sum_{z_1, \dots, z_k \in \GG^\varepsilon} |\GG^\varepsilon|^k |f^\varepsilon_k(z_1, \dots, z_k)|^2 \\
	& = \int_{\widehat{(\GG^\varepsilon})^k} \dd x_1 \dots \dd x_k \left| \widehat{f^\varepsilon_k}(x_1, \dots, x_k)\right|^2
\end{align*}
and thus
\begin{align*}
	\int_{((\widehat{\GG^\varepsilon})^k)^c}\dd x_1 \dots \dd x_k\,\left| \big( (\widehat{f^\varepsilon_k})_{\mathrm{ext}}\, p^\varepsilon \big)(x_1, \dots, x_k) \right|^2 = \int_{\widehat{(\GG^\varepsilon})^k} \dd x_1 \dots \dd x_k \big(| \widehat{f^\varepsilon_k}|^2 (1 - |p^\varepsilon|^2 \big)(x_1, \dots, x_k) \,.
\end{align*}
Since $\1_{(\widehat{\GG^\varepsilon})^k} \widehat{f^\varepsilon_k}$ is uniformly in $\varepsilon$ bounded by $g_k \in L^2((\RR^d)^k)$ and since $1 - |p^\varepsilon|^2$ converges pointwise to zero, it follows from the dominated convergence theorem that $\1_{((\widehat{\GG^\varepsilon})^k)^c} (\widehat{f^\varepsilon_k})_{\mathrm{ext}} \,p^\varepsilon_k$ converges to zero in $L^2((\RR^d)^k)$. Thus, we get
\begin{align*}
   \lim_{\varepsilon \to 0} \| (\widehat{f^\varepsilon_k})_{\mathrm{ext}}\, p^\varepsilon_k - \widehat{f_k} \|_{L^2((\RR^d)^k)} & = \lim_{\varepsilon \to 0} \| \1_{(\widehat{\GG^\varepsilon})^k} \widehat{f^\varepsilon_k} p^\varepsilon_k  - \widehat{f_k} \|_{L^2((\RR^d)^k)} \\
   & \le \lim_{\varepsilon \to 0} \| (\1_{(\widehat{\GG^\varepsilon})^k} \widehat{f^\varepsilon_k}  - \widehat{f_k}) p^\varepsilon_k  \|_{L^2((\RR^d)^k)} + \lim_{\varepsilon \to 0} \| \widehat{f_k} (1 - p^\varepsilon_k ) \|_{L^2((\RR^d)^k)} = 0,
\end{align*}
where for the first term we used that $p^\varepsilon_k$ is uniformly bounded in $\varepsilon$ and that by assumption $\1_{(\widehat{\GG^\varepsilon})^k} \widehat{f^\varepsilon_k}$ converges to $\widehat{f_k}$ in $L^2((\RR^d)^k)$ and for the second term we combined the fact that $p^\varepsilon_k$ converges pointwise to $1$ with the dominated convergence theorem. We have therefore shown \eqref{eq:ProofChaosConvergenceCentralStatement}. Note that this implies
\begin{align}
\label{eq:ProofChaosConvergenceDiagonals}
\|f^\varepsilon_k([\cdot]_\eps) \mathbf{1}_{\forall i\neq j\, [z_i]_\eps\neq [z_j]_\eps}-f_k\|_{L^2(\RR^d)}\rightarrow 0 \qquad\& \qquad\|f^\varepsilon_k([\cdot]_\eps)  \mathbf{1}_{\exists i\neq j \, [z_i]_\eps=[z_j]_\eps} \|_{L^2(\RR^d)} \rightarrow 0\,.
\end{align}

As in the proof of Lemma \ref{lem:DiscreteItoIsometry} we identify $\GG^\varepsilon$ with an enumeration $\NN\rightarrow \GG^\varepsilon$ and use the set $A^k_r=\{a\in \mathbb{N}^r \,\vert\, \sum_i a_i=k\}$ so that we can write
\begin{align*}
\mathscr{I}_k f^\varepsilon_k=\sum_{1\leq r\leq k,\,a\in A^k_r} r! \sum_{z_1<\ldots<z_r} |\GG^\varepsilon|^k \tilde{f}^k_{\varepsilon,a}(z_1,\ldots,z_r)\cdot \prod_{j=1}^r \xi^\eps(z_j)^{\diamond {a_j}}\,,
\end{align*}
where we denote as in the proof of Lemma \ref{lem:DiscreteItoIsometry} by $\tilde{f}^k_{\varepsilon,a}$ the symmetrized restriction of $f^k_\varepsilon$ to $(\RR^d)^r$. By Theorem~2.3 of \cite{Caravenna} we see that due to \eqref{eq:ProofChaosConvergenceDiagonals} the $r=k$ term of $\mathscr{I}_k f^\varepsilon_k$ 
converges in distribution to the desired limit, so that we only have to show that the remaining terms vanish as $\varepsilon$ tends to 0. The idea is to redefine for fixed $a \in A^k_r$ the noise as $\overline{\xi}^\varepsilon_j(z)=\xi^\eps(z)^{\diamond {a_j}}/r_j^\varepsilon(z)$ where $r^\varepsilon_j(z):=\sqrt{\mathrm{Var}(\xi^\eps(z)^{\diamond {a_j}})\cdot |\GG^\varepsilon|}\lesssim |\GG^\varepsilon|^{(1-a_j)/2}$, so that in view of \cite[Lemma 2.3]{Caravenna} it suffices to show that 
\begin{align*}
\sum_{z_1<\ldots<z_r} |\GG^\varepsilon|^r \prod_{j=1}^r r^\varepsilon_j(z_j)^2 \cdot |\tilde{f}^\varepsilon_{k,a}(z_1,\ldots,z_r)|^2\lesssim 
\sum_{z_1<\ldots<z_r} |\GG^\varepsilon|^k \cdot |\tilde{f}^\varepsilon_{k,a}(z_1,\ldots,z_r)|^2 \rightarrow	 0\,,
\end{align*}
but this follows from \eqref{eq:ProofChaosConvergenceDiagonals}.
\end{proof}

\begin{lemma}
\label{lem:Fatou}
Let $(f_n)_{n\geq 0}$ be a sequence which is bounded in the space $\mathscr{L}^{\gamma,\alpha}_{p,T}(\GG,e^\sigma_l)$ and let $\alpha'\in(0,\alpha)$ and $\sigma'\in (0,\sigma)$. There is a subsequence $(f_{n_k})_{k\geq 0}$, convergent in $\mathscr{L}^{\gamma,\alpha'}_{p,T}(\GG,e^{\sigma'}_l)$, with limit $f$ such that
\begin{align}
\label{eq:Fatou}
\|f\|_{\mathscr{L}^{\gamma,\alpha}_{p,T}(\GG,e^{\sigma}_l)}\leq \liminf_{k\rightarrow \infty} \|f_{n_k}\|_{\mathscr{L}^{\gamma,\alpha}_{p,T}(\GG,e^{\sigma}_l)}
\end{align}
\end{lemma}
\begin{proof}
Take in the following $\tilde{\alpha}=\frac{\alpha+\alpha'}{2}$ and $\tilde{\sigma}=\frac{\sigma+\sigma'}{2}$. By Definition of $\mathscr{L}^{\gamma,\alpha}_{p,T}(\GG,e^\sigma_l)$ we know that $(g_n)_{n\geq 0}:=\big((t,x)\mapsto t^{\gamma} f_n(t,x) \big)_{n\geq 0}$ is bounded in $C_T^{\alpha/2}L^p(\GG,e^\sigma_l)\cap C_T\cC_p^{\alpha}(\GG,e^\sigma_l)$. Interpolation then shows that $(g_n)_{n\geq 0}$ is bounded in $C_T^{\tilde{\alpha}/2}\cC_p^{\delta_x}(\GG,e^\sigma_l)\cap C_T^{\delta_t}\cC_p^{\tilde{\alpha}}(\GG,e^\sigma_l)$ for some $\delta_x,\,\delta_t>0$. We obtain by compact embedding (Lemma \ref{lem:DiscreteEmbedding}) for $\delta_x'\in (0,\delta_x),\,\delta_t'\in (0,\delta_t)$ the existence of a convergent subsequence $(g_{n_k})_{k\geq 0}$ in $C_T^{\alpha'}\cC_p^{\delta_x'}(\GG,e^{\sigma'}_l)\cap C_T^{\delta_t'}\cC_p^{\alpha'}(\GG,e^{\sigma'}_l)$ with some limit $g$. From the convergence of $g_{n_k}\rightarrow g$ in $C_T^{\alpha'}\cC_p^{\delta_x'}(\GG,e^{\sigma'}_l)\cap C_T^{\delta_t'}\cC_p^{\alpha'}(\GG,e^{\sigma'}_l)$ it follows that for $f:=t^{-\gamma} g$ we have $f_{n_k}\rightarrow	f$ in $\mathscr{L}^{\gamma,\alpha'}_{p,T}(\GG,e^{\sigma'}_l)$. 

The estimate \eqref{eq:Fatou} is then just an iterative application of Fatou like arguments for the norms from which $\|\cdot\|_{\mathscr{L}^{\gamma,\alpha}_{p,T}(\GG,\rho)}$ is constructed. 
\end{proof}

{
\setlength{\glsdescwidth}{7in}
\small
\printglossaries
}


\end{document}